\documentclass{amsart}

\usepackage[dvipsnames]{xcolor}
\usepackage{amsmath, amsfonts, amssymb, amsthm, dsfont, mathrsfs, wasysym, graphics, graphicx, url, hyperref, hypcap,
xargs, multicol, pdflscape, multirow, hvfloat, array, ae, aecompl, pifont, mathtools, a4wide, float, blkarray, overpic, nicefrac, stmaryrd, anyfontsize, yfonts, tabularx, fontawesome}

\usepackage[T1]{fontenc}
\usepackage{bbm}
\usepackage[shortlabels, inline]{enumitem}
\usepackage[noabbrev,capitalise]{cleveref}
\usepackage[normalem]{ulem}
\usepackage{marginnote}
\usepackage{txfonts}
\hypersetup{colorlinks=true, citecolor=darkblue, linkcolor=darkblue}
\usepackage[all]{xy}
\usepackage{tikz}
\usepackage{tikz-cd}
\usepackage{tkz-graph}
\usetikzlibrary{trees, decorations, decorations.pathmorphing, decorations.markings, decorations.shapes, shapes, arrows, matrix, calc, fit, intersections, patterns, angles}
\graphicspath{{figures/}{figures/diagonals/}{figures/walks/}{figures/tubes/}{figures/blocks/}}
\makeatletter\def\input@path{{figures/}}\makeatother
\usepackage{caption,subcaption}
\usepackage[export]{adjustbox}

\usepackage{paralist}
\usepackage{shuffle}

\newtheorem{thmUniv}{Theorem}

\newtheorem{theorem}{Theorem}[subsection]
\newtheorem{corollary}[theorem]{Corollary}
\newtheorem{proposition}[theorem]{Proposition}
\newtheorem{lemma}[theorem]{Lemma}
\newtheorem{conjecture}[theorem]{Conjecture}
\newtheorem*{theorem*}{Theorem}

\theoremstyle{definition}
\newtheorem{definition}[theorem]{Definition}
\newtheorem{example}[theorem]{Example}
\newtheorem{remark}[theorem]{Remark}

\newtheorem{question}[theorem]{Question}

\AddToHook{env/theorem/begin}{\crefalias{section}{theorem}}
\AddToHook{env/proposition/begin}{\crefalias{theorem}{proposition}}
\AddToHook{env/corollary/begin}{\crefalias{theorem}{corollary}}
\AddToHook{env/lemma/begin}{\crefalias{theorem}{lemma}}
\AddToHook{env/conjecture/begin}{\crefalias{theorem}{conjecture}}
\AddToHook{env/definition/begin}{\crefalias{theorem}{definition}}
\AddToHook{env/example/begin}{\crefalias{theorem}{example}}
\AddToHook{env/remark/begin}{\crefalias{theorem}{remark}}
\AddToHook{env/observation/begin}{\crefalias{theorem}{observation}}
\AddToHook{env/question/begin}{\crefalias{theorem}{question}}
\AddToHook{env/notation/begin}{\crefalias{theorem}{notation}}
\AddToHook{env/assumption/begin}{\crefalias{theorem}{assumption}}
\AddToHook{env/convention/begin}{\crefalias{theorem}{convention}}

\newcommand{\R}{\mathbb{R}} 

\renewcommand{\c}[1]{{\mathcal{#1}}} 
\renewcommand{\b}[1]{{\boldsymbol{#1}}} 
\newcommand{\scr}[1]{{\mathscr{#1}}} 

\renewcommand{\emptyset}{\varnothing} 
\renewcommand{\epsilon}{\varepsilon} 

\newcommand{\set}[2]{\left\{ #1 ~;~ #2 \right\}} 
\newcommand{\ssm}{\smallsetminus} 

\newcommand{\one}{{1\!\!1}} 
\newcommandx{\ones}[1][1=n]{\one_{#1}} 
\newcommand{\eqdef}{\coloneqq}
\newcommand{\defeq}{\mbox{~\ensuremath{=}\raisebox{0.2ex}{\scriptsize\ensuremath{\mathrm:}} }} 

\DeclareMathOperator{\conv}{conv} 
\DeclareMathOperator{\cone}{cone} 
\DeclareMathOperator{\interior}{int} 

\DeclareMathOperator{\lin}{lin} 
\DeclareMathOperator{\aff}{aff} 

\newcommandx{\di}[2][1=\fw,2=F]{\mathrm{dir}_{#1}({#2})}

\newcommandx{\cv}[1][1=X]{\mathbbm{1}_{#1}}
\newcommandx{\cve}[1][1=X]{(\mathbbm{1}_{#1})}

\newcommandx{\restr}[2][1=\fw,2=X]{{#1}_{#2}}
\newcommandx{\contr}[2][1=\fw,2=X]{{#1}_{/#2}}

\newcommand{\ie}{\textit{i.e.}~} 
\newcommand{\eg}{\textit{e.g.}~} 
\newcommand{\cf}{\textit{cf.}~} 
\newcommand{\aka}{\textit{a.k.a.}~} 
\definecolor{darkblue}{rgb}{0,0,0.7} 
\definecolor{green}{RGB}{57,181,74} 
\definecolor{violet}{RGB}{147,39,143} 
\newcommand{\darkblue}{\color{darkblue}} 
\newcommand{\defn}[1]{\textsl{\darkblue #1}} 
\newcommand{\mathdefn}[1]{{\darkblue #1}} 

\newcommand{\yes}{\textcolor{green}{\ding{52}}}
\newcommand{\no}{\textcolor{red}{\ding{55}}}

\makeatletter
\def\part{\@startsection{part}{1}%
	\z@{.7\linespacing\@plus\linespacing}{.8\linespacing}%
	{\LARGE\sffamily\centering}}
\makeatother

\newcommand{\pol}[1][P]{\mathsf{#1}}
\newcommand{\polF}{\pol[F]}

\newcommand{\polZ}{\pol[Z]}
\newcommand{\polQ}{\pol[Q]}

\newcommand{\p}[1][\phi]{\b {#1}}
\newcommand{\lam}{{\b\lambda}}
\newcommandx{\bl}[1][1=\lam]{\b {#1}}
\newcommandx{\dv}[1][1=\lam]{\bl[#1]}
\newcommand{\fw}[1][F]{\mathcal{#1}}
\newcommand{\Efw}[1][F]{\overline{\mathcal{#1}}}
\newcommand{\DefoCone}[1][\pol]{\mathbb{DC}(#1)}
\newcommand{\EDefoCone}[1][\fw]{\mathbb{DC}(#1)}
\newcommand{\permuto}[1][n]{\pol[\Pi]_{#1}}
\newcommand{\verts}{V}
\newcommand{\facets}{F}
\newcommand{\edges}{E}
\newcommand{\nde}{{E_{nd}}}
\newcommand{\de}{{E_{deg}}}
\newcommand{\Iedges}{\overline E}
\newcommand{\ndie}{{\overline E_{nd}}}
\newcommand{\die}{{\overline E_{deg}}}
\newcommand{\G}{\mathcal{G}}
\newcommand{\der}{{\sim_{deg}}}
\newcommand{\dec}[1]{{[{#1}]_{deg}}}
\newcommandx{\subr}[1][1=W]{{\sim_{#1}}}
\newcommandx{\subc}[2][2=W]{{[{#1}]_{#2}}}

\newcommand{\pcl}[1][S]{\mathrm{cl}_{\mathrm{pin}}({#1})}

\newcommand{\inc}{\partial}

\newcommand{\mat}[1][M]{#1}
\newcommand{\matpol}[1][M]{\pol_{{#1}}}
\newcommand{\ground}{S}
\newcommand{\bases}[1][B]{\mathcal{#1}}

\newcommand{\nodes}{N}
\newcommand{\arcs}{A}

\DeclareMathOperator{\id}{id}

\newcommand{\ELD}[1][\fw]{ED(#1)}
\newcommand{\EELD}[1][\fw]{{ED}(\overline{#1})}
\newcommand{\zELD}[1][ij]{ED_{#1}}

\newcommandx{\ZG}[1][1=G]{\polZ_{#1}} 
\newcommandx{\ZGKnm}[2][1=n,2=m]{\polZ_{#1, #2}}

\newcommand{\dashcircle}{
\begin{tikzpicture}
    \draw (0, 0) -- (0.2, 0);
    \draw[black, fill=white] (0.2, 0) circle (1.25pt);
\end{tikzpicture}
}
\newcommand{\circledashcircle}{
\begin{tikzpicture}
    \draw (0, 0) -- (0.2, 0);
    \draw[black, fill=white] (0, 0) circle (1.25pt);
    \draw[black, fill=white] (0.2, 0) circle (1.25pt);
\end{tikzpicture}
}

\newcommandx{\Pnm}[2][1=n,2=m]{\polZ^{\dashcircle}_{#1, #2}}
\newcommandx{\Qnm}[2][1=n,2=m]{\polZ^{\circledashcircle}_{#1, #2}}

\newcommand{\simplex}[1][n]{\mathsf{\Delta}_{#1}}

\newcommandx{\restrG}[2][1=\mu,2=G]{{#2}\left|_{#1}\right.}
\newcommandx{\contrG}[2][1=\mu,2=G]{{#2}\left/_{#1}\right.}
\newcommandx{\Knm}[2][1=n,2=m]{K_{#1, #2}}

\newcommand{\AO}{\c A} 

\usepackage{todonotes}

\title{The graph of implicit edge dependencies for indecomposability and beyond}

\author{Arnau Padrol}
\address[Arnau Padrol]{Universitat de Barcelona \& Centre de Recerca Matemàtica}
\email{arnau.padrol@ub.edu}

\author{Germain Poullot}
\address[Germain Poullot]{Universität Osnabrück}
\email{germain.poullot@uni-osnabrueck.de} 

\thanks{The research of A.P. is partially supported by the Spanish project PID2022-137283NB-C21 of MICIU/AEI/10.13039/501100011033 / FEDER, UE, the Spanish--German project COMPOTE (AEI PCI2024-155081-2 \& DFG 541393733), the Severo Ochoa and María de Maeztu Program for Centers and Units of Excellence in R\&D (CEX2020-001084-M), the Departament de Recerca i Universitats de la Generalitat de Catalunya (2021 SGR 00697), and the French--Austrian project PAGCAP (ANR-21-CE48-0020 \& FWF I 5788). }

\begin{document}

\begin{abstract}
	A polytope is called indecomposable if it cannot be expressed nontrivially as a Minkowski sum of other polytopes.
    Since Gale introduced the concept in 1954, several increasingly strong criteria have been developed to characterize indecomposability.
    In this paper, we introduce a new approach to indecomposability for frameworks and polytopes based on the graph of implicit edge dependencies, which records proportionalities between edge lengths across all deformations.
    This yields a new indecomposability criterion that unifies and generalizes most previous approaches, and has additional consequences in the study of deformation cones.
		
	As a main application, we construct new indecomposable deformed permutahedra that are not matroid polytopes.
    In 1970, Edmonds already noted the difficulty of characterizing the extreme rays of the submodular cone, equivalently, indecomposable deformed permutahedra.
    Matroid polytopes of connected matroids form a well-known family of such examples.
    We exhibit a new infinite family of indecomposable deformations of the permutahedron, not arising from matroid polytopes, obtained by suitable truncations of certain graphical zonotopes.
    
    We further demonstrate the scope of our methods through several additional applications. In particular, we refute a conjecture of Smilansky (1987) on the relation between the numbers of vertices and facets of indecomposable polytopes. Moreover, we obtain new bounds on the dimensions of deformation cones and we construct and analyze uniquely decomposable polytopes.
\end{abstract}

\maketitle

\tableofcontents


\section{Introduction}

The \defn{Minkowski sum} of two polytopes $\pol[Q],\pol[R]\subset\R^d$ is the polytope 
$\pol=\pol[Q]+\pol[R] := \set{\b q +\b r}{\b q\in\pol[Q], \b r\in\pol[R]}$.
In this case, $\pol[Q]$ and $\pol[R]$ are called \defn{Minkowski summands} of $\pol$. 
They are called \defn{weak Minkowski summands} or \defn{deformations} of $\pol$ if they are Minkowski summands of $\b\lambda \pol$ for some $\b\lambda > 0$. If all deformations of $\pol$ are translated dilates of $\pol$, \ie if~$\pol$ cannot be written as a nontrivial Minkowski sum, we say that $\pol$ is \defn{indecomposable}.
The set of deformations of a polytope $\pol$ is closed under Minkowski addition and dilation, and therefore forms a convex cone, called the \defn{deformation cone} of $\pol$. 
Its rays correspond precisely to the indecomposable deformations of $\pol$.

Minkowski addition is a fundamental operation on convex polytopes, and the study of polytope
deformations and decomposability has become a central topic in polyhedral geometry, with applications across a wide range of areas. Introduced by Hermann Minkowski~\cite{Minkowski1911}, it led to the \emph{Brunn-Minkowski theory}, which studies the interplay of (mixed)-volumes and Minkowski addition~\cite{SchneiderConvexBodies}. These ideas are central in McMullen's proof of the $g$-theorem via the \emph{polytope algebra}~\cite{McMullen1,McMullen1989}, where multiplication corresponds to Minkowski addition, and have more recently played a role in combinatorial Hodge theory~\cite{AHK2018}.
The deformation cone of a rational polytope coincides with its \emph{nef cone}, encoding projective embeddings of the associated toric variety~\cite[Section~6]{CoxLittleSchenck}. Further applications include sparse elimination and polynomial factorization~\cite{GaoLauder2001,SturmfelsSolving,GelfandKapranovZelevinsky}, tropical geometry \cite{MaclaganSturmfelsTropical}, toric deformations~\cite{Altmann1995}, integer programming and combinatorial optimization~\cite{Edmonds,Schrijver1986}, game theory~\cite{Shapley:1971}, and, more recently, neural networks~\cite{MontufarRenZhang2022,HertrichPhD}.

Indecomposable polytopes are the building blocks of Minkowski addition: they correspond to the rays of the deformation cones, and every polytope is a finite sum of indecomposable ones. Since Gale gave the first criteria to certify that a polytope is indecomposable in 1954~\cite{Gale1954}, 
many successively stronger criteria have been found \cite{Shephard1963,Kallay1982,McMullen1987,PrzeslawskiYost2008,PrzeslawskiYost2016} (see \cite[Chapter~6]{PinedaVillavicencio} for a recent survey). However, none of these apply to the truncated graphical zonotopes that we will study.


Deciding indecomposability amounts to computing the dimension of the deformation cone.
Shephard described deformations of a polytope in terms of its edges~\cite{Shephard1963}, and this characterization allowed Kallay to extend the notions of deformation and indecomposability to arbitrary \defn{frameworks} (also known as \defn{geometric graphs}) in \cite{Kallay1982}. This extension not only broadens the scope, but it also simplifies many intermediate computations even for the case of polytopes, providing the right level of generality for our purposes. For example, in \cref{cor:polytopeto2dimframework} we show that the deformation cone of any polytope coincides with the deformation cone of a $2$-dimensional framework.

Our main new tool is the \defn{graph of (implicit) edge dependencies} of a framework~$\fw$.
Its nodes are the edges of~$\fw$, and two edges are adjacent if the ratio of their lengths remains constant across all deformations of $\fw$. In this case they are called \defn{dependent}. A framework is indecomposable exactly when this graph is complete. 
Using this graph, we provide a new indecomposability criterion for frameworks that encompasses most previously known criteria.
Besides indecomposability, these graphs enable us to study other properties of deformation cones, such as its simpliciality. We showcase this in \cref{sec:newrays,sec:edgedecomposable,sec:constructions}. 



Most previous indecomposability criteria relied on finding suitable ``\emph{strong chains}'' of ``\emph{indecomposable subgraphs}'' in a framework or in the $1$-skeleton of a polytope. 
Our approach isolates the core properties of these structures that are needed to certify indecomposability, and in doing so gains the flexibility to deduce relations between edges even through non-indecomposable subgraphs, in particular via parallelisms and projections.
A key new concept is that of \defn{implicit edges}, which 
capture all pairs of vertices that behave as edges under deformation, even if not present combinatorially. They allow us to treat projections and degenerate cases uniformly, and to deduce hidden dependencies.
Moreover, they unveil the power of \defn{pinning closures} of vertex sets, which allow to iteratively discover new dependent implicit edges.
%

As we develop our tools, we present indecomposability criteria of increasing strength and complexity, culminating in the following main indecomposability theorem:

\begin{thmUniv}[\Cref{thm:mainthm}]\label{thm:mainthmintro}
	Let $\fw=(\verts,\edges,\p)$ be a framework. If there is a dependent subset of vertices $S\subseteq \verts$ whose pinning closure covers all vertices, $\pcl=\verts$, then $\fw$ is indecomposable.
\end{thmUniv}
The terms and concepts appearing in this statement will be properly defined and explained in the upcoming sections, where we also show how this criterion subsumes several classical ones.
For the sake of this introduction, we give some examples to illustrate the keywords:
\begin{compactitem}
	\item A \defn{framework} $\fw$ can be given by the vertices and edges of a polytope.
	\item Edges of an indecomposable framework are pairwise \defn{dependent}. For example, the edges of a triangle are pairwise dependent. Another paradigmatic example of dependent edges are parallel edges in a parallelogram.
	\item If there is a collection of pairwise dependent edges $X$ connecting every pair of vertices of $S\subseteq \verts$, then $S$ is \defn{dependent}.
	\item The pinning closure $\pcl$ is defined iteratively, and requires some notation explained in \Cref{ssec:CoveringFamilyOfFlats}. But for example, if every facet of a polytope contains a vertex in $S$, then $\pcl=\verts$.	
\end{compactitem}

\vspace{0.15cm}

While edge dependencies may appear very specialized, because they capture only a restricted class of linear relations among edge lengths, we know that these are present in indecomposable frameworks, and the challenge resides in unveiling and manipulating them. In practice, our methods suffice to recover most known indecomposability proofs. 

In particular, \Cref{thm:mainthmintro} encapsulates several previously established indecomposability criteria in a single statement (\eg \Cref{thm:Shephard,cor:McMullen1,cor:McMullen}), and demonstrates its versatility through simple new proofs of indecomposability for classical and recent examples, such as Kallay's conditionally decomposable polytope~\cite{Kallay1982} in~\Cref{exm:KallayConditionallyDecomposable}, connected matroid polytopes~\cite{Nguyen1986-SemimodularFunctions} in~\Cref{thm:ConnectedMatroidPolytopesAreIndecomposable}, or extremal weight dominant polytopes~\cite{BurrullGuiHu2024StronglyDominantWeightPolytopesAreCubes} in~\cref{ex:hyperorder}.
Although \Cref{ex:iterativeclosure} shows that the full strength of \Cref{thm:mainthmintro} can be needed, the simpler criteria developed along the way already suffice for many of these applications.

Beyond indecomposability, we extend \Cref{thm:mainthmintro} to a general upper bound on the dimension of deformation cones in \Cref{thm:DCdimBoundedByNumberOfDependentSetsCovering}.

In subsequent papers, the tools developed here have been applied in the proof that all $0/1$-polytopes that are not a Cartesian product are indecomposable in \cite{HigashitaniPadrolSanyal}, and to compute the dimension of the deformation cones of Platonic, Archimedean, and Johnson solids~\cite{Poullot-JohnsonSolids}.

We now turn to applications of these results to concrete families of polytopes.  \Cref{sec:newrays} is devoted to one key application: producing new rays of the submodular cone.
The {$n$-dimensional} \defn{permutahedron}~$\permuto$ is the convex hull of the $n!$ permutations of the vector $(1, 2, \dots, n)\in \R^n$. 
A \defn{deformed permutahedron} (\aka \defn{generalized permutahedron}) is a deformation of~$\permuto$. Originally introduced by Edmonds in 1970 under the name of \defn{polymatroids} in linear optimization~\cite{Edmonds}, they have been popularized by Postnikov~\cite{Postnikov2009} within the algebraic combinatorics community. Nowadays, they are one of the most studied families of polytopes, 
in areas such as algebraic combinatorics, statistics, optimization, and game theory. 
The deformation cone of the $\permuto$ is parametrized by the cone of \defn{submodular functions} (\ie set function $f : 2^{[n]} \to \R$ satisfying $f(A) + f(B) \geq f(A\cup B) + f(A\cap B)$ for all subsets $A, B \subseteq[n]$). Edmonds ended his seminal paper~\cite{Edmonds} by observing the difficulty of characterizing its extreme rays (\ie indecomposable deformed permutahedra). This has been studied from several disciplines \cite{Shapley:1971,RosenmullerWeidner,Nguyen1986-SemimodularFunctions,GirlichHodingSchneidereitZaporozhets:1995,Kashiwabara:2000,studeny2000extreme,StudenyKroupa2016,studeny2016basicfactsconcerningsupermodular,GrabischKroupa:2019,HaimanYao:2023,LohoPadrolPoullot2025RaysSubmodularCone}, but the question is still wide open. 
Although counting all the rays has been deemed ``impossible'' \cite{CsirmazCsirmaz-AttemptingTheImpossible}, there remains a strong need for explicit and structurally distinct examples. An infinite super-exponential family of rays is given by matroid polytopes of connected matroids~\cite[Section 10]{Nguyen1986-SemimodularFunctions} (see also~\cite[Section~7.2]{StudenyKroupa2016} and \Cref{sec:matroidpolytopes}), which until recently gave the best lower bound on the number of rays of the submodular cone. 

The starting point for the construction of the present article are graphical zonotopes: Minkowski sums of some edges of the standard simplex indexed by the arcs of an associated graph~$G$. They are deformed permutahedra, but decomposable. 
We show that, when~$G$ is a complete bipartite graph~$\Knm$, we can (deeply) truncate one or two (particular) vertices of its graphical zonotope to obtain polytopes that are simultaneously: deformed permutahedra (\cref{thm:PandQareGP}), indecomposable (\cref{thm:PandQareIndecomposable}), and not matroid polytopes (\cref{thm:PandQnotMatroidPolytopes}).
This way, we obtain $2\lfloor\frac{n-1}{2}\rfloor$ new indecomposable deformations of the $n$-permutahedron $\permuto\subseteq \R^n$ (\cref{cor:NewMIGP}). Moreover, we use some of these examples to refute a conjecture of Smilansky from 1987 {\cite[Conjecture~6.12]{Smilansky1987}}.
%
%

Taking \defn{permutahedral wedges} over these truncated graphical zonotopes, in \cref{sec:wedges} we enlarge this to a bigger family of indecomposable deformed permutahedra that are not matroid polytopes.

\begin{thmUniv}[\Cref{cor:BetterLowerBound}]
	For $N\geq 4$, there are at least $\frac{1}{3!}(N-1)!$ non-normally-equivalent indecomposable deformed $N$-permutahedra which are not normally-equivalent to matroid polytopes.
\end{thmUniv}

In a follow-up separate article~\cite{LohoPadrolPoullot2025RaysSubmodularCone}, together with Georg Loho we devised alternative methods to construct an even larger infinite family, gathering more than $2^{2^{n-2}}$ rays.
While the first indecomposability proofs from \cite{LohoPadrolPoullot2025RaysSubmodularCone} relied on the methods developed here, the final version instead exploits structural properties of the submodular cone.
As we discuss in \cref{rmk:NotGeneratedByLPPmethod}, the family constructed in \cite{LohoPadrolPoullot2025RaysSubmodularCone} does not contain the polytopes that we construct here.


We then display more applications of our setup. In \Cref{sec:edgedecomposable} we study \defn{uniquely decomposable frameworks}, those that have a unique nontrivial description as a Minkowski sum of indecomposables (up to translation), which we characterize by the fact that they have a simplicial deformation cone. 
This is a very remarkable property that has been observed for some important families of polytopes. For example, all the polytopes \emph{combinatorially equivalent} to a product of simplices have this property~\cite{CDGRY2020}. Another remarkable example is Loday's associahedron, for which this was implicitly observed in the realizations of kinematic associahedra~\cite{ArkaniHamedBaiHeYan} 
used to study the behavior of scattering particles in mathematical physics. This was further developed in~\cite{PadrolPaluPilaudPlamondon}, where it was shown that unique decomposability leads to a simple description of all polytopal realizations in the kinematic space~\cite{ArkaniHamedBaiHeYan}. This inspired the search for other uniquely decomposable families, such as 
associahedra for finite type cluster algebras and brick 2-acyclic gentle bound quivers~\cite{PadrolPaluPilaudPlamondon}, interval nestohedra~\cite{PPP2023Nesto}, triangle-free graphical zonotopes~\cite{PPP2023gZono}, and certain permutreehedra~\cite{AlbertinPilaudRitter2021}.

We use the graph of edge dependencies to define and study \defn{edge decomposable} frameworks, an important structural property that implies unique decomposability. For example, we show that all polytopes that have only triangle and parallelogram $2$-faces have this property in \Cref{cor:triangleparallelogram}. 
The deformation cone of a product is the product of the deformation cones of the factors, see \Cref{thm:products}, and in \Cref{ssec:parallelogramicMinkowskisums} we show that this property extends to certain Minkowski sums that are ``generic enough'', which we call \defn{parallelogramic} (\eg parallelogramic zonotopes). Parallelogramic Minkowski sums of indecomposable polytopes provide a large family of edge decomposable polytopes.

Finally, in \Cref{sec:constructions}, we extend our previous results to provide further constructions of indecomposable and uniquely decomposable polytopes. First, in \Cref{ssec:DeepTruncationsOfParallelogramicZonotopes}, we generalize our construction of truncated graphical zonotopes to other parallelogramic zonotopes; and in  \Cref{ssec:StackingOnParallelogramicZonotopes}, we study how the dimension of the deformation cone changes when stacking on certain facets of a parallelogramic Minkowski sum.

While finalizing this manuscript, we learned that Spencer Backman and Federico Castillo independently proved that certain truncated zonotopes are indecomposable, also exploiting their parallelogram faces \cite{BackmanCastillo}.

\section{New indecomposability criteria}\label{sec:criteria}

\subsection{Preliminaries}\label{ssec:Prelim}

\subsubsection{Frameworks}

A \defn{framework} in $\R^d$ is a triple $\fw=(\verts,\edges,\p)$, where $\verts$ is a finite set (the \defn{vertices}), $\edges$ is a subset of $\binom{\verts}{2}$ (the \defn{edges}), and $\p$ is a function $\p:\verts\to \R^d$ (the \defn{realization} of $\fw$).
For $v\in \verts$ we will use $\p_v$ to denote the image of $v$ under~$\p$, and think of $\p$ as a configuration $\p(\verts)~=~\{\p_v\}_{v\in \verts}$ of points in~$\R^d$ labeled by the set~$\verts$. The \defn{dimension} of $\fw$ is the dimension of the affine span of~$\p(\verts)$.
We abbreviate $uv$ for the edge $\{u,v\}$ for vertices $u,v\in \verts$. We will say that $uv\in \edges$ is a \defn{degenerate} edge if $\p_u=\p_v$, and \defn{non-degenerate} otherwise. 
The set of non-degenerate edges of $\fw$ is called its \defn{support} and denoted $\nde(\fw)$. Its complement, the set of degenerate edges, is denoted $\de(\fw)$. These are simplified to just $\nde$ and $\de$ if the framework is clear from the context.

A \defn{graph} is a pair $G=(\nodes,\arcs)$ where $\nodes$ is a finite set (the \defn{nodes}), and $\arcs$ is a subset of $\binom{\nodes}{2}\cup \binom{\nodes}{1}$ (the \defn{arcs}), which can include \defn{loops} (but we will not need multi-edges).
Frameworks are also known as \emph{geometric graphs}, in particular in the literature on polytope deformations \cite{Kallay1982,PrzeslawskiYost2008,PrzeslawskiYost2016}, but we preferred to adopt the notation from rigidity theory to minimize confusions with other graphs appearing in the paper. For the same reason, we reserve the names \emph{vertex} and \emph{edge} for frameworks (and later also for polytopes), and use \emph{node} and \emph{arc} for other abstract graphs.   

Let $\fw=(\verts,\edges,\p)$ and $\fw[G]=(\verts,\edges,\p[\psi])$ be frameworks with the same underlying graph $G=(\verts,\edges)$. We say that $\fw[G]$ is a \defn{deformation} of $\fw$ if for every $uv\in \edges$ there is some $\lam_{uv}\in \R_+\eqdef [0,+\infty)$ such that
\begin{equation}\label{eq:ELDF}\lam_{uv}(\p_{v}-\p_{u})=\p[\psi]_{v}-\p[\psi]_{u}\end{equation}
Since all the deformations of $\fw$ have the same underlying graph, we can define the operations of point-wise addition and scalar multiplication induced by those in the space of functions $\verts\to \R^d$;  explicitly $\fw + \c G \eqdef (V, E, v\mapsto \b\phi_v + \b \psi_v)$, and $\alpha\fw \eqdef (V, E, v\mapsto \alpha \b \phi_v)$. With these operations, the set of deformations of~$\fw$ forms a convex cone. We will use the coefficients~$\lam_{uv}$ from \eqref{eq:ELDF} to parametrize this cone modulo translations.

Given a fixed framework $\fw=(\verts,\edges,\p)$, we associate to each of its deformations $\fw[G]=(\verts,\edges,\p[\psi])$ an \defn{edge-deformation vector} $\mathdefn{\dv_{\fw}(\fw[G])}\in\R_+^{\edges}$, 
where for each edge~$uv\in \edges$ we set
\begin{equation}\label{eq:defovector}
\dv_{\fw}(\fw[G])_{uv}=\begin{cases}
\frac{\p[\psi]_{v}-\p[\psi]_{u}}{\p_{v}-\p_{u}} & \text{ if }uv\in \nde\\
0 & \text{ if }uv\in \de
\end{cases}
\end{equation}
which is the coefficient from \eqref{eq:ELDF}, see \Cref{sfig:EdgeLengthVector}. When $\fw$ is clear from the context, we write $\mathdefn{\dv(\fw[G])}$ for the edge-deformation vector of~$\fw[G]$.
We first observe that $\mathdefn{\dv(\fw[G])}$ determines $\fw[G]$ up to translation of its connected components.

\begin{lemma}\label{lem:connectedcomponents}
For deformations $\fw[G]=(\verts,\edges,\p[\psi])$ and $\fw[H]=(\verts,\edges,\p[\omega])$ of a framework $\fw=(\verts,\edges,\p)$, we have that ${\dv_{\fw}(\fw[G])}={\dv_{\fw}(\fw[H])}$ if and only if 
there are vectors $\b t_1\dots \b t_s\in\R^d$ such that $\p[\omega]_v=\p[\psi]_v+\b t_i$ for all $v\in \verts_i$, where $\verts_1, \dots, \verts_s$ are the vertex sets of the connected components of the graph $G = (\verts,\edges)$.
\end{lemma}

\begin{proof}
Assume that for every edge $e\in\edges$, we have that ${\dv_{\fw}(\fw[G])}_e={\dv_{\fw}(\fw[H])}_e=\lam_e$. Let $u, v\in \verts_i$ be vertices in the same connected component, and let $u=u_0,u_1,\dots,u_k=v$ be a path joining $u$ and $v$. Then 
\[\p[\psi]_v-\p[\psi]_u= \sum_{j=1}^k (\p[\psi]_{u_j}-\p[\psi]_{u_{j-1}})= \sum_{j=1}^k \lam_{u_{j-1}u_j}(\p_{u_j}-\p_{u_{j-1}})= \sum_{j=1}^k (\p[\omega]_{u_j}-\p[\omega]_{u_{j-1}})=
\p[\omega]_v-\p[\omega]_u\]
which proves our claim with $\b t_i=\p[\omega]_{u}-\p[\psi]_u$ for any $u\in C_i$.
\end{proof}	


The \defn{deformation cone}\footnote{To be precise, the \defn{deformation cone} is the cone of deformations in the space of functions $\verts \to \R^d$, and this should be called the \defn{edge-deformation cone} instead. Since both cones are isomorphic up to lineality, we decided to adopt this notation to simplify the presentation. } $\EDefoCone[\fw]\subset \R^{\edges}$ of $\fw$ is the set 
containing the edge-deformation vectors~${\dv(\fw[G])}$ of all the deformations $\fw[G]$ of $\fw$.
The following lemma gives an alternative characterization, that makes it clear that $\EDefoCone$ is a polyhedral cone.

\begin{lemma}\label{lem:edge-deformations-frameworks}
The deformation cone $\EDefoCone[\fw]$ of a framework~$\fw=(\verts,\edges,\p)$ is the set of nonnegative vectors $\dv\in\R_+^{\edges}$ such that 
\begin{equation}\label{eq:degenerate_equation}
\dv_e=0 \text{ for every degenerate edge $e\in\de$,}
\end{equation} and such that for every 
cycle $C$ of the graph $G = (\verts,\edges)$ with cyclically ordered vertices  $u_1 u_2\dots u_k$, the following \defn{cycle equation} is satisfied:
\begin{equation}\label{eq:cycle_equation}
\dv_{u_1u_2}(\p_{u_2}-\p_{u_1}) +\dv_{u_2u_3}(\p_{u_3}-\p_{u_2}) +\dots+\dv_{u_ku_1}(\p_{u_1}-\p_{u_k}) = \b 0
\end{equation}
\end{lemma}
\begin{proof}
	
	If $\fw[G]=(\verts,\edges,\p[\psi])$ is a deformation of $\fw$, then for every edge $uv\in \edges$ there is some $\lam_{uv}\geq 0$ such that
	\[\p[\psi]_v = \p[\psi]_u+(\p[\psi]_v-\p[\psi]_u)= \p[\psi]_u+\lam_{uv}(\p_v-\p_u)\]
	Therefore, for a cycle of $G = (V, E)$ with cyclically ordered vertices  $u_1 u_2\dots u_k$, we have 
	\[\p[\psi]_{u_1} = \p[\psi]_{u_1}+ \Bigl(\lam_{u_1u_2}(\p_{u_2}-\p_{u_1}) +\lam_{u_2u_3}(\p_{u_3}-\p_{u_2}) +\dots+\lam_{u_ku_1}(\p_{u_1}-\p_{u_k})\Bigr)\]
	showing that its edge-deformation vector $\b\lambda_{\fw}(\fw[G])$ satisfies the cycle equations. The fact that $\dv_e=0$ for $e\in\de$ is by definition.
	
Conversely, fix $\dv \in \R_+^\edges$ satisfying the cycle equations. Fix a node $u_i\in C_i$ for each connected component of $G$, and for every $v\in C_i$ consider a path $u_i=w_1\dots w_k=v$ from $u_i$ to $v$ in $G$ to define 
	\[\p[\psi]_v = \p_{u_i}+
	\Bigl(\dv_{w_1w_2}(\p_{w_2}-\p_{w_1}) +\dv_{w_2w_3}(\p_{w_3}-\p_{w_2}) +\dots+\dv_{w_kw_{k+1}}(\p_{w_{k}}-\p_{w_{k-1}})\Bigr).
	\]
	The cycle equations guarantee that $\p[\psi]_v$ is independent of the choice of the path, because the difference between two paths from $u_i$ to $v$ is a linear combination of cycles, and therefore it is well defined.
	In particular, if $vw\in \edges$, we can concatenate a path from $u_i$ to $v$ with the edge $vw$ to get: 
	$\p[\psi]_w =\p[\psi]_v +\dv_{vw}(\p_{w}-\p_{v})$.
	This shows that the framework $(\verts,\edges,\p[\psi])$ is a deformation of $\fw$ whose edge-deformation vector is~$\dv$.
\end{proof}

\begin{remark}
\Cref{eq:cycle_equation} is a linear equation between vectors in $\R^d$, and each cycle equation gives rise to $d$ equations between the real numbers $\dv_e$.
Therefore, the cycle equations provide up to $d$ times the dimension of the cycle space of $G = (V, E)$ many equations that a vector $\b\lambda$ has to satisfy
to be a deformation of $\fw = (\verts,\edges,\p)$.
In practice, this system of equations can be highly redundant.
\end{remark}

We call the \defn{deformation space} of $\fw$ the linear subspace of $\R^\edges$ defined by the equations~\eqref{eq:degenerate_equation} and \eqref{eq:cycle_equation}. The deformation cone $\DefoCone[\fw]$ is the intersection of the non-negative orthant $\R_+^\edges$ with the deformation space.

Fix an arbitrary orientation of the graph $G=(\verts,\edges)$, and let $\inc:\R^\edges\to\R^\verts$ be the associated incidence map that sends the oriented edge $\vec{uv}$ to $\b e_v-\b e_u$, where $\{\b e_v\}_{v\in\verts}$ is the standard basis of~$\R^\verts$. The \defn{cycle space} of $G$ over~$\R$ is $\ker(\inc)$. While this definition depends on the orientation, notice that changing the orientation only affects by an isomorphism that flips the signs of the corresponding variables. The signed incidence vectors of (oriented) cycles generate $\ker(\inc)$, and a basis is given by the fundamental cycles with respect to any spanning forest (see for example \cite[Chapter~4]{Biggs}). 

In particular, it suffices to consider a \defn{cycle basis} of the cycle space to define deformation spaces and deformation cones. The following presentation of this statement will be useful afterwards.

\begin{lemma}\label{lem:cyclebasis}
Let $\fw=(\verts,\edges,\p)$ be a framework, and $C_1,\dots,C_k$ be a collection of cycles generating the cycle space of $G=(\verts,\edges)$. Define $\fw_i=(\verts,\edges(C_i),\p)$ to be the restriction of $\fw$ to the edges in $C_i$. Then 
\[\EDefoCone[\fw]=\bigcap_{1\leq i\leq k}\EDefoCone[\fw_i].\]	
\end{lemma}

\begin{proof}
For $\dv\in\R^\edges$, consider the linear map $f_{\dv}\colon \R^\edges\to \R^d$ defined by $\b x\mapsto \sum_{uv\in\edges}(\p_v-\p_u)\b\lambda_{uv}\b x_{uv}$. 
Then $\dv$ satisfies the cycle equations if and only if every incidence vector of an oriented cycle belongs to the kernel of $f_{\dv}$. Of course, since incidence vectors of oriented cycles belong to the cycle space $\ker(\inc)$, it is enough that this is true for a generating set of $\ker(\inc)$, by linearity. The incidence vector of $C_i$ belongs to $\ker(f_{\dv})$ precisely when the cycle equations of $\fw_i$ are fulfilled.
\end{proof}

A connected framework $\fw=(\verts,\edges,\p)$ is \defn{indecomposable} if for every deformation $\fw[G]=(\verts,\edges,\p[\psi])$ of $\fw$ there are $\alpha\in \R_+$ and $\b t\in \R^d$ such that $\p[\psi]_v=\alpha \p_v+\b t$ for all $v\in \verts$. In the labeled configuration notation, this means that $\p[\psi](\verts)$ is (a translation of) a dilation of $\p(\verts)$.
Equivalently, $\fw$ is indecomposable if it is either empty ($\edges=\emptyset$) or its deformation cone $\EDefoCone[\fw]$ is one-dimensional.
By convention, a disconnected framework is decomposable.

\begin{remark}\label{rmk:dimension}
	If we restrict to non-degenerate edges, then the edge-deformation vector ${\dv(\fw[G])}$ belongs to the interior of the non-negative orthant $\R_+^\nde$ and to the deformation space. Since the deformation cone can be described as the intersection of $\R_+^\nde$ with the deformation space, this means that the dimension of the deformation cone coincides with the dimension of the deformation space. In particular, a nonempty framework is indecomposable if and only if its deformation space is $1$-dimensional.
    Note that, in practice, it is usually easier to determine the dimension of a vector space than the dimension of a cone.
\end{remark}

Motivated by the result below, we say that a framework $\fw$ is \defn{uniquely decomposable} if its deformation cone $\EDefoCone[\fw]$ is simplicial (\ie it has linearly independent rays). Note that if $\fw$ is uniquely decomposable, so are all its deformations.

\begin{lemma}\label{lem:uniquelydecomposable}
A framework has a simplicial deformation cone if and only if it has a unique decomposition as sum of indecomposable frameworks with distinct support, up to translation. 
\end{lemma}
\begin{proof}
Let $\fw_1,\dots,\fw_k$ be frameworks representing the rays of $\EDefoCone[\fw]$; that is, representatives (up to scaling and translation) of the indecomposable deformations of $\fw$. Note that the support of a deformation of $\fw$ indexes in which face of the deformation cone it lies (its facets are supported by equations of the form $\dv_e=0$), thus $\fw_1,\dots,\fw_k$ have distinct support.
Since $\fw$ belongs to the interior of $\EDefoCone[\fw]$ (as it does not verify any facet defining-equation), there are $\alpha_i>0$ such that $\fw=\alpha_1\fw_1+\cdots+\alpha_k\fw_k$.

Assume that $\EDefoCone[\fw]$ is simplicial. Setting $\fw_i'=\alpha_i\fw_i$, we have the decomposition $\fw=\fw_1'+\cdots+\fw_k'$.
Since the $\fw_i$ are linearly independent such a linear combination is unique.

For the converse, suppose that $\fw$ has a non-simplicial deformation cone. Then there is a nontrivial linear relation $\beta_1\fw_1+\cdots+\beta_k\fw_k=0$. Then for some $\varepsilon>0$ sufficiently small, \(\fw = \sum_{i=1}^k(\alpha_i+\varepsilon \beta_i)\fw_i\) also has positive coefficients, producing two distinct decompositions of $\fw$ into indecomposable frameworks.
\end{proof}

\begin{remark}
If $\fw$ has a unique decomposition, then any deformation $\fw[G]\in \DefoCone[\fw]$ has too.
Indeed, the face of a simplicial cone is also a simplicial cone, and $\DefoCone[{\fw[G]}]$ is a face of $\DefoCone[\fw]$.
\end{remark}

\subsubsection{Polytopes}

A \defn{polytope} $\pol$ in~$\R^d$ is the convex hull of finitely many points. Its \defn{faces} are the zero-sets of non-negative affine functions on~$\pol$.  For a linear functional $ \b c\in (\R^d)^*$, let \defn{$\pol^{\b c}$} be the face of $\pol$ containing the points maximizing ${\b c}$; that is to say $\pol^{\b c} \coloneqq \{\b x\in \pol ~;~ \b c(\b x) = \max_{\b y\in \pol} \b c(\b y)\}$. For a face $\pol[F]$, the set of $\b c$ such that $\pol^{\b c}=\pol[F]$ is the \defn{normal cone} of $\pol[F]$, and the union of the normal cones forms the \defn{normal fan}. Two polytopes are \defn{normally equivalent} if they share the same normal fan.
Faces of dimension~$0$, dimension~$1$, and codimension~$1$ are called \defn{vertices}, \defn{edges} and \defn{facets}, respectively. 
The sets of vertices, edges and facets of $\pol$ are denoted $\verts(\pol)$, $\edges(\pol)$, and $\facets(\pol)$. 


The \defn{framework} of~$\pol$ is \defn{$\fw(\pol)\eqdef(\verts(\pol),\edges(\pol),\id)$}. To avoid notational clutter, we allow for slight abuses of notation in this definition. Firstly, we identify an edge $\pol[e]\in \edges(\pol)$ of the polytope (which formally is a face of $\pol$, and hence a segment in~$\R^d$) with the pair $\{\b p,\b q\}$ formed by the two vertices of $\pol$ contained in $\pol[e]$. Secondly, we use the identity map $\id: \verts(\pol)\hookrightarrow \R^d$ to denote the canonical inclusion of $\verts(\pol)$ inside $\R^d$.

The \defn{Minkowski sum} of the polytopes $\pol$ and $\pol[Q]$ is the polytope $\pol + \pol[Q] \eqdef \bigl\{\b p + \b q ~;~ \b p \in \pol, ~ \b q\in \pol[Q]\bigr\}$.
Note that Minkowski sums of polytopes are always defined, whereas Minkowski sums of frameworks are defined only if the two frameworks share a common underlying graph.
We say $\pol[Q]$ is a \defn{deformation} or a \defn{weak Minkowski summand} of $\pol$ if there exists a polytope $\pol[R]$ and a real number $\b\lambda > 0$ such that  $\pol = \b\lambda \pol[Q] + \pol[R]$.
A polytope $\pol$ is \defn{(Minkowski) indecomposable} if its only deformations are of the form $\b\lambda \pol+\b t$ with $\b\lambda \geq 0$ and $\b t\in\R^d$ (the polytope $0 \,\pol$ is the point $\b 0$), see \Cref{fig:MinkowskiSumAndEdgeLengthVector} (left).

\begin{figure}
\centering
\begin{subfigure}[b]{0.6\linewidth}
     \centering
     \includegraphics[width=\linewidth]{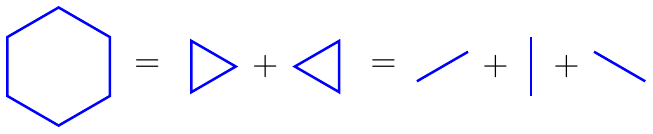}
     \caption{Two different sums giving a regular hexagon}
     \label{sfig:MinkowskiSum}
\end{subfigure}\hfill
\begin{subfigure}[b]{0.28\linewidth}
     \centering
     \includegraphics[width=\textwidth]{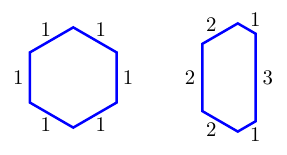}
     \caption{Edge-deformation vectors}
     \label{sfig:EdgeLengthVector}
\end{subfigure}
\caption[(Minkowski) sums and edge-deformation vector]{(Left) The regular hexagon can be written as the Minkowski sum of indecomposable polytopes in several ways.
(Right) Two deformations of the regular hexagon, where each edge $\pol[e]$ is labeled by the coordinate of the associated edge-deformation vector $\dv(\pol[Q])_{\pol[e]}$.}
\label{fig:MinkowskiSumAndEdgeLengthVector}
\end{figure}


For two polytopes $\pol, \pol[Q]\subseteq\R^d$ and $\b c\in (\R^d)^*$, we have $(\pol+\pol[Q])^{\b c}=\pol^{\b c}+\pol[Q]^{\b c}$. 
Hence, if $\pol[Q]$ is a deformation of $\pol$, then, $\dim(\pol[Q]^{\b c})\leq \dim(\pol^{\b c})$ for all ${\b c}\in (\R^d)^*$. In particular, there is a (not necessarily injective but surjective) correspondence $f_{\pol\pol[Q]}:\verts(\pol)\to \verts(\pol[Q])$ between the vertices of $\pol$ and the vertices of $\pol[Q]$. 
Moreover, consider a pair $\b p, \b p'$ of vertices of $\pol$ forming an edge. If $\b q, \b q'$ are the associated vertices of $\pol[Q]$, then $\conv(\b q, \b q')$ is either a point or an edge parallel to $\conv(\b p,\b p')$.
It must have the same direction as $\conv(\b p,\b p')$ since the normal vectors to the vertices of $\pol$ are normal vectors to the associated vertices of $\pol[Q]$. Therefore, there is some $\lam\in \R_+$ such that 
\begin{equation}\label{eq:ELDpol}\lam(\b p-\b p')=\b q-\b q'.\end{equation}

In \cite{Shephard1963}, Shephard proved that the previous property characterizes deformations of $\pol$.
In particular, this proves directly the following lemma, which reduces polytope deformations to framework deformations. This correspondence was first formalized by 
Kallay~\cite{Kallay1982}, who initiated the study of framework deformations in full generality.

\begin{lemma}[{\cite{Shephard1963}, \cite{Kallay1982}}]
Let $\pol\subset\R^d$ be a polytope.
	\begin{compactenum}
		\item If $\fw[G]=(\verts(\pol), \edges(\pol),\p )$ is a deformation of $\fw(\pol)$, then $\conv(\{\p_{\b p}\}_{\b p\in \verts(\pol)})$ is a deformation of $\pol$.
		\item If $\pol[Q]$ is a deformation of $\pol$, then $\fw[G]=(\verts(\pol),\edges(\pol),f_{\pol\pol[Q]})$ is a deformation of $\fw(\pol)$.
	\end{compactenum}
\end{lemma}

This allows one to define the \defn{edge-deformation vector} 
of a deformation $\pol[Q]$ of $\pol$ as $\mathdefn{\dv_{\pol}(\pol[Q])}\eqdef{\dv_{\fw(\pol)}(\fw[G])}\in\R_+^{\edges(\pol)}$, where $\fw[G]=(\verts(\pol),\edges(\pol),f_{\pol\pol[Q]})$. We will denote it simply $\dv(\pol[Q])$ when ${\pol}$ is clear from the context, see \Cref{fig:MinkowskiSumAndEdgeLengthVector} (right). Note that, up to translation, $\pol[Q]$ is completely determined by its edge-deformation vector (by \cref{lem:edge-deformations-frameworks}, since the graph of a polytope is always connected).
Similarly, we define the \defn{deformation cone} of a polytope $\pol$ as $\mathdefn{\EDefoCone[\pol]} \eqdef \EDefoCone[\fw(\pol)]$. Note that a polytope is indecomposable if and only if its framework is, which is equivalent to the deformation cone being of dimension~$1$. We say that a polytope is \defn{uniquely decomposable} if its deformation cone is simplicial, and we have the following direct corollary of \Cref{lem:uniquelydecomposable}, where we call two polytopes \defn{normally distinct} if they do not have the same normal fan.

\begin{corollary}\label{cor:uniquelydecomposablepolytope}
	A polytope has a simplicial deformation cone if and only if it has a unique decomposition as sum of indecomposable polytopes that are pairwise normally distinct, up to translation. 
\end{corollary}

\begin{proof}
Direct application of \Cref{lem:uniquelydecomposable} to the case of polytopes.
\end{proof}

\begin{remark}\label{rmk:cyclebasispol}
For a polytope $\pol$, each $2$-face $\polF$ supports a cycle in the graph $G = \bigl(\verts(\pol), \edges(\pol)\bigr)$.
The associated cycle equations are called the \defn{polygonal face equations} of $\pol$ associated to $\polF$.
The collection of cycles supported by the $2$-faces of $\pol$ generates the cycle space of $G$ \cite{PinedaVillavicencio2022}.
Hence, by \Cref{lem:cyclebasis}, the deformation cone of a polytope $\pol$ can equivalently be defined as the set of edge-deformation vectors satisfying all polygonal face equations of $\pol$.
This is the description given in \cite[Definition 15.1 (2)]{PostnikovReinerWilliams}.
\end{remark}

It is worth noting that while the point of view of frameworks leads to this parametrization of the cone of deformations of $\pol$ (up to translation) by the edge-deformation vectors, there exist other alternative parametrizations that are also commonly used (\eg via support functions and wall-crossing inequalities, or via Gale duality).
We refer to \cite[Appendix 15]{PostnikovReinerWilliams} for more details.

\subsection{The graph of edge dependencies}\label{ssec:GraphEdgeDependencies}

Our main tool to certify polytope indecomposability will be the graph of edge dependencies $\ELD$ associated to a framework~$\fw$, which will be used to certify that the deformation cone $\EDefoCone$ is one-dimensional. 
The graph $\ELD$ also has other interesting applications, and we will showcase some examples afterwards. 

\subsubsection{Definition and examples}

\begin{definition}
The \defn{graph of edge dependencies} of a framework $\fw=(\verts,E,\p)$, denoted \defn{$\ELD$} is the graph whose node set is $\nde(\fw)$, the set of non-degenerate edges of $\fw$, and where two edges $e,f\in \nde(\fw)$ are linked with an arc if for every $\dv\in \EDefoCone[\fw]$ we have $\dv_{e}=\dv_{f}$.
In this case, we say that ${e}$ and ${f}$ are \defn{dependent}. 
In particular, $\{e,e\}$ is a loop of $\ELD$ for every $e\in \nde$. 
A set of edges is said to be dependent if all the edges it contains are pairwise dependent.
For a polytope $\pol$, we abbreviate $\ELD[\pol]$ for the graph of edge dependencies of its framework $\fw(\pol)$. See \Cref{fig:Cube_and_SmallExamples} for some examples.
\end{definition}

Before listing properties of these graphs, we give some examples of graphs of edge dependencies of polytopes. In particular, triangles and parallelograms will become building blocks for our constructions.

\begin{example}\label{ex:triangle}
If $\pol$ is a $2$-dimensional triangle, then there is only 1 cycle equation, which gives 2 linearly independent relations between 3 variables, so the space of solutions is of dimension $3 - 2 = 1$. This implies that the three edges of $\pol$ are pairwise dependent, that $\ELD[\pol]$ is the complete graph (with loops) on $\edges(\pol)$, and that $\pol$ is indecomposable, see \Cref{sfig:Triangle}. This example goes back to Gale's seminal abstract \cite{Gale1954}.
\end{example}

\begin{example}\label{ex:parallelogram}
If $\pol$ is a $2$-dimensional quadrilateral with edges $\pol[e], \pol[f], \pol[e]', \pol[f]'$ in cyclic order, then $\ELD[\pol]$ depends on the parallelisms between the edges of $\pol$.
In addition to the loops on the nodes of $\ELD[\pol]$, we have:
\begin{compactitem}
\item If no two edges of $\pol$ are parallel, then $\ELD[\pol]$ consists of~$4$ isolated nodes on $\pol[e], \pol[f], \pol[e]', \pol[f]'$, \Cref{sfig:ScaleneQuadrilateral}. 
\item If $\pol[e]$ and $\pol[e]'$ are parallel, \ie if $\pol$ is a trapezoid on bases $\pol[e]$ and $\pol[e]'$, then $\pol[f]\pol[f]'$ is an arc of $\ELD[\pol]$. If  $\pol[f]$ and $\pol[f]'$ are not parallel, then this is the only arc of~$\ELD[\pol]$, \Cref{sfig:Trapezoid}.
\item Finally, if $\pol$ is a parallelogram, then $\ELD[\pol]$ consists of two disjoint arcs $\pol[e]\pol[e]'$ and $\pol[f]\pol[f]'$, \Cref{sfig:Parallelogram}.
\end{compactitem}	
\end{example}

\begin{example}
In \Cref{fig:Cube_and_SmallExamples} (Bottom) we display some $3$-dimensional polytopes: the $3$-cube, a triangular prism, the hemicube, and a hexagonal pyramid. To present their graphs of edge dependencies, we only depict some arcs, the remaining can be recovered by replacing all connected components by cliques (\cf \cref{lem:WclusterGraph}), and adding loops.
The graph of edge dependencies of the $3$-dimensional cube (\Cref{sfig:Cube}) consists of $3$ cliques of size $4$, one per class of parallel edges of the cube. 
Thanks to \Cref{ex:parallelogram}, we see that in the standard cube $[0, 1]^d$, two edges are dependent if and only if they are parallel.

This example, along with the triangular prism and the hemicube (\Cref{sfig:Prism,sfig:Hemicube}), shows that there can be a subset of edges which induces a connected subgraph in $\ELD$, but a disconnected subframework of~$\fw$ (recall that in $\ELD$ the edges of $\fw$ play the role of nodes).

\begin{figure}[htpb]
\centering
\begin{subfigure}[b]{0.21\linewidth}
     \centering
     \includegraphics[width=\textwidth]{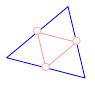}
     \caption{Triangle}
     \label{sfig:Triangle}
\end{subfigure}
\begin{subfigure}[b]{0.24\linewidth}
     \centering
     \includegraphics[width=\textwidth]{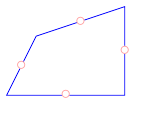}
     \caption{Scalene quadrilateral}
     \label{sfig:ScaleneQuadrilateral}
\end{subfigure}
\begin{subfigure}[b]{0.29\linewidth}
     \centering
     \includegraphics[width=\textwidth]{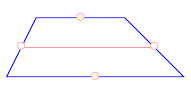}
     \caption{Trapezoid}
     \label{sfig:Trapezoid}
\end{subfigure}
\begin{subfigure}[b]{0.21\linewidth}
     \centering
     \includegraphics[width=\textwidth]{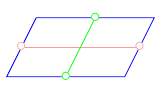}
     \caption{Parallelogram}
     \label{sfig:Parallelogram}
\end{subfigure}
\begin{subfigure}[b]{0.21\linewidth}
     \centering
     \includegraphics[width=\textwidth]{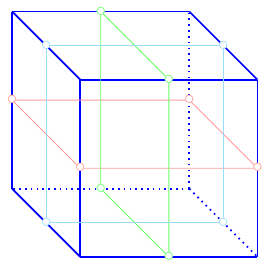}
     \caption{Cube}
     \label{sfig:Cube}
\end{subfigure}
\begin{subfigure}[b]{0.24\linewidth}
     \centering
     \includegraphics[width=\textwidth]{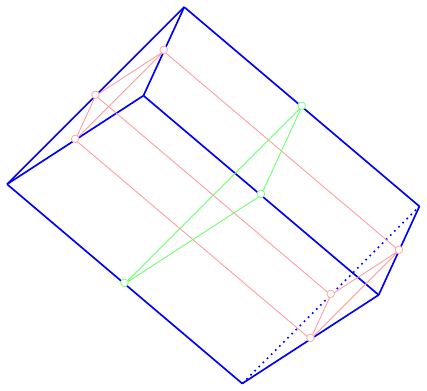}
     \caption{Prism over a triangle}
     \label{sfig:Prism}
\end{subfigure}
\begin{subfigure}[b]{0.29\linewidth}
     \centering
     \includegraphics[width=\textwidth]{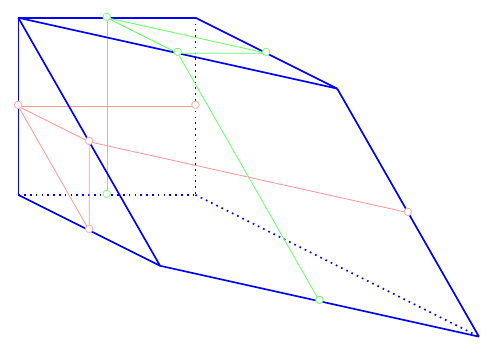}
     \caption{Hemicube}
     \label{sfig:Hemicube}
\end{subfigure}
\begin{subfigure}[b]{0.21\linewidth}
     \centering
     \includegraphics[width=\textwidth]{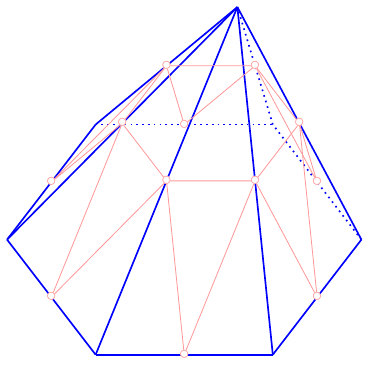}
     \caption{Pyramid}
     \label{sfig:Pyramid}
\end{subfigure}
\caption[Graph of edge dependencies of a cube, a prism, a hemicube and a pyramid]{(Left to right) The $3$-cube, a prism over a triangle, a hemicube, a pyramid on a hexagon.
    Each polytope is in \textcolor{blue}{blue}. In \textcolor{OrangeRed}{red}, \textcolor{green}{green} and \textcolor{cyan}{cyan}, the subgraphs of $\ELD[\pol]$ obtained by linking opposite edges of parallelograms and edges of triangles.
    In each case, the graphs $\ELD[\pol]$ is the (union of) cliques on the subgraphs drawn.
    Note that the $3$-cube is the Minkowski sum of 3 linearly independent segments; the prism over a triangle is the sum of a triangle and a (linearly independent) segment; the hemicube is the sum of 2 orthogonal triangles.
    The pyramid is indecomposable.
    }
\label{fig:Cube_and_SmallExamples}
\end{figure}
\end{example}	

\subsubsection{Relation to indecomposability}

We start with the straightforward but central observation that indecomposability can be read from the graph of edge dependencies.
Note that if $\fw$ is a disconnected framework, then $\ELD$ is also disconnected.

\begin{lemma}\label{lem:ELDindecomposable}
A framework $\fw$ is indecomposable if and only if $\ELD$ is a complete graph.
\end{lemma}

\begin{proof}
A connected framework $\fw=(\verts,\edges,\p)$ is indecomposable if and only if for every deformation $\fw[G]=(\verts,\edges,\p[\psi])$ of $\fw$ there are $\b\lambda\in \R_+$ and $\b t\in \R^d$ such that $\p[\psi]_v=\b\lambda \p_v+\b t$ for all $v\in \verts$. Since $\dv(\fw[G])$ determines $\fw[G]$ up to translation, this is equivalent to the existence, for any deformation $\fw[G]$, of a $\b\lambda\in\R_+$ such that $\dv(\fw[G])_e=\b\lambda$ for any non-degenerate edge $e\in\nde$, which is equivalent to $\ELD$ being complete.

A disconnected framework is decomposable, and its graph of edge dependencies is not connected.
\end{proof}

Dependency is an equivalence relation on $\edges$.
Thus, $\ELD$ is what is sometimes called a \defn{cluster graph}:

\begin{lemma}\label{lem:WclusterGraph}
All the connected components of the graph of edge dependencies $\ELD$ are cliques.
\end{lemma}

\begin{proof}
Consider three non-degenerate edges $e,f,g\in\nde$. 
If for every $\dv\in\EDefoCone[\fw]$ we have $\dv_{e}=\dv_{f}$ and $\dv_{f}=\dv_{g}$, then we have $\dv_{e}=\dv_{g}$. Hence, adjacency is transitive: all connected components are cliques.
\end{proof}

This observation will be useful in order to relax the condition from \cref{lem:ELDindecomposable}.

\begin{corollary}\label{cor:WconnectedPMinkowskiIndecomposable}
The framework $\fw$ is indecomposable if and only if $\ELD$ is connected (in this case it is a complete graph).
\end{corollary}

\begin{proof}
By \cref{lem:ELDindecomposable}, the framework $\fw$ is indecomposable if and only if $\ELD$ is a complete graph, and by \cref{lem:WclusterGraph}, this is equivalent to $\ELD$ being connected.
\end{proof}

This is a particular case of the following more general result.
We will not study it in detail here because we concentrate on indecomposability, but see \cref{sec:edgedecomposable,sec:constructions} for applications of such theorems, and more general results on deformation cones obtained via the indecomposability graph. Our stronger indecomposability results will also translate into similar bounds for the dimension of the deformation space. 

\begin{theorem}\label{thm:dim1}
For a framework $\fw$, the dimension of $\DefoCone[\fw]$ is bounded above by the number of connected components of $\ELD$.
\end{theorem}

\begin{proof}
If $\c C = (C_1, \dots, C_r)$ are the connected components of $\ELD$, then $\DefoCone[\fw]$ belongs to the linear subspace $\bigcap_{C_i\in \c C}\{\b\lambda ~;~ \b\lambda_{e} = \b\lambda_{f} \,\text{ for }\, e,f \in C_i\}$, whose dimension is $|\c C|$.
\end{proof}

\subsubsection{Subframeworks}
Our strategy for showing indecomposability consists in successively revealing arcs of $\ELD$ until proving its connectedness. This is done in a local to global approach, via subframeworks.

If $\fw=(\verts,\edges, \p)$ is a framework, we say that $\fw'=(\verts',\edges', \p')$ is a \defn{subframework} if $\verts'\subseteq \verts$, $\edges'\subseteq \edges$, and $\p'=\p_{|\verts'}$.
While previous works do not use this language explicitly, many exploited indecomposable subframeworks (specially arising from faces, in the polytopal case) to prove the connectedness of $\ELD$, and hence indecomposability.

\begin{proposition}\label{prop:subframework}
	If $\fw'=(\verts',\edges', \p')$ is a subframework of $\fw=(\verts,\edges, \p)$, then the canonical projection $\pi:\R^\edges\to\R^{\edges'}$ that restricts to the coordinates indexed by elements in $\edges'$ induces a linear map \linebreak $\pi:\DefoCone[\fw]\to\DefoCone[\fw']$.
	In particular, $\ELD[\fw']$ is a subgraph of $\ELD$.
\end{proposition}

\begin{proof}
	Note that $\pi$ maps the non-negative orthant $\R_+^\edges$ onto the non-negative orthant $\R_+^{\edges'}$. It only remains to check that the cycle equations of $\fw'$ are also valid for $\fw$.
    This is immediate, as every cycle of $\fw'$ is a cycle of $\fw$.
    This also directly gives that every pair of dependent edges in $\fw'$ is also dependent in $\fw$.
\end{proof}

%

\begin{corollary}\label{cor:indecsubframeclique}
If $\fw'$ is an indecomposable subframework of $\fw$, then the edges of $\pol[\fw']$ form a clique in~$\ELD$.
\end{corollary}

\begin{proof}
Immediate from \Cref{lem:ELDindecomposable,prop:subframework}.
\end{proof}
%

\begin{example}\label{ex:triangles}
One of the earliest indecomposability criteria states that any polytope whose $2$-faces are all triangles (\eg the simplicial polytopes, the polytopes whose 1-skeleton is complete, or the 24-cells) is indecomposable \cite[Statement (13)]{Shephard1963}.
This can be easily proven using the graph of edge dependencies.

Let $\pol$ be a polytope whose $2$-faces are triangles. Triangles are indecomposable by \cref{ex:triangle}. Therefore, if two edges of $\pol$ belong to a common triangle, they are dependent.
This implies that $\ELD[\pol]$ contains as a subgraph the 1-incidence graph of $\pol$ (the graph whose nodes are the edges of $\pol$, and where two edges are linked by an arc if they belong to a common $2$-face). The latter is connected (it is even $2(\dim\pol-1)$-connected if $\dim\pol\geq 4$ and $3$-connected if $\dim\pol=3$, see \cite[Thm. 1.1]{Athanasiadis-IncidenceGraphConnectivity}).
By \Cref{cor:WconnectedPMinkowskiIndecomposable}, the polytope $\pol$ is indecomposable. 
%
\end{example}

\begin{example}\label{ex:pyramid}
Similarly, another seminal result by Gale~\cite{Gale1954} states that any pyramid is indecomposable.
To showcase the graph of edge dependencies, we use dependent edges to re-prove it, see \Cref{sfig:Pyramid}.
	
Consider the pyramid $\pol[pyr](\pol)$ over a polytope $\pol$ with apex $\b q$.
All the $2$-faces containing the apex of the pyramid are triangles $\conv(\pol[e]\cup\{\b q\})$ for an edge $\pol[e]\in \pol$.
As triangles are indecomposable, their three edges are dependent.
If two edges $\pol[e]$ and $\pol[f]$ of $\pol$ share an endpoint, then the two associated triangles $\conv(\pol[e]\cup\{\b q\})$ and $\conv(\pol[f]\cup\{\b q\})$ share an edge, hence their edge sets are pairwise dependent.
Finally, as the 1-skeleton of $\pol$ is connected, so is its line graph. Consequently $\ELD[{\pol[pyr](\pol)}]$ is connected, and $\pol[pyr](\pol)$ is indecomposable according to \Cref{cor:WconnectedPMinkowskiIndecomposable}.
\end{example}

\subsection{Implicit edges and the graph of implicit edge dependencies}

\subsubsection{Implicit edges}
For many applications, it is useful to extend a framework with additional edges consisting of pairs of vertices that always satisfy an equation like \eqref{eq:ELDF} even if they do not originally form an edge. This allows for a much simpler and uniform treatment. Moreover, they become very useful in combination with the results of the previous section, as they can be used to uncover subframeworks with controlled graph of edge-dependencies, and are a pillar for the techniques in the upcoming sections.

\begin{definition}
	An \defn{implicit edge} of a framework~$\fw=(\verts,\edges,\p)$ is a pair of vertices $u,v\in V$ such that for any deformation $\fw[G]=(\verts,\edges,\p[\psi])$ of $\fw$, there exists some $\b\lambda_{uv}\in\R_+$ such that
	\begin{equation}\label{eq:implicitedge}\p[\psi]_v-\p[\psi]_u=\b\lambda_{uv}(\p_v-\p_u).\end{equation}

	The \defn{closure} of $\fw$ is the framework $\Efw=(\verts,\Iedges,\p)$, where $\Iedges$ is the set of implicit edges of $\fw$ (note that $\edges\subseteq \Iedges$).
	An implicit edge $uv\in \Iedges$ is \defn{degenerate} if $\p_u=\p_v$, and \defn{non-degenerate} otherwise. The sets of degenerate and non-degenerate implicit edges of $\fw$ are respectively denoted~$\die(\fw)$ and $\ndie(\fw)$ (or just $\die$ and $\ndie$ if the framework is clear from the context).
	
	The \defn{graph of implicit edge dependencies} of $\fw$, is the graph of edge dependencies of the closure~\defn{$\EELD$}.
	
\end{definition}
Note that with our definition, all edges are also implicit edges (but in general not the other way round). 
For our purposes, the most important property of implicit edges is that adding them does not alter the space of deformations, but many properties of $\ELD$ are explained in a much simpler and uniform way if implicit edges are treated as any other edge. The following result motivates this choice.

\begin{theorem}\label{thm:characterizationimplicit}
Let $\fw=(\verts,\edges,\p)$ be a framework, $e=uv\notin \edges$ for some vertices $u,v\in \verts$, and $\fw'=(\verts,\edges\cup\{e\},\p)$. 
Then the canonical projection $\pi:\EDefoCone[\fw']\to \EDefoCone[\fw]$ that forgets about the coordinate indexed by $e$ is a linear isomorphism if and only if $e$ is an implicit edge of $\fw$.
\end{theorem}

\begin{proof}
As $\fw$ is a subframework of $\fw'$, every cycle equation for $\fw$ is also a cycle equation for $\fw'$, which shows that $\pi$ is well defined. 

If $u$ and $v$ are in different connected components of $\fw$, then $e$ is not implicit because we can freely translate each connected component by \Cref{lem:connectedcomponents}. The map $\pi$ is not injective, because translating the connected component of $v$ by a positive multiple of $\p_v-\p_u$ yields a valid deformation of $\fw'$ where all the edge-lengths are unchanged except for that corresponding to $e$.

If $u$ and $v$ are in the same connected component, but $e$ is not implicit, this means that there is some deformation $\fw[G]=(\verts,\edges,\p[\psi])$ of $\fw$ in which $\p[\psi]_v-\p[\psi]_u\neq \b\lambda (\p_v-\p_u)$ for all $\b\lambda\geq 0$, which means that its deformation vector does not belong to the image of $\pi$, and $\pi$ is not surjective.

If $e$ is implicit, in contrast, then any deformation of $\fw$ is also a deformation of $\fw'$ (by definition of $e$ being implicit), which shows that the map is surjective. Moreover, if $u=u_0,u_1,\dots,u_r=v$ is a path in $\fw$, then for any $\dv'\in\EDefoCone[\fw']$ we have 
\[\dv'_e(\p_v - \p_{u})= \sum\nolimits_{1\leq i\leq r} \dv'_{u_{i-1}u_i}(\p_{u_i}-\p_{u_{i-1}}),\]
which shows that $\dv'_e$ can be uniquely determined by $\pi(\dv')$, and thus that the map is injective.
\end{proof}

This directly gives the following properties, which imply that to study the graph of edge dependencies~$\ELD$, we can always study the graph of implicit edge dependencies~$\EELD$ instead, and later restrict to its subgraph $\ELD$

\begin{corollary}\label{cor:extendedgraph}
	Let $\fw=(\verts,\edges,\p)$ be a framework, and $\Efw=(\verts,\Iedges,\p)$ its closure. 
	\begin{compactenum}[(i)]
	\item The deformation cones $\EDefoCone[\fw]\subset\R^\edges$ and $\EDefoCone[\Efw]\subset\R^{\Iedges}$ are linearly isomorphic.
	\item $\ELD$ is the subgraph of $\EELD$ induced by the node set $\edges$.
	\item $\fw$ is indecomposable if and only if $\Efw$ is.
	\end{compactenum}
\end{corollary}

\begin{proof}
$(i)$ follows from multiple applications of \Cref{thm:characterizationimplicit}.
$(ii)$ and $(iii)$ result from the application of $(i)$ to the context of graphs of edge dependencies and to indecomposability.
\end{proof}

In particular, we have the following characterization in dimensions $\geq 2$. (The statement is false for frameworks of dimension $1$, where every non-degenerate pair gives an implicit edge, but if there are at least three distinct points then they can be 
moved freely in the line without breaking collinearity, which shows that these edges are not pairwise dependent and the framework is decomposable.)

\begin{lemma}\label{lem:Dim2AllImplicitImpliesDependent}
Let $\fw$ be a framework, and $S$ a subset of vertices such that $\p(S)$ is of dimension at least $2$.
If every pair of vertices of $S$ is an implicit edge, then all these implicit edges are pairwise dependent (when they are non-degenerate).
\end{lemma}

\begin{proof}
	  We prove it by induction on $|S|$. If $|S|=3$, and $S=\{u,v,w\}$, then $\p_u,\p_v,\p_w$ are affinely independent because the dimension is $\geq 2$, \ie the vertices of a triangle. By \Cref{ex:triangle}, we have that the restriction of $\Efw$ to $S$ is indecomposable, and hence these implicit edges are pairwise dependent. 
	
	If $|S|>3$, there must be some $u$ such that $S\ssm\{u\}$ is of dimension $\geq 2$. By induction hypothesis, the non-degenerate implicit edges among vertices of $S\ssm\{u\}$ are pairwise dependent. Now let $v\in S\ssm\{u\}$ such that $uv$ is a non-degenerate implicit edge. We want to prove that $uv$ is dependent with the non-degenerate implicit edges of $S\ssm\{u\}$. To this end, choose some $w$ such that $\p_u,\p_v,\p_w$ are affinely independent (which exists because $S$ is of dimension $\geq 2$). By \Cref{ex:triangle} again, we have that the restriction of $\Efw$ to $\{u,v,w\}$ is indecomposable, and thus $uv$ and $uw$ are dependent. We conclude that all the non-degenerate implicit edges among vertices in $S$ are pairwise dependent.
\end{proof}

\begin{corollary}\label{cor:indeccompleteclosure}
	A framework $\fw=(\verts,\edges,\p)$ of dimension $\geq 2$ is indecomposable if and only if every pair of vertices is an implicit edge. 
\end{corollary}

\begin{proof}
If $\fw$ is indecomposable, then for any deformation $\fw[G]=(\verts,\edges,\p[\psi])$ we have $\psi=\b\lambda \p +\b t$ for some $\b\lambda\in \R$ and $\b t\in \R^d$, and every pair of vertices fulfills \eqref{eq:implicitedge} with this $\b\lambda$.

Conversely, if every pair of vertices of $\fw$ is an implicit edge, then as $\fw$ is of dimension at least 2, by \Cref{lem:Dim2AllImplicitImpliesDependent}, the graph $\EELD$ is a complete graph.
By \Cref{cor:extendedgraph} (iii), $\fw$ is indecomposable.
\end{proof}


	

\subsubsection{Unveiling implicit edges}
While it is not straightforward to compute all implicit edges (\ie the nodes of $\EELD$), this is usually not needed.
The next result shows a simple tool to derive implicit edges. More involved methods will be presented in \Cref{ssec:CoveringFamilyOfFlats}.

\begin{proposition}\label{prop:connectedimplicit}
	If $u$ and $v$ are vertices of a framework $\fw=(\verts,\edges,\p)$ that are connected through a path of pairwise dependent edges. Then $uv$ is an implicit edge, and it is dependent in $\EELD$ with all the edges in this path.
\end{proposition}

\begin{proof}
	Suppose the path is $u=u_0,u_1,\dots, u_k=v$. Let $\fw[G]=(\verts,\edges,\p[\psi])$ be a deformation of $\fw$, and $\dv$ be its deformation vector. The edges $u_{i-1}u_i$ are pairwise dependent, and hence fulfill 
	$\dv_{u_{0}u_1}=\dots=\dv_{u_{k-1}u_k}\defeq \b\lambda$ for some $\b\lambda\geq 0$. We have therefore :
	\[\p[\psi]_v-\p[\psi]_u=\sum_{1\leq i\leq k} \p[\psi]_{u_i} - \p[\psi]_{u_{i-1}}=\b\lambda \sum_{1\leq i\leq k} (\p_{u_i} - \p_{u_{i-1}})=\b\lambda(\p_v-\p_u),\]
	which shows that $uv$ is an implicit edge with $\dv_{uv}(\fw[G])=\dv_{u_{i-1}u_i}(\fw[G])$ for any deformation $\fw[G]$.
\end{proof}

\begin{example}\label{exm:KallayConditionallyDecomposable}
	
	To showcase the use of implicit edges, we reproduce in \Cref{fig:KallayExample} a clever construction by Kallay \cite[Paragraph 5 \& Fig. 1]{Kallay1982} of 
	two combinatorially equivalent polytopes $\pol$ and ${\pol[Q]}$ such that $\pol$ is decomposable but ${\pol[Q]}$ is indecomposable, and analyze them via the graph of implicit edge dependencies.
    These polytopes are combinatorially equivalent to the Johnson solid $J_{15}$, the elongated square bipyramid, see \cite{Johnson1966-Solids}.
    Here, for $\pol$ a polytope, we denote $\EELD[\pol]$ its graph of implicit-edge dependencies $\ELD[\Efw(\pol)]$.

	\begin{figure}[htpb]
		\centering
		\includegraphics[width=0.85\linewidth]{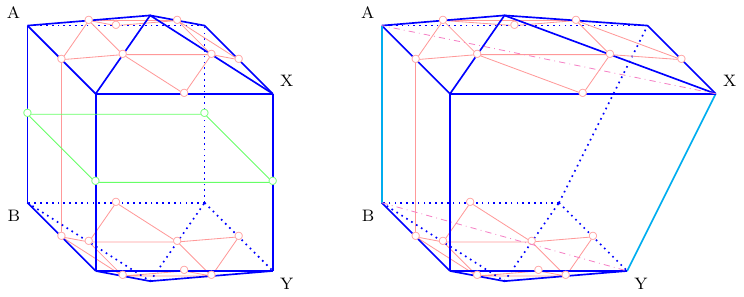}
		\caption[A conditionally-decomposable polytope \& its graph of implicit edge dependencies]{\cite[Fig. 1]{Kallay1982} (left) A decomposable realization of a twice-stacked cube.
			(Right) An indecomposable realization of a twice-stacked cube, where the edges $AB$ and $XY$ are not coplanar.
			In each, we draw some edges of $\ELD$: edges of the same color are dependent.
			On the right, using the rigid cycle (whose convex hull is a tetrahedron) we prove that $AB$ and $XY$ are also part of the \textcolor{OrangeRed}{red} component.}
		\label{fig:KallayExample}
	\end{figure}

	The polytope $\pol$ is obtained by taking the standard cube and stacking one vertex on top of two opposite 2-faces, see \Cref{fig:KallayExample} (left).
	The resulting polytope $\pol$ is the sum of an octahedron and the segment $[A, B]$.
	Its graph of edge dependencies $\ELD[\pol]$ is a clique on the union of the two pyramids (\textcolor{OrangeRed}{red} in \Cref{fig:KallayExample} left), and a clique on the edges parallel to $AB$ (\textcolor{green}{green} in \Cref{fig:KallayExample} left).
	Its graph of implicit edge dependencies $\EELD[\pol]$ is a clique on the union of \textbf{pairs of vertices} of each pyramid, and a clique on the edges parallel to~$AB$. One can easily see that it is decomposable because doing a translation on one of the pyramids parallel to $AB$ gives a valid deformation.
	
	The polytope ${\pol[Q]}$ is obtained by taking a non-standard cube, where the edges $AB$ and $XY$ are not coplanar, and stacking vertices on opposite 2-faces.
	We cannot prove that ${\pol[Q]}$ is indecomposable using our tools on the graph of edge dependencies $\ELD[{\pol[Q]}]$, but it is straightforward using the graph of implicit edge dependencies $\EELD[{\pol[Q]}]$:
	As the two opposite pyramids are indecomposable by \Cref{ex:pyramid}, their edges induce cliques on $\ELD[\polQ]$ by \Cref{cor:indecsubframeclique}.
    By \Cref{prop:connectedimplicit}, $AX$ and $BY$ are implicit edges and are dependent with the edges of each pyramid;
	hence, $A, B, Y, X$ is a cycle of implicit edges whose convex hull is a 3-simplex. We will see in the upcoming \Cref{cor:RigidCycle0}, that this implies that the pairs $AB$, $AX$, $XY$ and $BY$ form a clique in $\EELD[{\pol[Q]}]$.
    Finally, as the vertical edges of $\polQ$ are dependent with $AB$ or with $XY$, all edges of $\polQ$ are dependent with $AB$: hence $\polQ$ is indecomposable by \Cref{cor:WconnectedPMinkowskiIndecomposable}.
\end{example}

\subsection{Degenerate edges and projections of frameworks}\label{ssec:DegeneracyAndProjection}

The choice of working with frameworks instead of polytopes introduced the possibility of having degenerate edges, which needed a separate treatment in the definition of deformation vectors in~\eqref{eq:defovector}. Nevertheless, we consider degenerate edges to be a feature instead of a bug, and we will exploit them to obtain simple proofs.

\subsubsection{Degenerate edges, quotients, and lifts}

\begin{lemma}\label{lem:degenerate}
	Let $uv\in\de$ and $uw\in \nde$ be respectively a degenerate and a non-degenerate edge of a framework $\fw=(\verts,\edges,\p)$. Then $vw$ is an implicit edge of $\fw$ that is dependent with $uw$ in $\EELD$.
\end{lemma}
\begin{proof}
Let $\fw[G]=(\verts,\edges,\p[\psi])$ be a deformation of $\fw$, and $\Efw[G]$ the associated deformation of $\Efw$. 
Since $uv\in \edges$ and $\p_{u}=\p_{v}$ (by degeneracy), we have that $\p[\psi]_u=\p[\psi]_v$ by \eqref{eq:ELDF}; and therefore that
\[
\p[\psi]_w-\p[\psi]_v= \p[\psi]_w-\p[\psi]_u =\dv(\fw[G])_{uw} (\p_w-\p_u)=\dv(\fw[G])_{uw} (\p_w-\p_v).
\]
It shows that $vw$ is an implicit edge and $\dv(\Efw[G])_{uw}=\dv(\fw[G])_{uw}=\dv(\Efw[G])_{vw}$ proving that it is dependent with $uw$.
\end{proof}

We can reinterpret this result via the equivalence relation $\der$ on $\verts$ defined as the transitive closure of $u\der v$ whenever $uv\in \de$ (it is very important to restrict to the case $uv\in \de$ and not only $\p_u=\p_v$, because if $\p_u=\p_v$ but $uv\notin E$ then we do not necessarily have $\p[\psi]_u=\p[\psi]_v$ in every deformation\linebreak $\fw[G]=(\verts,\edges,\p[\psi])$ of~$\fw$).
We denote by $\dec{u}$ the equivalence class of $u$ under~$\der$. Note that $\der$ induces an equivalence relation on the set of edges too, by setting $uv\der u'v'$ whenever $u\der u'$ and $v\der v'$.
We define the framework $\mathdefn{\fw_{/\der}}\eqdef(\verts_{/\der}, \edges_{/\der}, \p_{/\der})$, where $\p_{/\der}:\verts_{/\der}\to \R^d$ is the map induced on the quotient, which is well defined because $u\der v$ implies that $\p_u = \p_v$. Note that $\ELD[\fw_{/\der}]=\ELD_{/\der}$, where we interpret $\der$ as an equivalence relation on the nodes of $\ELD$ (that is, the edges of $\fw$).

If $\sim$ is an equivalence relation on a set $\nodes$, and $G=(\nodes_{/\sim},\arcs)$ is a graph, we say that the \defn{lift} of $G$ is the graph with nodes $\nodes$ and arcs $xy$ whenever there is an arc between $[x]$ and $[y]$ in $\arcs$.

\begin{proposition}\label{prop:degenerate}
The lift of $\ELD[\fw_{/\der}]$ is a subgraph of $\ELD$.
\end{proposition}
\begin{proof}

Let $uv$ and $wx$ be non-degenerate edges such that $[uv][wx]$ is an arc in $\ELD[\fw_{/\der}]$, we will prove that they are dependent and form an arc in $\ELD$. We distinguish two cases:
\begin{compactitem}
	\item If $[uv]=[wx]$, \ie it is a loop . That means that $uv\der wx'$. Recall that we have $u\der v$ if there is a path of degenerate edges joining $u$ and $v$, and that  $uv\der wx'$ if there are respective paths of degenerate edges $w=u_0u_1\dots u_n=u$ and $x=v_0v_1\dots v_m=v$.
	We will prove that all the $u_iv_j$ are pairwise dependent non-degenerate implicit edges for all $i,j$ by induction on $i+j$. It is true when $i=j=0$ by hypothesis. We prove without loss of generality that the property is preserved when passing from $i$ to $i+1$.	
    Indeed, $u_iu_{i+1}$ is a degenerate edge and $u_iv_j$ a non-degenerate implicit edge. Then, by \Cref{lem:degenerate}, we have that $u_{i+1}v_j$ is a non-degenerate implicit edge dependent with $u_iv_j$. 

	\item If $[uv]\neq[wx]$, this means that there exists a pair of dependent edges $u'v'$ and $w'x'$ with $u'v'\in[uv]$ and $w'x'\in[wx]$. Recall that there is a loop on every node of $\ELD[\fw_{/\der}]$ by definition of the edge dependency graph. This means that there is a loop on $[uv]$ and by the previous point $u'v'$ and $uv$ are dependent. Similarly, we see that $wx$ and $w'x'$ are dependent. Since $w'x'$ and $u'v'$ are dependent, we conclude that $wx$ and $uv$ are dependent by transitivity.     \hfill\qedhere
\end{compactitem}
\end{proof}


\begin{corollary}\label{cor:degeneratecontraction}
$\EDefoCone$ is linearly isomorphic to $\EDefoCone[\fw_{/\der}]$.
\end{corollary}
\begin{proof}
The linear map $\R^{\edges_{/\der}}\to \R^\edges$ that maps $\dv\in \EDefoCone[\fw_{/\der}]$ to the vector $\dv[\mu]\in\R^\edges$ with $\dv[\mu]_e=\dv_{[e]}$ (where $[e]$ is the equivalence class of $e\in \edges$ under $\der$) maps $\EDefoCone[\fw_{/\der}]$ to $\EDefoCone$ because the edges in an equivalence class are pairwise dependent. It is injective and induces a linear isomorphism between the linear spans of both cones.
\end{proof}

\begin{example}
Consider three distinct points $\b x, \b y, \b z\in \R^d$ and construct the framework $\fw = (\verts, \edges, \p)$ with $\verts = \{1, \dots, 6\}$, $\edges = \{12, 23, 34, 45, 56, 16\}$ and $\p_1 = \p_2 = \b x$, and $\p_3 = \p_4 = \b y$, and $\p_5 = \p_6 = \b z$. See the triangle in \Cref{sfig:LemmaProjectionLine} (bottom right), where doubly circled points are those that are labeled twice in $V$.
Then, $\fw_{/\der}$ is the triangle with $V_{/\der} = \bigl\{[1,2], [4,5], [5,6]\bigr\}$, and $E_{/\der} = \bigl\{[1,2][3,4],\, [3,4][5,6],\, [1,2][5,6]\bigr\}$, and $\p_{/\der}([1,2]) = \b x$, and $\p_{/\der}([3,4]) = \b y$, and $\p_{/\der}([5,6]) = \b z$.
As $\fw_{/\der}$ is a triangle, it is indecomposable and $\ELD[\fw_{/\der}]$ is the complete graph (with loops).
Thus, $\ELD$ is a complete graph (with loops), and $\fw$ is indecomposable, according to \Cref{prop:degenerate,cor:degeneratecontraction}.
\end{example}

\subsubsection{Projections}

The power (and naturality) of degenerate edges becomes apparent when we consider projections of frameworks.
The following geometric property is at the core of many properties that will serve as the main building blocks of our future constructions.

\begin{proposition}\label{prop:projection1}
	Let $\fw=(\verts,\edges,\p)$ be a framework in $\R^d$, and $\pi:\R^d \to \R^k$ an affine map. Then $\mathdefn{\pi(\fw)}\eqdef (\verts,\edges,\pi\circ \p)$ is a framework that fulfills
	\begin{compactenum}[(i)]
		\item\label{it:picone} $\EDefoCone[\fw]\subseteq \EDefoCone[\pi(\fw)] + \R_+^{\nde(\fw)\ssm \nde(\pi(\fw))}$, and 
		\item\label{it:pigraph} $\ELD[\pi(\fw)]$ is a subgraph of $\ELD$.
	\end{compactenum}
\end{proposition}

\begin{proof}
For \eqref{it:picone}, recall that $\EDefoCone[\pi(\fw)] $ is the cone cut by the cycle equations of $\pi(\fw)$ together with the equations $\dv_e=0$ for all degenerate edges $e\in\de(\pi(\fw))$. 
The Minkowski sum $\EDefoCone[\pi(\fw)] + \R_+^{\nde(\fw)\ssm \nde(\pi(\fw))}$ 
contains the non-negative vectors $\dv$ for which there is some $\dv'\in \EDefoCone[\pi(\fw)]$ with $\dv_e=\dv'_e$ for $e\in \nde(\pi(\fw))$.
Since the variables corresponding to degenerate edges do not appear in cycle equations (they have a $0$ coefficient), we have that $\EDefoCone[\pi(\fw)] + \R_+^{\nde(\fw)\ssm \nde(\pi(\fw))}$ is the cone defined by the cycle equations of $\pi(\fw)$ (geometrically, it is isomorphic to the Cartesian product of $\EDefoCone$ with a positive orthant). 

Now, by linearity, each cycle equation \eqref{eq:cycle_equation} for $\fw$
\[\dv_{u_1u_2}(\p_{u_2}-\p_{u_1}) +\dv_{u_2u_3}(\p_{u_3}-\p_{u_2}) +\dots+\dv_{u_ku_1}(\p_{u_1}-\p_{u_k})=\b 0\]
implies the corresponding cycle equation for $\pi(\fw)$
\[\dv_{u_1u_2}(\pi(\p_{u_2})-\pi(\p_{u_1})) +\dv_{u_2u_3}(\pi(\p_{u_3})-\pi(\p_{u_2})) +\dots+\dv_{u_ku_1}(\pi(\p_{u_1})-\pi(\p_{u_k}))=\b 0\]
Therefore, any vector in $\EDefoCone$ verifies the cycle equations of $\pi(\fw)$, giving the desired inclusion.

For \eqref{it:pigraph}, note that if $e,f\in \nde(\pi(\fw))$, then if $\dv_e=\dv_f$ for any $\dv\in \EDefoCone[\pi(\fw)]$, then the same is true for any $\dv\in \EDefoCone$, by the inclusion above. Our claim follows.
\end{proof}

\begin{remark}
	Note that while $\ELD[\pi(\fw)]$ is a subgraph of $\ELD$, it is in general not true that  $\EELD[\pi(\fw)]$ is a subgraph of $\EELD$. There could be implicit edges of $\pi(\fw)$ that are not implicit edges of $\fw$.
\end{remark}


Of course, deformation cones are invariant under affine isomorphisms.

\begin{lemma}\label{lem:affineisomorphism}
	If $\pi:\R^d\to \R^d$ is an affine isomorphism, then $\EDefoCone = \EDefoCone[\pi(\fw)]$ (and hence also $\ELD=\ELD[\pi(\fw)]$).
\end{lemma}
\begin{proof}
The affine change of coordinates replaces the system of cycle equations by an equivalent one, obtained via left-multiplication by an invertible matrix, leaving the solution space unchanged.
\end{proof}

\subsubsection{Computing the deformation cone of a polytope is a 2-dimensional problem}
\Cref{lem:affineisomorphism} allows us to reduce the dimension of the frameworks we consider.

\begin{proposition}\label{lem:cycle_preserving_projection}
	Let $\fw=(\verts,\edges,\p)$ be a framework, $C_1,\dots,C_k$ be a collection of cycles generating the cycle space of $G=(\verts,\edges)$, and $\pi:\R^d\to \R^e$ an affine map. If for every $1\leq i\leq k$, the restriction of $\pi$ to the affine span of $\{\p_v\}_{v\in C_i}$ is injective, then $\EDefoCone = \EDefoCone[\pi(\fw)]$ (and hence also $\ELD=\ELD[\pi(\fw)]$).
\end{proposition}
\begin{proof}
Follows directly from combining \Cref{lem:cyclebasis,lem:affineisomorphism}.
\end{proof}

In particular, this shows that the deformation cone of any polytope reduces to the deformation cone of a $2$-dimensional framework.

\begin{corollary}\label{cor:polytopeto2dimframework}
	If $\pol\subseteq \R^d$ is a polytope and $\pi:\R^d\to \R^e$ an affine map that is injective on any $2$-face of~$\pol$, then $\EDefoCone[\pol] = \EDefoCone[\pi(\fw(\pol))]$ (and hence also $\ELD=\ELD[\pi(\fw(\pol))]$).
\end{corollary}
\begin{proof}
Follows from the fact that the cycle space of $\fw(\pol)$ is generated by the cycles induced by the $2$-faces of~$\pol$~\cite{PinedaVillavicencio2022}, c.f.~\Cref{rmk:cyclebasispol}.
\end{proof}

\begin{remark}
For any polytope $\pol\subseteq \R^d$, a generic projection $\pi : \R^d \to \R^2$ is injective on any 2-face of~$\pol$.
Hence, $\DefoCone[\pol] = \DefoCone[\pi(\fw(\pol))]$ where $\pi(\fw(\pol))$ is 2-dimensional, justifying the title of this subsection.
Note that this is not the case in general for frameworks:
we will see in \Cref{cor:RigidCycle0} that a framework composed of a cycle on affinely independent points is indecomposable, yet its 2-dimensional projections are decomposable.

In the case of a $3$-dimensional polytope $\pol$, we can go a step further: the dimension of $\DefoCone[\pol]$, and thus the indecomposability of $\pol$, can be computed by looking at a \emph{planar} framework, \ie 2-dimensional frameworks with no crossing edges.
Indeed, according to  Kallay \cite{Kallay-DimDCprojectiveInvariant}, the dimension of the deformation cone of a polytope is invariant under projective transformation.
Hence, one can mimic the Schlegel projection: find a 3-dimensional projective transformation such that the image of $\pol$ projects linearly and  without crossing onto the plane spanned by one of its 2-face.
\end{remark}

\begin{remark}
It is sometimes beneficial, in particular for implementation purposes, to embed deformation cones of polytopes in larger spaces. Namely, consider a fixed polytope $\pol$ with  $2$-dimensional faces $\polF_1, \dots, \polF_r$. For each $i\in [r]$, define variables 
$\ell_{\polF, \pol[e]}$ for each $2$-face of $\polF$ of $\pol$ and each edge $\pol[e]\in\edges(\polF)$. Compute the cycle equations (\ie polygonal face equations) for each $2$-face $\polF$ on the variables $\ell_{\polF, \pol[e]}$, and finally add the equation $\ell_{\polF, \pol[e]} = \ell_{\pol[G], \pol[e]}$ whenever $\polF,\pol[G]$ are two $2$-faces intersecting on $\pol[e]$, and intersect with the positive orthant. This is the approach used in \cite{Poullot-JohnsonSolids} to compute the dimension of the deformation cones of Platonic, Archimedean and Johnson solids.
\end{remark}

\subsubsection{Affinely independent cycles of vertices}
One of the strengths of \Cref{prop:projection1} is when it is used in combination with \Cref{prop:degenerate}.
Indeed, by choosing appropriate projection maps $\pi$, we can control the degenerate edges of $\pi(\fw)$.
For a vector subspace $W\subseteq \R^d$, let $\mathdefn{\edges_W}\eqdef\set{uv\in\edges}{\p_v-\p_u\in W}$.
We define an equivalence relation \defn{$\subr$} on $\verts$ as the transitive closure of $u\subr v$ if $uv\in \edges_W$.
Denote by \defn{$\subc{u}$} the equivalence class of $u$ under $\subr$.
The relation $\subr$ induces an equivalence relation on $\edges\ssm \edges_W$ too by setting $uv\subr u'v'$ if $u\subr u'$ and $v\subr v'$.
Let \defn{$\pi_W$} be a linear map whose kernel is $W$, and $\mathdefn{\fw_{/W}}\eqdef \pi_W(\fw)$ (different choices of $\pi_W$ give affinely equivalent frameworks, which have the isomorphic deformation cones by \Cref{lem:affineisomorphism}).
Then the edges in $\edges_W$ are precisely the degenerate edges of $\fw_{/W}$, and $\subr$ is the equivalence relation $\der$ on this framework.
Thus:

\begin{proposition}\label{prop:projection}
The lift of $\ELD[\fw_{/W}]$ is a subgraph of $\ELD$.
\end{proposition}

\begin{proof}
Direct application of \Cref{prop:degenerate,prop:projection1}.
\end{proof}

We want to highlight the case when the projection is indecomposable, because it shows how one can exploit parallelisms and other coplanarities to deduce edge dependencies, see \Cref{fig:ProjectionLemma}. This tool was not present in previous approaches, and will become fundamental for proving indecomposability of the constructions in \Cref{sec:newrays,sec:edgedecomposable,sec:constructions}.

\begin{corollary}\label{cor:subspaceprojection}
	Let $W$ be a vector subspace of $\R^d$. If $\fw_{/W}$ is indecomposable, then the edges in $\nde\ssm \edges_{W}$ are pairwise dependent.
\end{corollary}
\begin{proof}	
Here, $\ELD[\fw_{/W}]$ is a complete graph on $(\nde\ssm \edges_{W})_{/\subr}$, and its lift is a complete graph on $\nde\ssm \edges_{W}$.\end{proof}


\begin{figure}
	\centering
	\begin{subfigure}[b]{0.45\linewidth}
		\centering
		\includegraphics[width=\textwidth]{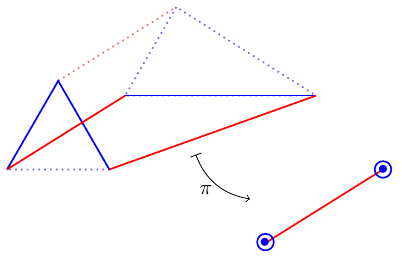}
		\caption{Projection parallel to a plane}
		\label{sfig:LemmaProjectionPlane}
	\end{subfigure}\hfill
	\begin{subfigure}[b]{0.45\linewidth}
		\centering
		\includegraphics[width=\textwidth]{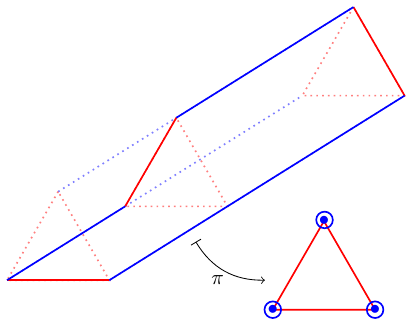}
		\caption{Projection parallel to a line}
		\label{sfig:LemmaProjectionLine}
	\end{subfigure}
	\caption[Projections help finding dependent edges]{
    (Left) Projecting the framework on top, parallel to the plane spanned by its \textcolor{blue}{blue} edges, gives rise to a segment with 2 degenerate edges. Consequently, the 2 \textcolor{red}{red} edges are dependent because they are projected to the same edge (while the \textcolor{blue}{blue} edges are not dependent).
    (Right) Projecting the framework on top, parallel to the direction spanned by its \textcolor{blue}{blue} edges, give rise to a triangle with 3 degenerate edges and degenerate \textcolor{blue}{blue} vertices. Consequently, the 3 \textcolor{red}{red} edges are dependent (while the \textcolor{blue}{blue} edges are not).}
	\label{fig:ProjectionLemma}
\end{figure}

\Cref{cor:subspaceprojection} can be used to show that the edges in a cycle whose vertices are affinely independent form a clique in $\ELD$, generalizing \Cref{ex:triangle} to higher dimensions. The following is a reinterpretation in our language of \cite[Prop.~3]{PrzeslawskiYost2016} (which generalizes the ideas of the proofs of \cite[Thm.~9~\&~10]{Kallay1982} to higher dimensions). 

\begin{corollary}[{\cite[Prop.~3]{PrzeslawskiYost2016}}]\label{cor:RigidCycle0}
	Let $\fw=(\verts,\edges,\p)$ be a framework containing a cycle with vertices $u_1,\dots, u_k\in \verts$ such that $u_iu_{i+1}\in \edges$ for all $1\leq i \leq k$ (with indices taken modulo $k$, implying $u_{k+1}=u_1$). If the points $\p_{u_{1}},\dots \p_{u_{k}}$ are affinely independent, then the edges $\set{u_iu_{i+1}}{1\leq i\leq k}$ are pairwise dependent (\ie the nodes $u_iu_{i+1}$ of $\ELD$ with $1\leq i\leq k$ form a clique).
\end{corollary}
\begin{proof}
	The proof is by induction on $k$. If $k=1$, then there is a single edge and nothing to prove.
	Assume now that $k\geq 2$, and consider the subspace $W$ spanned by $\p_{u_2}-\p_{u_1}$. We have that $\fw_{/W}$ is a cycle of affinely independent vertices. By induction hypothesis, it is indecomposable, and by \Cref{cor:subspaceprojection}, the edges $\set{u_iu_{i+1}}{2\leq i\leq k}$ are pairwise dependent in $\ELD$. We repeat the same argument with the subspace spanned by $\p_{u_k}-\p_{u_1}$ to conclude that $\set{u_iu_{i+1}}{1\leq i\leq k}$ are pairwise dependent.		
\end{proof}

Kallay proved the following result that we reformulate below  (and used it for constructing \Cref{exm:KallayConditionallyDecomposable}). Using implicit edges, it becomes a simple corollary of the result above. As the reader will easily see from the proof, these statements can be easily generalized to arbitrarily large cycles (or with stronger results as the upcoming \Cref{cor:RigidCycle}). We do not explicit them because we see them as consequences of applying our more powerful theorems as \Cref{thm:mainthmframeworks} on the graph of implicit edge dependencies.
\begin{corollary}[{\cite[Thm. 9 \& 10]{Kallay1982}}]
	Let $\fw=(\verts,\edges,\p)$ be a framework. If one of the following holds, then $\fw$ is indecomposable:
	\begin{compactenum}
		\item[(i)] $\fw$ is the union of three indecomposable frameworks $\fw_1=(\verts_1,\edges_1,\p_{|\verts_1})$, $\fw_2=(\verts_2,\edges_2,\p_{|\verts_2})$ and $\fw_3=(\verts_3,\edges_3,\p_{|\verts_3})$ such that there are vertices $u_{12}, u_{13}, u_{23}$ with $u_{ij}\in \verts_i\cap \verts_j$ such that $\p_{u_{12}}\p_{u_{13}}\p_{u_{23}}$ are not collinear.
		\item[(ii)]
		$\fw$ is the union of two indecomposable frameworks $\fw_1=(\verts_1,\edges_1,\p_{|\verts_1})$, $\fw_2=(\verts_2,\edges_2,\p_{|\verts_2})$ and two disjoint edges $u_1u_2$ and $v_1v_2$ with $u_1,v_1\in\verts_1$, $u_2,v_2\in\verts_2$ such that $\p_{u_{1}}\p_{u_{2}}\p_{v_{1}}\p_{v_{2}}$ are not coplanar.
	\end{compactenum}
\end{corollary}

\begin{proof}
	$(i)$ As for $i=1,2,3$ we have that $\fw_i$ is indecomposable, we know that any pair of vertices of $\fw_i$ is an implicit edge, and that these are all pairwise dependent. In particular, $u_{12}u_{13}u_{23}$ form a cycle of implicit edges. Since they are not collinear, we can apply \Cref{cor:RigidCycle0} and we derive that the three implicit edges of the cycle are pairwise dependent. Since each of the $\fw_i$ contains one of these implicit edges, we conclude that all the implicit edges are pairwise dependent, and hence that $\fw$ is indecomposable.
	
	$(ii)$ As for $i=1,2$ we have that $\fw_i$ is indecomposable, all the pairs of vertices of $\verts_i$ are implicit edges that are pairwise dependent.
	Hence $u_1, u_2, v_2, v_1$ is a cycle of implicit edges whose vertices are affinely independent.
	By \Cref{cor:RigidCycle0}, the four implicit edges of this cycle form a clique in $\EELD$. Since each $\fw_i$ contains one of the implicit edges of the cycle, we conclude that all the implicit edges are pairwise dependent, and hence that $\fw$ is indecomposable.
\end{proof}

Note that \Cref{cor:subspaceprojection} provides another proof of the dependencies in trapezoids, see \Cref{ex:parallelogram}. 
This can be generalized as a combination of \Cref{cor:subspaceprojection,cor:RigidCycle0}. While the statement is quite verbose in contrast with the simpler (and stronger) \Cref{cor:subspaceprojection}, it has the advantage of being an explicit property straightforward to check.
We denote here $\mathdefn{[k]} \coloneqq \{1, \dots, k\}$.

\begin{corollary}\label{cor:RigidCycle}
	Let $\fw=(\verts,\edges,\p)$ be a framework containing a cycle with vertices $u_1,\dots, u_k\in \verts$ such that $u_iu_{i+1}\in \edges$ for all $1\leq i \leq k$ (with indices taken modulo $k$, implying $u_{k+1}=u_1$). For a subset $S\subseteq [k]$, define the linear span $W_S\eqdef\lin\set{\p_{u_{i+1}}-\p_{u_{i}}}{i\in S}$. 
	
	If there is a subset $S\subseteq [k]$ such that $k-|S|=\dim(W_{[k]}) - \dim(W_{S})+1$, then the edges $\set{u_iu_{i+1}}{i\notin S}$ are pairwise dependent (\ie the nodes $u_iu_{i+1}$ of $\ELD$ with $i\in [k]\ssm S$ form a clique).
\end{corollary}

\begin{proof}
It derives from \Cref{cor:subspaceprojection,cor:RigidCycle0}. Indeed, $\fw_{/W_{S}}$ is a $(\dim(W_{[k]}) - \dim(W_{S}))$-dimensional framework. If $k-|S|=\dim(W_{[k]}) - \dim(W_{S})+1$, this means that $\fw_{/W_{S}}$ is a cycle supported on affinely independent vertices, and hence indecomposable. The lift of $\ELD[\fw_{/W_{S}}]$ gives the desired clique.
\end{proof}

\subsubsection{Cartesian products}

Using projections, we can express the deformation cone of a product of frameworks (or of polytopes) as the product of the deformation cones of each factor.
This result seems to be absent from the current literature. 
We give here a direct proof.

The \defn{Cartesian product} of two frameworks $\fw = (\verts,\edges,\p)$ with $\p_v\in\R^d$ and $\fw[G] = (W, F, \p[\psi])$ with $\p[\psi]_{v'}\in\R^e$ if the framework $\mathdefn{\fw \times \fw[G]}$ whose vertex set is $\verts\times W$, edge set is $\bigl\{(u, x)(v, y) ~;~ uv\in \edges \text{ or } xy\in F\bigr\}$, and realization is $\p_{(v, x)} = (\p_v, \p[\psi]_x) \in \R^{d + e}$.
We abuse notation and denote $\fw\times\p[\psi]_x$ for the Cartesian product of $\fw$ with the framework consisting of a unique vertex $x$ of realization $\p[\psi]_x$.

For instance, the Cartesian product of two polytopes $\pol\subseteq \R^d$ and $\pol[Q]\subseteq \R^e$ is the polytope denoted $\pol\times\pol[Q]=\set{(\b p, \b q)\in\R^{d+e}}{\b p\in \pol, \b q\in \pol[Q]}\subseteq \R^{d+e}$. Its faces are of the form $\pol[F]\times\pol[G]$ where $\pol[F], \pol[G]$ are faces of $\pol, \pol[Q]$.
	
\begin{theorem}\label{thm:products}
For two connected frameworks $\fw$ and $\fw[G]$,
all deformations of a Cartesian product $\fw\times\fw[G]$ are of the form $\fw'\times\fw[G]'$, where $\fw'$ and $\fw[G]'$ are deformations of $\fw$ and $\fw[G]$ respectively. 
Consequently, the deformation cone of $\fw\times\fw[G]$ is the Cartesian product of the deformation cones of $\fw$ and $\fw[G]$:
$$\EDefoCone[\fw\times\fw[G]] = \EDefoCone[\fw]\times \EDefoCone[\fw[G]]$$
\end{theorem}

\begin{proof}
Let $uv$ be an edge of $\fw$ and $x, y$ two vertices of $\fw[G]$.
Consider $W = \{\b 0\}\oplus\R^e$ (\ie $\pi_W$ projects out~$\fw[G]$).
The two edges $(u, x)(v, x)$ and $(u, y)(v, y)$ of $\fw\times\fw[G]$  have the same image by the projection $\pi_W$.
As $\fw[G]$ is connected, in the framework of $(\fw\times\fw[G])_{/W}$, there is a degenerate edge $\pi\bigl((u,x)\bigr)\pi\bigl((u,y)\bigr)$, and a degenerate edge $\pi\bigl((v,x)\bigr)\pi\bigl((v,y)\bigr)$.
Thus, the edges $\pi\bigl((u,x)\bigr)\pi\bigl((v,x)\bigr)$ and $\pi\bigl((u,y)\bigr)\pi\bigl((v,y)\bigr)$ are dependent in $(\fw\times\fw[G])_{/W}$.
By \Cref{prop:projection} we get that $(u, x)(v, x)$ and $(u, y)(v, y)$ are dependent in $\fw\times\fw[G]$.

Hence, in a deformation $\fw[H]$ of $\fw\times\fw[G]$, all the sub-frameworks corresponding to $\fw\times\p[\psi]_x$ with $x\in \fw[G]$ are translations of the same deformation $\fw'$ of $\fw$, because they have the same edge lengths.
Similarly, all the subframeworks of the form $\p_v\times \fw[G]$ with $v\in\fw$ are translations of the same deformation $\fw[G]'$ of $\fw[G]$.
Since $\fw[H]$ and $\fw'\times \fw[G]'$ have the same edge lengths, they must coincide up to translation.
\end{proof}

\begin{corollary}\label{cor:ProductIndecomposables}
Let $\fw_1, \dots, \fw_r$ be indecomposable frameworks.
Then ${\fw_1\times\dots\times\fw_r}$ is uniquely decomposable and
 $\DefoCone[{\fw_1\times\dots\times\fw_r}]$ is a simplicial cone of dimension $r$ whose $r$ rays are the deformation cones of $\fw_1, \dots, \fw_r$. In particular, all the deformations of $\fw_1\times\dots\times\fw_r$ are of the form $(\b\lambda_1\fw_1)\times\dots\times(\b\lambda_r\fw_r)$ for some $\b\lambda_1, \dots, \b\lambda_r \geq 0$.
\end{corollary}

\begin{proof}
As each $\fw_i$ is indecomposable, it is in particular connected, and $\DefoCone[\fw_i]$ is a $1$-dimensional ray.
The claim follows from \Cref{thm:products}, because the product of $r$ rays is a $r$-dimensional simplicial cone.
\end{proof}

\begin{example}
By \Cref{cor:ProductIndecomposables}, the deformation cone of a product of $r$ simplices is a simplicial cone of dimension $r$.
In \cite{CDGRY2020}, the authors show a stronger version: the deformation cone of any polytope which is \emph{combinatorially equivalent} to a product of $r$ simplices is a simplicial cone of dimension $r$.
\end{example}

\subsection{Dependent sets of vertices}
\subsubsection{Dependent sets of vertices}
Let $\fw=(\verts,\edges,\p)$ be a framework. We say that a subset $S\subseteq\verts$ of vertices is \defn{dependent} if every pair of vertices of $S$ is an implicit edge, and the non-degenerate ones are all pairwise dependent in $\EELD$. We extend naturally the notation and say that a subset $S\subseteq\verts(\pol)$ of vertices of a polytope~$\pol$ is \defn{dependent} if they are as vertices of the framework~$\fw(\pol)$.

Let us highlight that, despite our presentation, this dependence is not arising from an equivalence relation on vertices (as it could be that both $uv$  and $uw$ form implicit edges, but that these are not dependent, like in a parallelogram), but from an equivalence relation on \textbf{pairs} of vertices. We use this notation nevertheless because it simplifies our presentation.

Note first that, in dimension at least~$2$, we can weaken the definition of dependent sets of vertices, removing the assumption that the implicit edges are all pairwise dependent.
It follows from \Cref{lem:Dim2AllImplicitImpliesDependent}:

\begin{corollary}\label{cor:Dim2DependentVerticesDependentEdges}
A subset of vertices $S$ with $\p(S)$ of dimension at least~$2$ is dependent if and only if every pair of vertices of $S$ is an implicit edge.
\end{corollary}

We can rephrase \Cref{lem:ELDindecomposable} in the language of dependent sets of vertices:

\begin{lemma}\label{cor:DependentVerticesIndecomposable}
Let $\fw=(\verts,\edges,\p)$ be a framework. Then $\ELD$ is a complete graph and $\fw$ is indecomposable if and only if $\verts$ is a dependent set of vertices.
\end{lemma}

\begin{proof}
Combine \Cref{lem:ELDindecomposable,cor:extendedgraph} (iii).
\end{proof}



For the applications, the condition defining vertex dependency can be relaxed using \Cref{prop:connectedimplicit}:

\begin{lemma}\label{lem:connecteddependentrigid}
Let $\fw=(\verts,\edges,\p)$.  
A subset of vertices $S\subseteq\verts$ with $\p(S)$ of dimension at least~$2$ is dependent if and only if for every $u,v\in S$ there is a path of pairwise dependent implicit edges connecting $u$ and $v$.
\end{lemma}

\begin{proof}
By \Cref{prop:connectedimplicit}, if $u,v\in S$ are connected by a path of pairwise dependent edges, then $uv$ is an implicit edge.
Hence, the subframework of $\Efw$ spanned by~$S$ is complete.
By \Cref{cor:indeccompleteclosure}, this means that this subframework is indecomposable, and hence that all its implicit edges are pairwise dependent.
\end{proof}

\subsubsection{Recovering Shephard's criteria}
The combination of \Cref{cor:DependentVerticesIndecomposable} and \Cref{lem:connecteddependentrigid} directly gives our first new indecomposability criterion. 
It relies on proving that certain subsets of edges are dependent,
as many previous indecomposability criteria, but its scope is more general, because dependencies can not only be deduced via 
indecomposable subframeworks (such as triangles), but we can also exploit other structures such as parallelisms to deduce dependencies, for example via \Cref{cor:subspaceprojection}.

Even though this result will be subsumed by the stronger, but harder to state, \Cref{thm:mainthm}, the following simpler version is already enough for our applications in \Cref{sec:newrays}, where the previously known methods did not work.


%
%

\begin{theorem}\label{thm:mainthmframeworks}
	Let $\fw=(\verts,\edges,\p)$ be a framework of dimension $\geq2$. 
	If any pair of vertices of $\verts$ is connected through a path of pairwise dependent implicit edges, then $\ELD$ is a complete graph and $\fw$ is indecomposable.
\end{theorem}

\begin{proof}
By \Cref{prop:connectedimplicit}, any pair in $\verts$ is an implicit edge.
By \Cref{cor:indeccompleteclosure}, $\ELD$ is complete.
\end{proof}

%


As a first application, we use \Cref{thm:mainthmframeworks} to reprove one the first known indecomposability criteria for polytopes, given by Shephard in 1963 \cite{Shephard1963}. 
There are two slightly different versions of Shephard's indecomposability theorem. The original version~\cite[Thm.~12]{Shephard1963} is equivalent to \Cref{thm:Shephard}, whereas it is cited by McMullen in \cite[Thm.~1]{McMullen1987} with a slightly stronger statement (that we give in \Cref{cor:McMullen1}). 
We could not find the argument proving the stronger version anywhere in the literature, but we will see that it follows from \Cref{thm:mainthmframeworks}.
In our setup, the nuance is that in the first version, we will connect all the vertices through paths of edges that are all dependent with a given edge $\pol[e]$; whereas in the second version, the edges of each of the paths are pairwise dependent, but different paths could belong to different equivalence classes of the dependency relation. This subtle difference the reason why \Cref{cor:indeccompleteclosure} fails in dimension~$1$, as explained above its statement. 

Using McMullen's notation from \cite{McMullen1987}, we define a \defn{strong chain} of faces of a polytope~$\pol$ as a sequence of faces $\pol[F]_0,\dots,\pol[F]_k$ such that $\dim(\pol[F]_{i}\cap\pol[F]_{i-1})\geq 1$ for $1\leq i\leq k$. Any vertex or edge of $\pol[F]_0$ is said to be connected to any vertex or edge of $\pol[F]_i$ with $1\leq i\leq k$.

\begin{corollary}[{\cite[Thm. 12]{Shephard1963}}]\label{thm:Shephard}
	Let $\pol$ be a polytope. If there exist an edge $\pol[e]$ to which each vertex of~$\pol$ is connected by a strong chain of indecomposable faces, then $\pol$ is itself indecomposable.
\end{corollary}

\begin{proof}
	Consider the connected component $X$ of $\pol[e]$ in $\ELD[\pol]$. For every vertex, we consider a strong chain $\pol[F]_0,\dots,\pol[F]_k$ connecting it to $\pol[e]$. By \Cref{cor:indecsubframeclique}, all the edges in the $1$-skeleton of each $\pol[F]_i$ form a clique in $\ELD[\pol]$. As  $\pol[F]_{i-1}$ and $\pol[F]_i$ intersect on at least one edge, these cliques belong to~$X$. Moreover, since the $1$-skeletons of each $\pol[F]_i$ are all connected and intersect, each vertex of $\pol$ is connected to both endpoints of $\pol[e]$ by edges in $X$. Therefore all pairs of vertices of $\pol$ are connected together by paths of dependent edges, and we can apply \Cref{thm:mainthmframeworks} to deduce that $\pol$ is indecomposable.
\end{proof}

We can also prove the slightly stronger version of Shephard's criterion stated in \cite[Thm.~1]{McMullen1987}.

\begin{corollary}[{\cite[Thm. 12]{Shephard1963}\cite[Thm. 1]{McMullen1987}}]\label{cor:McMullen1}
	Let $\pol$ be a polytope. If any pair of vertices of $\pol$ can be connected by a strong chain of indecomposable faces, then $\pol$ is itself indecomposable.
\end{corollary}

\begin{proof}
	If $\b p,\b q\in\verts(\pol)$ are connected by a strong chain of indecomposable faces, then they are connected by a dependent set of edges, by \Cref{cor:indecsubframeclique}. We conclude by \Cref{thm:mainthmframeworks}.
\end{proof}

Finally, we advertise the following corollary of \Cref{thm:mainthmframeworks}.
For polytopes, it also follows from Shephard's criterion, but it can be easily extended to frameworks with some structural assumptions.

\begin{corollary}\label{cor:VertexAdjacentToAll}
If $\fw = (\verts, \edges, \p)$ is a framework with a vertex $v\in \verts$ such that (i) $v$ is connected to every vertex of $\fw$ with an edge, (ii) $v$ does not belong to the line spanned by any two other vertices of $\fw$, and (iii) the subframework $\fw\ssm v$ obtained by deleting $v$ from $\fw$ is connected, then $\fw$ is indecomposable.
\end{corollary}

\begin{remark}\label{rmk:VertexAdjacentToAll}
\Cref{cor:VertexAdjacentToAll} implies that, for a polytope $\pol$, if a vertex $\b v$ of $\pol$ is connected to every other vertex, then $\pol$ is indecomposable: if $\dim \pol = 1$, it is obvious; if $\dim \pol \geq 2$, then we can apply the corollary because the 1-skeleton of $\pol$ is $2$-connected.
\end{remark}

\begin{proof}[Proof of \Cref{cor:VertexAdjacentToAll}]
Let $xy$ be an edge of $\fw\ssm v$, then as $vx$ and $vy$ are edges of $\fw$, \Cref{ex:triangle} applied to the triangle $vxy$ ensures that the edges $vx$, $vy$ and $xy$ are dependent.
Moreover, if $xy$ and $xz$ are edges of $\fw$, then $xy$ and $xz$ are both dependent with $vx$, so they are dependent to each other.
In particular, $\ELD$ contains the line graph of the graph $\bigl(\verts\ssm\{v\},\, \edges\ssm\{vx~;~x\in \verts\}\bigr)$.
As $\fw\ssm v$ is connected, so is its line graph.
Finally, the nodes $vx$ are connected to the nodes of this line graph in $\ELD$, so the set $\bigl\{vx ~;~ x\in \verts\ssm\{v\}\bigr\}$ is a set of dependent edges.

By \Cref{thm:mainthmframeworks}, as all the vertices $x, y\in \verts\ssm\{v\}$ are connected through a path of dependent edges (namely $(vx, vy)$), the framework $\fw$ is indecomposable.
\end{proof}

\begin{example}\label{ex:hyperorder}
We thank Federico Castillo for bringing to our consideration the following example.

As per usual, we denote $\b e_{X} = \sum_{j\in X} \b e_j$, especially $\b e_{[i,\, n]} = \b e_i + \b e_{i+1} + \dots + \b e_{n-1} + \b e_n$.

For $k \leq n$, let $\pol_{n, k} = \left\{\b x\in \R^n ~;~ 0\leq x_1\leq \dots \leq x_n\leq 1 \text{ and } \sum_{i=1}^n x_i = k\right\}$, which is the intersection of the $(n, k)$-hypersimplex $\left\{\b x\in \R^n ~;~ 0\leq x_i\leq 1 \text{ for all } i\in[n] \text{, and } \sum_{i=1}^n x_i = k\right\} = \conv(\b e_X ~;~ X\subseteq[n],\, |X| = k)$ with the $n$-order cone $\{\b x\in \R^n ~;~ x_1\leq \dots \leq x_n\}$.
Equivalently, the polytope $\pol_{n, k}$ is the intersection of the $n$-order simplex $\pol[O]_n \coloneqq \{\b x\in \R^n ~;~ 0\leq x_1\leq \dots \leq x_n\leq 1\} = \conv(\b e_{[i,\, n]}~;~ i\in [n])$ with the hyperplane $\c H_k \coloneqq \left\{\b x\in \R^n ~;~ \sum_{i=1}^n x_i = k\right\}$.

Note that the point $\b e_{[n-k+1,\, n]}$ is a vertex of $\pol[O]_n$ which lies in $\c H_k$.
It is the only face of $\pol[O]_n$ that is contained in $\c H_k$.
Hence, except for $\b e_{[n-k+1,\, n]}$ itself, the vertices of $\pol_{n, k} = \pol[O]_n\cap \c H_k$ are the non-empty intersections of the edges of $\pol[O]_n$ with $\c H_k$; and the edges of $\pol_{n, k}$ are obtained as the intersection of $2$-faces of $\pol[O]_n$ with $\c H_k$.
In particular, as $\pol[O]_n$ is a simplex, all its edges either contain $\b e_{[n-k+1,\, n]}$, or share a $2$-face with $\b e_{[n-k+1,\, n]}$.

This ensures that $\b e_{[n-k+1,\, n]}$ is a vertex of $\pol_{n, k}$ that shares an edge with all the other vertices of $\pol_{n, k}$.
By \Cref{cor:VertexAdjacentToAll} (especially by \Cref{rmk:VertexAdjacentToAll}), the polytope $\pol_{n, k}$ is indecomposable.

We do not detail a proof here, but from $\pol_{n, k} = \pol[O]_n\cap \c H_k$, it follows that the vertices of $\pol_{n, k}$ are $(0, \dots, 0, \frac{k-j}{i}, \dots, \frac{k-j}{i}, 1, \dots, 1)$ with $i$ copies of $\frac{k-j}{i}$ and $j$ copies of $1$, satisfying: $j \leq k$ and $i \leq n-j$.

The polytope $\pol_{n, k}$ is an instance of a \emph{dominant weight polytope} for the type $A$ root system, presented in \cite{BurrullGuiHu2024StronglyDominantWeightPolytopesAreCubes}.
For disambiguation:
In type $A$, the cone $\overline{C_+}$ of \cite[Section 1]{BurrullGuiHu2024StronglyDominantWeightPolytopesAreCubes} is the $n$-order cone presented here, and the Weyl group $W_f$ is the group of permutations of $[n]$.
Hence for the dominant weight $\b\lambda = (0, \dots, 0, 1, \dots, 1)$ with $k$ occurrences of $1$s, the polytope $\conv (W_f\b\lambda)$ is the $(n, k)$-hypersimplex, and the polytope $\pol_{n ,k}$ is the dominant weight polytope $\pol^{\b\lambda} = \conv(W_f\b\lambda)\cap \overline{C_+}$.
(Note that, for $\b\lambda=(1,2,\dots,n)$, the polytope $\pol^{\b\lambda}$ is one of the \emph{poset permutahedra} from \cite{BlackRehbergSanyal2025-PosetPermutahedra}.)
In \cite[Proposition 2.13]{BurrullGuiHu2024StronglyDominantWeightPolytopesAreCubes}, the authors show that $\pol^{\b\lambda+\mu} = \pol^{\b\lambda} + \pol^{\mu}$, moreover, in their section, they more-or-less explain that $\{\pol^\mu ~;~ \mu\in \overline{C_+}\}$ is the deformation cone of $\pol^{\b\lambda}$ for any $\b\lambda\in C_+ = \interior\overline{C_+}$.
Hence, $\pol_{n, k}$ is a ray of a deformation cone because $(0, \dots, 0, 1, \dots, 1)$ is a ray of $\overline{C_+}$: this also shows that $\pol_{n, k}$ is indecomposable. 
\end{example}

\subsubsection{The dimension of the deformation cone}

Finally, note that \Cref{thm:mainthmframeworks} can be improved to an upper bound on the dimension of~$\DefoCone[\fw]$ (indecomposable frameworks are those whose deformation cone is $1$-dimensional). We use ideas coming from dependent set of vertices to relax the conditions for our previous upper bound for $\dim(\DefoCone[\fw])$ from \Cref{thm:dim1}. 

\begin{theorem}\label{thm:DCdimBoundedByNumberOfDependentSets}
	If $X_1, \dots, X_r$ are disjoint dependent sets of implicit edges such that any two vertices of~$\fw$ are connected through a path of implicit edges in $\bigsqcup_i X_i$, then $\dim\DefoCone[\fw]\leq r$.
\end{theorem}

\begin{proof}
	Suppose $X_1, \dots, X_r$ satisfy the hypotheses, and let $\b\lambda\in \DefoCone[\fw]$, with $\fw_{\b\lambda} = (\verts, \edges, \p[\psi])$ the associated deformation.
	For $j\in [r]$, let $\ell_j = \b\lambda_{e}$ for any $e\in X_j$.
	Consider two adjacent vertices $u, v\in V$.
	There is a path $\c P = (u = u_0, u_1, \dots, u_r = v)$ connecting $u$ and $v$, in which all edges $uv$ belong to some $X_{j_i}$.
	We have:
	$$\p[\psi]_v - \p[\psi]_u = \sum\nolimits_{1\leq i\leq r} (\p[\psi]_{u_i} - \p[\psi]_{u_{i-1}}) = \sum\nolimits_{1\leq i\leq r} \ell_{j_i} (\p_{u_i}-\p_{u_{i-1}})$$
	
	By definition, $\p[\psi]_v - \p[\psi]_u = \b\lambda_{uv} (\p_v - \p_u)$.
	Hence, for every edge $uv$ of $\fw_{\b\lambda}$, the value $\b\lambda_{uv}$ can be deduced from the values of $\ell_j$, $j\in [r]$.
	Equivalently, the projection $\b\lambda \mapsto (\ell_j ~;~ j\in [r])$ is injective on $\DefoCone$.
	This gives the bound $\dim\DefoCone[\fw] \leq r$.
\end{proof}

\begin{figure}[htpb]
\centering
\begin{subfigure}[b]{0.21\linewidth}
     \centering
     \includegraphics[width=\textwidth]{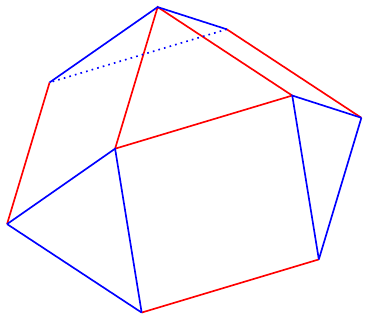}
     \caption{Triangular cupola}
     \label{sfig:TriangularCupola}
\end{subfigure}\hfill
\begin{subfigure}[b]{0.26\linewidth}
     \centering
     \includegraphics[width=\linewidth]{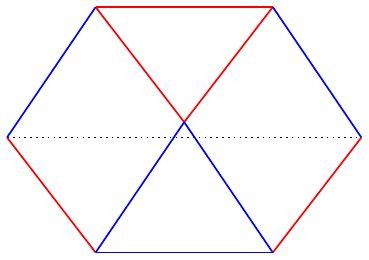}
     \caption{A sum of two triangles}
     \label{sfig:TriangleSum}
\end{subfigure}
\hfill
\begin{subfigure}[b]{0.16\linewidth}
     \centering
     \includegraphics[width=\textwidth]{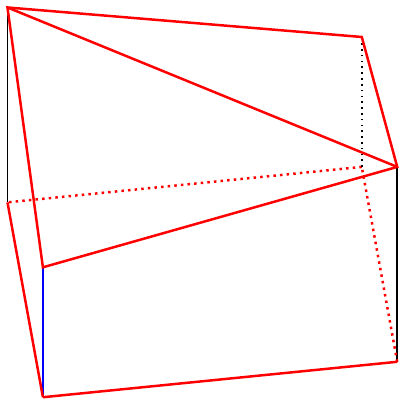}
     \caption{Chiseled cube}
     \label{sfig:FrustrumRegular}
\end{subfigure}
\caption[A sum of 2 triangles, a triangular cupola and a chiseled cube]{(Left) A triangular cupola $J_{3}$.
(Middle) A 3-dimensional polytope constructed as the sum of 2 triangles which share 1 edge direction.
(Right) A chiseled cube, obtained by lowering the $z$-coordinate of two diagonally opposite vertices.
Their deformation cones have dimension 2, and have 2 rays.}
\label{fig:TriangleSum}
\end{figure}

\begin{example}
The triangular cupola $J_3$ in \Cref{sfig:TriangularCupola} is a Minkowski sum of a tetrahedron with a triangle, which already shows that $\dim\DefoCone[J_3]\geq 2$. Using parallelograms and triangles, we deduce that any two edges with the same color in \Cref{sfig:TriangularCupola} are dependent.
Note that the hypotheses of \Cref{thm:DCdimBoundedByNumberOfDependentSets} are satisfied with $X_1$ the set of \textcolor{blue}{blue} edges and $X_2$ the set of \textcolor{red}{red} edges: this ensures that $\dim\DefoCone[J_3] \leq 2$. Consequently, $\dim\DefoCone[J_3] = 2$, and $\DefoCone[J_3]$ is a cone with 2 rays (the tetrahedron and the triangle), because every cone of dimension 2 has exactly 2 rays.

Now let $\pol$ be the polytope constructed as in \Cref{sfig:TriangleSum} by doing the sum of two triangles (the \textcolor{blue}{blue} one and the \textcolor{red}{red} one).
Again, as $\pol$ is constructed as a Minkowski sum, we get: $\dim\DefoCone \geq 2$.
Using parallelograms and triangles, we deduce that any two edges with the same color in \Cref{sfig:TriangleSum} are dependent. Even if we do not cover all the edges with these two colors, the hypotheses of \Cref{thm:DCdimBoundedByNumberOfDependentSets} are still satisfied, and we deduce that $\dim\DefoCone = 2$.
The cone $\DefoCone$ is the cone whose rays are the \textcolor{blue}{blue} and the \textcolor{red}{red} triangles.

Finally, consider the chiseled cube in \Cref{sfig:FrustrumRegular}, which is obtained from a standard $3$-cube by \emph{chiseling} two opposite vertices of the top face. Note that the vertical facets are trapezoids, and we can exploit \Cref{ex:parallelogram} to get the displayed edge-dependencies among the top and bottom edges in red. Taking one of the vertical edges as the second class, we get a connected spanning subgraph, proving $\dim\DefoCone \leq 2$. The rays of the deformation cone are a vertical segment, and the polytope obtained from the chiseled cube by Minkowski subtracting the vertical segment until one of the vertical edges collapses into a point.
\end{example}

\subsection{Pinning closures}\label{ssec:CoveringFamilyOfFlats}

In our final improvement, we will show how to grow dependent families of vertices and unveil more implicit edges using the affine flats spanned by subsets of vertices of our framework.

\subsubsection{Vertices pinned by families of flats}
The main new object are collections of \emph{flats} and the vertices they \emph{pin}.
\begin{definition}
Let $\fw=(\verts,\edges,\p)$ be a framework.
 A \defn{flat} of $\fw$ is a subset of vertices $F\subseteq \verts$ inducing a connected subframework of $\fw$ (that is, that for every pair of vertices in~$F$ there is a path of edges in $\fw$ only using vertices in $F$). 
 The \defn{direction} of a flat~$F$ is $\di=\lin\set{\p_u-\p_v}{u,v\in F}$. 
 
 We say that a collection of flats $\c C$:
	\begin{compactitem}

		\item  \defn{pins} a vertex $v\in \verts$ if there is a subset $\c I\subseteq \c C$ such that $v\in \bigcap_{F\in \c I}F$ and $\bigcap_{F\in \c I}\di=\{\b 0\}$; and that 
		\item  is \defn{grounded} on a subset of vertices $S\subseteq\verts$ if every flat of $\c C$ contains some vertex in $S$.
	\end{compactitem}
\end{definition}

\begin{remark}
	Note that, for a polytope $\pol$, the flats of $\fw(\pol)$ \textbf{are not} the flats of its associated matroid.
	
	However, as the graph of any face of $\pol$ is connected, the set of vertices of any face of $\pol$ forms a flat in~$\fw(\pol)$.
	In particular, the collection formed by all the facets of $\pol$ pins each of its vertices. Furthermore, if $\c C$ is a sub-collection of facets of the $d$-polytope $\pol$ such that each vertex of $\pol$ is contained in at least $d$ facets in $\c C$ that have linearly independent normal vectors, then $\c C$ pins each of the vertices of $\pol$.
\end{remark}

\begin{lemma}\label{lem:DirDeformation}
	For a deformation $\fw[G]$ of $\fw$, if $F$ is a flat in $\fw$, then $F$ is a flat in $\fw[G]$ with $\di[{\fw[G]}] \subseteq \di$.
\end{lemma}

\begin{proof}
	Let $\fw = (\verts,\edges,\p)$ and $\fw[G] = (\verts,\edges,\p[\psi])$. 
	Since both share the same underlying graph, they have the same flats, because these do not depend on the realization. 
	If $F$ is a flat, then it is connected, and therefore \[\di=\lin\set{\p_u-\p_v}{u,v\in F}=\lin\set{\p_u-\p_v}{u,v\in F\text{ and }uv\in E}\] 
	because every vector $\p_u-\p_v$ can be decomposed as $\sum_{i=0}^{k-1}\p_{u_{i+1}}-\p_{u_i}$ if $u=u_o,u_1,\dots, u_k=v$ is a path joining $u$ and $v$ in the induced subgraph. Moreover, since $\fw[G]$ is a deformation of $\fw$, for any edge $uv$, we have $\p[\psi]_u - \p[\psi]_v = \b\lambda_{uv}(\p_u - \p_v)$ for some $\b\lambda\in \R$.
	This yields $\di[{\fw[G]}] \subseteq \di$ (where the inclusion can be strict if some $\b\lambda_{uv} = 0$).
\end{proof}

\begin{lemma}\label{lem:PinnedVertex}

Let $\c C$ be a collection of flats of $\fw$ grounded on a subset $S\subseteq \verts$ that pins a vertex $u$.
If $\fw[G] = (\verts, \edges, \p[\psi])$ and $\fw[G]' = (\verts, \edges, \p[\psi]')$ are deformations of $\fw$ such that $\p[\psi]_{v} = \p[\psi]'_{v}$ for every $v\in S$, then $\p[\psi]_{u}=\p[\psi]'_{u}$.
\end{lemma}

\begin{proof}
	For each $F\in \c C$, there is a point $v_F\in F\cap S$. 
	We have $H_F = \aff\{\p[\psi]_u ~;~ u \in F\} = \p[\psi]_{v_F} + \di[{\fw[G]}]$.
	As $\fw[G]$ is a deformation of $\fw$, by \Cref{lem:DirDeformation}, we have that $H_F \subseteq \bigl(\p[\psi]_{v_F} + \di\bigr)$.
	As $u\in F$, we have $\p[\psi]_u\in H_F$.
	In particular, we have $\p[\psi]_u \in \bigcap_{F\in \c C} H_F \subseteq \bigcap_{F\in \c C} \bigl(\p[\psi]_{v_F} + \di\bigr)$.
	As $\c C$ pins $u$, we get $\bigcap_{u\in F\in \c C} \di = \{\b 0\}$, so $\bigcap_{u\in F\in \c C} \bigl(\p[\psi]_{v_F} + \di\bigr)$ is a point.
	We conclude that $\bigcap_{u\in F\in \c C} \bigl(\p[\psi]_{v_F} + \di\bigr) = \{\p[\psi]_u\}$.
	Similarly, the same holds for $\fw[G]'$, yielding $\bigcap_{u\in F\in \c C} \bigl(\p[\psi]'_{v_F} + \di\bigr) = \{\p[\psi]'_u\}$.
	As $\p[\psi]_{v_F} = \p[\psi]'_{v_F}$ for all $F\in\c C$, we get: $\p[\psi]_u = \p[\psi]'_u$.
\end{proof}

The following result is our main insight. It allows to iteratively increase dependent subsets of vertices.

\begin{proposition}\label{prop:pinningdependent}
	Let $\fw = (\verts,\edges,\p)$ be a framework, $S\subseteq \verts$ a dependent subset of vertices, and $v\in \verts$. If there is a collection of flats $\c C$ grounded on~$S$ that pins~$v$, then $S\cup \{v\}$ is a dependent subset.	
\end{proposition}
\begin{proof}
	Let $u\in S$. If $uw$ forms a degenerate implicit edge (\ie $\p_u=\p_w$) for every $w\in S$, then for any collection of flats $\c C$ grounded on~$S$ we have $\bigcap_{u\in F\in \c C} \bigl(\p_{u} + \di\bigr) \supseteq \{\p_u\}$, and thus any vertex $v$ pinned by~$\c C$ fulfills $\p_u=\p_v$. In this case $S\cup \{v\}$ form a clique of degenerate implicit edges and there is nothing to prove.
	 
	Otherwise, consider a vertex $w\in S$ such that $u,w\in\ndie$. If $\fw[G]=(\verts,\edges,\p[\psi])$ is a deformation of $\fw$, then there is some $\lam\geq 0$ such that $\p[\psi]_w-\p[\psi]_u=\lam(\p_w-\p_u)$. 	Let $\fw[G]'=(\verts,\edges,\p[\psi]'\eqdef x\mapsto \p[\psi]_u +\lam (\p_x -\p_u))$. 
	Note that since $S$ is dependent, for all $x\in S$ we have that $ux$ is an implicit edge with the same edge-deformation coefficient~$\b\lambda$, and hence
	$\p[\psi]_x=\p[\psi]_u+\lam(\p_x-\p_u)=\p[\psi]'_x$. 
	
	Now, as $\c C$ pins $u$ and $F\cap S\ne\emptyset$ for all $F\in \c C$, \Cref{lem:PinnedVertex} gives $\p[\psi]_v = \p[\psi]_v' = \p[\psi]_u +\lam (\p_v -\p_u)$. Thus $\p[\psi]_v-\p[\psi]_u=\lam(\p_v-\p_u)$ and $uv$ forms an implicit edge dependent with $uw$. 
	
	This is independent of the chosen $u$, showing that $S\cup\{v\}$ forms a dependent subset.	
\end{proof}

This motivates the following definition.

\begin{definition}\label{def:pinningclosure}
	Let $\fw = (\verts,\edges,\p)$ be a framework, and $S\subseteq \verts$ a 
	subset of vertices. Set $S_0=S$ and for $i\geq 1$ define\[ 	S_i \eqdef S_{i-1} \cup	\set{v\in \verts}{v \text{ is pinned by a collection of flats grounded on } S_{i-1} }. 	\]
	Since $\verts$ is finite, the sequence stabilizes and there exists $k$ with $S_k = S_{k+1}$.  
	We call the resulting set $S_k$ the \defn{pinning closure} of $S$, and denote it by $\mathdefn{\pcl}\eqdef S_{k}$.
\end{definition}

Note that if $S$ is dependent, then so is its pinning closure $\pcl$, by \Cref{prop:pinningdependent}.

\subsubsection{The main indecomposability theorem}
The next theorem encompasses and improves all the indecomposability criteria presented in previous sections. The heavy lifting is done by \Cref{prop:pinningdependent} above.

\begin{theorem}\label{thm:mainthm}
	Let $\fw=(\verts,\edges,\p)$ be a framework. If there is a dependent subset of vertices $S\subseteq \verts$ whose pinning closure covers all vertices, $\pcl=\verts$, then $\fw$ is indecomposable.
\end{theorem}

\begin{proof}
If $S$ is dependent, then so is its pinning closure by \Cref{prop:pinningdependent}. Thus $\verts$ is a dependent subset of vertices and we conclude by \Cref{cor:DependentVerticesIndecomposable}.
\end{proof}

\begin{figure}[htpb]
	\centering
	\includegraphics[width=0.65\linewidth]{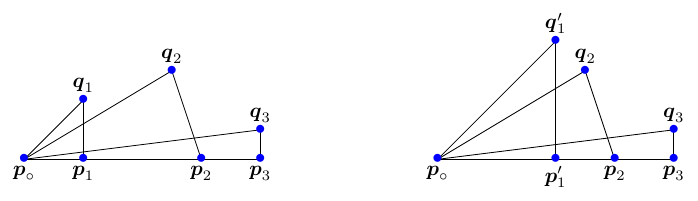}
	\caption{
		Two frameworks which are deformations of each other, where all $\b p_i$ are pairwise linked by an edge. The collection of black lines supports a covering collection of flats, and $S = \{\b p_\circ, \b p_1, \b p_2, \b p_3\}$ is a family of vertices touching every flat.
        Since $\p(S)$ is 1-dimensional, this does not guarantee indecomposability.}
	\label{fig:HypothesesMainThm}
\end{figure}

\begin{remark}
	If $\fw = (V, E, \p)$ is 1-dimensional, a collection $\c C$ of flats that pins a vertex $v$ of~$\fw$ must contain the singleton $\{v\}$. Hence $S=\pcl$ and \Cref{thm:mainthm} is straightforward: $S$ being dependent amounts to every pair of vertices forming dependent implicit edges.
	
	On the other hand, if $S$ is at least $2$-dimensional (\ie $\dim\lin\{\p_v ~;~ v\in S\} \geq 2$), then \Cref{cor:Dim2DependentVerticesDependentEdges} allows to replace ``dependent subset of vertices $S\subseteq V$'' by ``subsets of vertices $S\subseteq V$ pairwise linked by an implicit edge'' in \Cref{thm:mainthm}, which is a slightly weaker condition that is easier to check in practice.
	
	The condition that $S$ is at least $2$-dimensional is necessary. For example, consider the \emph{wedge of triangles}~$\fw$ given in \Cref{fig:HypothesesMainThm}.
It is decomposable.
		However, let $\c C$ be the collection formed by the flat $(\b p_\circ, \b p_1, \b p_2, \b p_3)$ together with the flats $(\b p_i, \b q_i)$ and $(\b p_\circ, \b q_i)$ for each $i\in [3]$.
		In the set of vertices $S = \{\b p_\circ, \b p_1, \b p_2, \b p_3\}$, there are (implicit) edges between every pair of vertices (but these edges are not pairwise dependent), and every flat of $\c C$ contains a vertex of~$S$; yet $\fw$ is decomposable.
\end{remark}

\begin{example}\label{ex:iterativeclosure}
	Consider a $3$-polytope $\pol$ whose graph coincides with that of \Cref{fig:IterativePinning}. We can use \Cref{thm:mainthm} to prove that $\pol$ is indecomposable as follows. The subset of vertices $S=\{A,B,C\}$ forms a triangle, and is therefore dependent by \Cref{ex:triangle}. The flats $\{A,D,E,F\}$, $\{B,E,G,H\}$ and $\{C,E,I,J\}$ are grounded on $S$, and pin $E$. Therefore $S_1=\{A,B,C,E\}$ is a dependent subset. The collection of all facets of $\pol$ is grounded on $S_1$, and pins all the vertices. Therefore $\pol$ is indecomposable. However, our previous theorems were not enough for proving it.
\end{example}

\begin{figure}[htpb]
	\centering
	\includegraphics[width=0.33\linewidth]{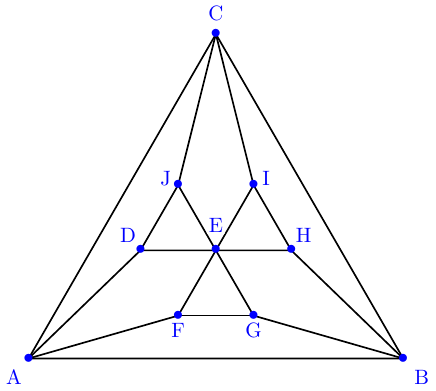}
	\caption{Any $3$-polytope with this graph is indecomposable.}
	\label{fig:IterativePinning}
\end{figure}

The pinning closure is defined iteratively, and \Cref{ex:iterativeclosure} shows that multiple iterations may be required. However, the special case when only one iteration is needed has a slightly simpler formulation that is sufficient in many cases. We say that a collection of flats is \defn{covering} if it pins all the vertices.

\begin{corollary}\label{cor:mainthm1step}
	Let $\fw=(\verts,\edges,\p)$ be a framework. 
	If there is a dependent subset of vertices $S\subseteq \verts$, and a covering collection of flats $\c C$ grounded on $S$, then $\fw$ is indecomposable.
\end{corollary}

\begin{proof}
As $\c C$ is covering, we get $\pcl = V$, hence \Cref{thm:mainthm} gives the claim.
\end{proof}

\subsubsection{Recovering McMullen's criterion}
Identifying covering collections can be difficult, but for a polytope $\pol$, it is direct to see that the facets of $\pol$ form a covering collection.

\begin{corollary}\label{cor:mainthmpolytopes}
	Let $\pol$ be a polytope. If there is a dependent subset of vertices $S\subseteq \verts(\pol)$ such that every facet of $\pol$ contains a vertex in $S$, then $\pol$ is indecomposable.

	In particular, if for $S\subseteq \verts(\pol)$ any pair of vertices of $S$ is connected through a path of pairwise dependent implicit edges and every facet of $\pol$ contains a vertex in $S$, then $\pol$ is indecomposable. 
\end{corollary}

\begin{proof}
	Direct application of \Cref{cor:mainthm1step}, as the collection of facets is a covering collection of flats.
\end{proof}

With it, we can deduce another indecomposability criterion introduced in~\cite{McMullen1987}. While not strictly comparable to \Cref{cor:McMullen1}, both turn out to be consequences of \Cref{cor:mainthmpolytopes}.
A family of faces $\scr F$ of a polytope is \defn{strongly connected} if for any $\pol[F],\pol[G]\in \scr F$ there is a strong chain $\pol[F]=\pol[F]_0,\pol[F]_1,\dots,\pol[F]_k=\pol[G]$ with each $\pol[F]_i\in \scr F$. We say that $\scr F$ touches a face $\pol[F]$ if $(\bigcup \scr F)\cap \pol[F]\neq \emptyset$.

\begin{corollary}[{\cite[Thm.~2]{McMullen1987}}]\label{cor:McMullen}
	If a polytope $\pol$ has a strongly connected family of indecomposable faces which touches each of its facets, then $\pol$ is itself indecomposable.
\end{corollary}

\begin{proof}
	The union of the vertex sets of the indecomposable faces in the strongly connected family form a dependent subset. Indeed, all the edges in each of the faces forms a clique in $\ELD[\pol]$ by \Cref{cor:indecsubframeclique}, and these cliques share a node because the family is strongly connected.
	Thus we have a dependent subset of vertices touching all facets, and we conclude with \Cref{cor:mainthmpolytopes}.
\end{proof}

\subsubsection{A final bound for the dimension of the deformation cone}
As we did with our previous indecomposability theorems, \Cref{thm:mainthm} extends to an upper bound for the dimension of deformation cones.


\begin{theorem}\label{thm:DCdimBoundedByNumberOfDependentSetsCovering}
	Let $\fw=(\verts,\edges,\p)$ be a framework, $X_1, \dots, X_r$ be dependent sets of implicit edges, and $S\subseteq \verts$ be a subset of vertices such that
any $2$ vertices of $S$ are connected in the graph $G = (\verts, \edges)$ through a path of implicit edges in~$\bigcup_i X_i$. If $\pcl=\verts$ then $\dim\DefoCone[\fw]\leq r$.
\end{theorem}

\begin{proof}
	We will combine (the proof of) \Cref{thm:DCdimBoundedByNumberOfDependentSets} with \Cref{lem:PinnedVertex}.
%
%
%
%
%

	Fix $\b\lambda\in \DefoCone[\fw]$ and $\fw[G] = (\verts, \edges, \p[\psi])$ the associated deformation of $\fw$.
	For $i\in [r]$, let $\ell_i(\b\lambda)  = \b\lambda_{e_i}$ for some $e_i\in X_i$.
	Note that, as $X_i$ is a set of dependent edges, all $\b\lambda_{e}$ share the same value for $e\in X_i$, hence $\b\ell(\b\lambda) \coloneqq \bigl(\b\ell_1(\b\lambda) ,\dots, \b\ell_r(\b\lambda) \bigr)$ is well-defined.
	We show that $\b\lambda\mapsto \b\ell(\b\lambda)$ is an injective linear projection.

	Let $S_0=S,S_1,\dots,S_k=\pcl$ be defined as in \Cref{def:pinningclosure}.
	
	Let $u, v\in S$, and $\c P = (u = u_0, u_1, \dots, u_k = v)$ a path from $u$ to $v$ using only edges in $\bigcup_i X_i$. 
	For $i\in [k]$, let $j_i$ such that $u_{i-1}u_i\in X_{j_i}$, we have: $\p[\psi]_v - \p[\psi]_u = \sum_{j\in [k]} \b\lambda_{u_{i-1}u_i} (\p_{u_i} - \p_{u_{i-1}}) = \sum_{j\in [r]} \ell_{i_j}(\b\lambda) \,(\p_{u_i} - \p_{u_{i-1}})$.
	
	Now, fix $v_\circ\in S$ and pick two deformations $\fw[G], \fw[G]'$ of $\fw$ with $\p[\psi]_{v_\circ} = \p[\psi]'_{v_\circ}$.
	Let $\b\lambda, \b\lambda'$ be the associated edge-deformation vectors.
	If $\b\ell(\b\lambda) = \b\ell(\b\lambda')$, then we have for all $u\in S$: $\p[\psi]_u - \p[\psi]_{v_\circ} = \sum_{j\in [r]} \ell_{i_j}(\b\lambda) \,(\p_{u_i} - \p_{u_{i-1}}) = \sum_{j\in [r]} \ell_{i_j}(\b\lambda') \,(\p_{u_i} - \p_{u_{i-1}}) = \p[\psi]'_u - \p[\psi]'_{v_\circ}$.
	Hence, $\p[\psi]_u = \p[\psi]'_u$ for all $u\in S=S_0$.
	
	Assume now for induction that $\p[\psi]_u = \p[\psi]'_u$ for all $u\in S_i$. Then every $v\in S_{i+1}$ is pinned by a collection of flats grounded on $S_i$. Then  \Cref{lem:PinnedVertex} ensures that $\p[\psi]_v = \p[\psi]'_v$. By induction we deduce that $\p[\psi]_v = \p[\psi]'_v$ for all $v\in S_k=\verts$.
	
	Consequently, $\b\lambda\mapsto\b\ell(\b\lambda)$ is an injective projection from $\DefoCone[\fw]$ to $\R^r$, implying: $\dim\DefoCone[\fw] \leq r$.
\end{proof}

\begin{figure}[htpb]
	\centering
	\includegraphics[width=0.8\linewidth]{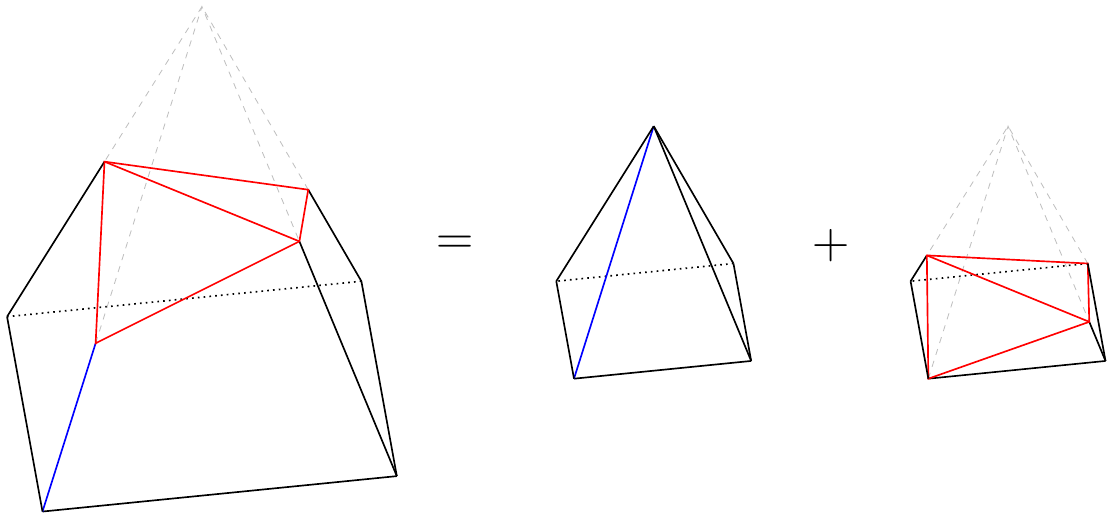}
	\caption{Decomposition of a chiseled square pyramid as a square pyramid $+$ a deeply chiseled cube.}
	\label{fig:Frustrum}
\end{figure}

\begin{example}
	Consider the left-most polytope $\pol$ in \Cref{fig:Frustrum}, a chiseled square pyramid. It is combinatorially equivalent to the chiseled cube from \Cref{sfig:FrustrumRegular}, but the vertical edges are not parallel any more. Therefore we cannot use trapezoids to derive dependencies. However, it is enough to consider $X_1$ to be the set of \textcolor{red}{red} edges, which are dependent because of the triangles, and $X_2$ a single \textcolor{blue}{blue} edge shown in \Cref{fig:Frustrum}. If $\c C$ is the collection of facets of $\pol$, and $S$ the set of vertices adjacent to either a \textcolor{blue}{blue} or a \textcolor{red}{red} edge, then the hypotheses of \Cref{thm:DCdimBoundedByNumberOfDependentSetsCovering} are satisfied.
	This implies that $\dim\DefoCone \leq 2$. As $\pol$ can be written as a Minkowski sum (see \Cref{fig:Frustrum}), we get that $\dim\DefoCone \geq 2$.
	Consequently, $\dim\DefoCone = 2$. 
	Given that a $2$-dimensional cone has exactly 2 rays, the unique (up to scaling) Minkowski decomposition of $\pol$ into indecomposable polytopes is the one written in \Cref{fig:Frustrum}.
\end{example}











\subsection{The indecomposability game}

One can exemplify the core idea of our methods to prove polytope indecomposability as a 1-player game.
Pick a polytope $\pol$, and a starting edge $\pol[e]$: paint it red.
Then recursively paint in red all edges that belong to a triangle with a red edge, or that are opposite to a red edge in a parallelogram.
Once done, if the collection of red edges is connected (in the 1-skeleton of $\pol$) and if for each facet of $\pol$ there is at least one vertex of $\pol$ which belongs to one of the red edges, then the polytope $\pol$ is indecomposable.
\Cref{fig:SmallDimPandQ} is a good example of the utility of such game.

\vspace{0.15cm}

This first and easy version of the indecomposability game is already sufficient in many cases, but we can go far beyond. 
To this end, one should interpret of our main result, \Cref{thm:mainthm}, as a meta-algorithm that can be used to prove the indecomposability of a framework (or a polytope). Indeed, it tells that to prove indecomposability one should find a large family of pairwise dependent implicit edges, but how to find those is left open (just like classical criteria such as
\Cref{thm:Shephard,cor:McMullen1,cor:McMullen} left open how to find strong chains and strongly connected families). 

Yet, we have provided several tools to uncover edge dependencies: 
\begin{compactitem}
	\item First, all degenerate edges should be contracted using \Cref{prop:degenerate}.
	\item The classical approach is to exploit indecomposable subframeworks via \Cref{prop:subframework}. This requires a previous database of indecomposable subframeworks, but on the other hand also turns every new framework that could not be proven to be indecomposable by previous methods into a new tool to find edge dependencies. The most basic indecomposable subframework consists of a cycle supported on affinely independent points \Cref{cor:RigidCycle0}, and the fundamental example are triangles.
	\item The novel insight from \Cref{prop:projection} is that projections can be used to lift dependencies. This does not require to have indecomposable subframeworks, but instead that there are other subframeworks that can be used to transmit information on edge dependencies, without necessarily being as rigid. The main strategy that we envision is to contract subspaces spanned by some of the edges of our framework, via \Cref{cor:subspaceprojection}. The fundamental example are parallelograms (or just trapezoids).
	\item Once these techniques are exhausted, one should uncover implicit edges, and add them to the framework. This will not alter the deformation cone (\Cref{cor:extendedgraph}), but the new implicit edges can be used to discover new edge dependencies.
	The simplest method for uncovering implicit edges is \Cref{prop:connectedimplicit}: paths of dependent edges can be replaced by an implicit edge between the extremities of the path.
	\item Finally, we can use the pinning method from \Cref{prop:pinningdependent}. The dependent implicit edges give rise to dependent subsets of vertices (\ie dependent cliques of implicit edges). We should look at collections of flats grounded in such a set to compute their pinning closure, potentially adding many new dependent implicit edges. Finding such collections of flats 
	depends a lot on the concrete applications, but a good example are the facets of a polytope, which pin all the vertices (\cf \Cref{cor:mainthmpolytopes}).
	
	\item Note that this starts a recursive procedure. Indeed, the new implicit edges might produce new subframeworks on which to apply the above techniques, that were not apparent before. Hence after unveiling the implicit edges, one should go back to the previous steps.

\end{compactitem}

Once this has been repeatedly applied, one has hopefully found that the whole set of vertices is dependent. 
If this approach is unsuccessful to prove indecomposability, then at least it will give an upper bound on the dimension of $\DefoCone[\fw]$, see \Cref{thm:dim1,thm:DCdimBoundedByNumberOfDependentSets,thm:DCdimBoundedByNumberOfDependentSetsCovering} (and one should try proving the framework at stake is decomposable, after all).

\begin{figure}
	\centering
	\includegraphics[width=0.99\linewidth]{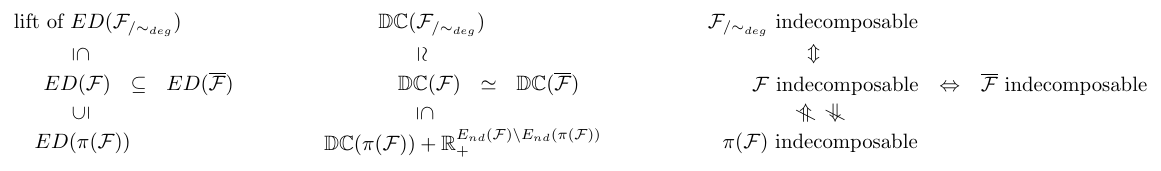}
	\caption[Interplay between the framework $\fw$, its closure $\Efw$, its quotient $\fw_{/\der}$, and its projection $\pi(\fw)$]{The interplay between the framework $\fw$, its closure $\Efw$, its quotient $\fw_{/\der}$, and its projection $\pi(\fw)$.
		(Middle) We denote ``$\simeq$'' the linear equivalence between cones.}
	\label{fig:DiagrammFrameWork}
\end{figure}

Some of these techniques are summarized in \Cref{fig:DiagrammFrameWork} (left): for a framework $\fw$, it shows the interplay between the graph of edge-dependencies of $\fw$ and its counter-parts for the closure of $\fw$, for its quotient and for its projections. 
The right of this figure illustrates the immediate consequences one can deduce from these inclusion of edge-dependencies graph.

\vspace{0.25cm}

Finally, note that the (in)decomposability of an explicit framework or polytope given by the coordinates of its vertices can be decided in polynomial time:
as presented in \Cref{ssec:Prelim}, the cone $\DefoCone[\fw]$ lies in $\R_+^{f_1(\fw)}$ and is defined by $f_2(\fw)$ equalities where $f_1(\fw)$ and $f_2(\fw)$ are the number of edges and cycles of~$\fw$ (or of 2-faces of a polytope $\pol$)\footnote{Dually, it is also polynomial in the number of facets and ridges of a polytope.}, respectively.
By \Cref{rmk:dimension}, we can use Gaussian elimination to test whether ``$\dim\DefoCone[\fw] = 1$'' or not, in polynomial time in $f_1(\fw)$ and $f_2(\fw)$.
Our method is not designed as an optimization of this algorithm, but rather finds its utility in other contexts.
For example, it allows us to decide indecomposability for families of frameworks (as in \Cref{thm:PandQareIndecomposable}), or to conclude it only with partial combinatorial information without the exact vertex coordinates (as in \Cref{thm:DecomposabilityVertexStacking}).

\section{New rays of the submodular cone}\label{sec:newrays}

In this section, we construct a new infinite family of rays of the submodular cone. These are constructed by truncating graphical zonotopes of complete bipartite graphs (and taking permutahedral wedges over them). We use the previous results to certify indecomposability.
Except for the indecomposability proofs in \Cref{thm:ConnectedMatroidPolytopesAreIndecomposable,thm:PandQareIndecomposable,thm:PermutahedralWedgesIndecomposable}, the rest of this section does not require~\Cref{sec:criteria}.

\subsection{Prelude: Connected matroid (base) polytopes}\label{sec:matroidpolytopes}

Before presenting our new rays of the submodular cone, we use the graph of edge dependencies to prove that matroid polytopes of connected matroids are indecomposable, thus providing a large family of indecomposable polytopes. We are very grateful to Federico Castillo, who encouraged us to add this section to our manuscript, and shared his thoughts on the proof. In particular, he pointed us to \cite[Proposition 4.3.2]{Oxley} as a potential key step for the proof.

Let $M$ be a matroid with ground set~$\ground$, and set of bases~$\bases(\mat)$ (see~\cite{Oxley} for a thorough introduction to matroid theory).
The \defn{matroid (base) polytope} of~$M$ is the convex hull of the indicator vectors of its bases:
\[\mathdefn{\matpol}\eqdef \conv\set{\b e_B\eqdef \sum_{x\in B}\b e_x}{B\in \bases(\mat)}\subseteq \R^\ground.\]
This gives rise to the following polyhedral characterization of matroids.

\begin{theorem}[{\cite[Thm.~4.1]{GelfandGoreskiMacPhersonSerganova}}]
	A collection $\bases$ of subsets of $\ground$ is the set of bases of a matroid if and only if all edges of 
	$\conv\{\b e_B ~;~ B\in \bases\}$
	are in direction $\b e_i-\b e_j$ for some $i,j\in\ground$.
\end{theorem}

That is, matroid polytopes are characterized by being the deformed permutahedra with $0/1$-coordinates \cite{GelfandGoreskiMacPhersonSerganova}. More precisely, there is an edge between the indicator vectors of two bases if and only if their symmetric difference has cardinal $2$.

The \defn{connected components} of $M$ are the equivalence classes of the relation:
$x\sim y$ if there is a circuit of~$M$ containing both $x$ and $y$. Every matroid is the direct sum of its connected components~\cite[Proposition 4.2.8]{Oxley}, 
and it follows from the definition that the matroid polytope of a direct sum is the Cartesian product of the matroid polytopes of its connected components (which is a Minkowski sum because they lie in complementary subspaces);
that is, if $\ground_1\sqcup \ground_2\sqcup \dots\sqcup \ground_r = \ground$ are the connected components of~$\mat$, then:
\begin{equation}\label{eq:matpoldecomposition}
	\matpol = \matpol[\mat_{|\ground_1}]\times\matpol[\mat_{|\ground_2}]\times \dots \times \matpol[\mat_{|\ground_r}] = \matpol[\mat_{|\ground_1}] + \matpol[\mat_{|\ground_2}] + \dots +\matpol[\mat_{|\ground_r}] 
\end{equation}

In~\cite[Section 10]{Nguyen1986-SemimodularFunctions}, Nguyen gave a converse statement (see \cite[Section 7.2]{StudenyKroupa2016} for an explication of the literature around Nguyen's theorem).
He proved that matroid polytopes of connected matroids are indecomposable.
We give a new proof, exploiting our tools and the following result on connected matroids.

For a base $B$ and an element $e\in \ground\ssm B$, we denote by $C(B,e)$ the \defn{fundamental circuit} of $e$ with respect to $B$ (the unique circuit in $B\cup \{e\}$).
The \defn{fundamental graph} of $\mat$ with respect to $B$ is the bipartite graph $G_B(\mat)$ with nodes $\ground$ and arcs $ij$ for $i\in B$, $j\notin B$ with $i\in C(B,j)$. 
Note that the arcs of $G_B(\mat)$ are in bijection with the edges of $\matpol$ incident to the vertex $\b e_B$. Namely, $ij$ is an arc of $G_B(\mat)$ if and only if $\b e_B+\b e_j -\b e_i$ is a vertex of $\matpol$ (which necessarily forms an edge with $\b e_B$).


\begin{lemma}[{\cite[Prop.~4.3.2]{Oxley}, originally from \cite{Cunningham1973-Decomposition,Krogdahl1977-DependenceBasesInMatroids}}]\label{lem:connectedfundamentalgraph}
	Let $\mat$ be a non-empty matroid and $B$ one of its bases. Then $\mat$ is connected (as a matroid) if and only if $G_B(\mat)$ is connected (as a graph).
\end{lemma}

%

\begin{theorem}[\cite{Nguyen1986-SemimodularFunctions}]\label{thm:ConnectedMatroidPolytopesAreIndecomposable}
	If $\mat$ is a connected matroid, then $\matpol$ is indecomposable.
\end{theorem}

\begin{proof}
We will prove first that for every vertex $\b e_B$ of $\matpol$, all the edges incident to $\b e_B$ are pairwise dependent.

Let $j\notin B$, and consider the fundamental circuit $C(B,j)$. 
For every $i\in C(B,j)$, let $B_{ij}\eqdef B\cup\{j\} \ssm\{i\}$ (in particular $B_{jj}=B$). Note that each of these $B_{ij}$ is a basis (it cannot contain a circuit because $C(B,j)$ is the unique circuit in $B\cup\{j\}$), and that for $i,k\in C(B,j)$, we have $|B_{ij}\,\Delta\, B_{kj}|=2$, and thus their indicator vectors $\b e_{B_{ij}}$ and $\b e_{B_{kj}}$ are joined by an edge in $\matpol$. 
Hence, the vertices $\b e_{B_{ij}}$ for $i\in C(B, j)$ induce a complete subframework of $\fw(\matpol)$.
Moreover, these vertices are affinely independent (because $\b e_{B_{ij}}$ has a non-zero $i^{\text{th}}$-coordinate and a zero $k^{\text{th}}$-coordinate for $k\in C(B, j)$ with $k\ne i$).
Therefore, we can apply \Cref{cor:RigidCycle0} to conclude that all its edges are dependent. 
This shows that for a given $j\notin B$, all the arcs of $G_B(\mat)$ of the form $ij$ are pairwise connected.
	
Now, let $i\in B$. And let $j,k\notin B$ such that $i\in C(B,j)\cap C(B,k)$. Then the bases $B$, $B_{ij}$, and $B_{ik}$ differ pairwise in two elements, and hence they form a triangle in $\matpol$.
Thus, by \Cref{ex:triangle}, the three edges between the corresponding indicator vectors are pairwise dependent.

We have shown that if two arcs of $G_B(\mat)$ are connected in its line graph, then the corresponding edges of $\matpol$ are dependent. Since $G_B(\mat)$ is connected by \Cref{lem:connectedfundamentalgraph}, so is its line graph. And hence all the edges incident to $\b e_B$ are pairwise dependent.
This implies that all the edges are dependent, because the graph of $\matpol$ is connected, and so is its line graph.
\end{proof}

Note that, as explained in \cite[Section 7.2]{StudenyKroupa2016}, the following corollary is often misquoted without having into account the loops and coloops.
A (nontrivial) connected matroid cannot have loops or coloops (which form their own connected components), but adding a loop or a coloop to a matroid only amounts to translating the matroid polytope, without altering its indecomposability.
The key is that in the decomposition of \eqref{eq:matpoldecomposition} the presence of loops and coloops gives rise to summands that are a single point (this is the only case, because if a matroid has a single basis, then all its elements are either loops or coloops).
\begin{corollary}
	If $\mat$ is a matroid, then $\matpol$ is indecomposable if and only if $\mat$ stripped of loops and coloops is connected.
\end{corollary}

\begin{proof}
The ``if'' is immediate from \Cref{thm:ConnectedMatroidPolytopesAreIndecomposable}, and the ``only if'' from \eqref{eq:matpoldecomposition}.
\end{proof}

\begin{corollary}\label{cor:matroidpolytopesimplicial}
	For any matroid $\mat$ with $r$ connected components $S_1, \dots, S_r$, the deformation cone $\DefoCone[{\matpol}]$ is a simplicial cone of dimension $r$.
	Consequently, $\matpol$ has a unique decomposition into indecomposable polytopes, namely $\matpol = \matpol[\mat_{|\ground_1}] + \matpol[\mat_{|\ground_2}] + \dots +\matpol[\mat_{|\ground_r}]$, and each of its deformations is normally equivalent to $\sum_{j\in A} \matpol[\mat_{|\ground_j}]$ for some $A\subseteq[r]$.
\end{corollary}

\begin{proof}
	As $\matpol = \matpol[\mat_{|\ground_1}] \times \dots \times \matpol[\mat_{|\ground_r}]$, we can apply \Cref{thm:products}, using the fact that $\DefoCone[{\matpol[\mat_{|\ground_j}]}]$ is a 1-dimensional ray, for each $j\in [r]$.
	The second sentence is the direct application of \Cref{cor:uniquelydecomposablepolytope}, and the fact that the deformation cone is simplicial.
\end{proof}

\subsection{Graphical zonotopes}\label{sec:graphicalzonotopes}

Let $G=(V,E)$ be a graph with node set $V$ and arc\footnote{Recall that we will reserve the names \emph{vertex} and \emph{edge} for polytopes and frameworks, and use \emph{node} and \emph{arc} for abstract graphs.} set $E$.
An orientation of $G$ is an \defn{acyclic orientation} if there is no directed cycle.
Let \defn{$\AO(G)$} be the set of acyclic orientations of $G$.
An \defn{ordered partition} of~$G$ is a pair $(\mu,\omega)$ consisting of a partition $\mu$ of $V$ where each part induces a connected subgraph of~$G$, together with an acyclic orientation $\omega$ of the contraction $\contrG$. 

The \defn{graphical zonotope} $\ZG \subset \R^V$ is defined as the Minkowski sum 
$\ZG 
= \sum_{e\in E} \simplex[e] 
\coloneqq \sum_{uv \in E} [\b e_u, \b e_v]$, where $(\b e_u)_{u \in V}$ is the standard basis of~$\R^V$, and $[\b p,\b q]$ denotes the segment between $\b p$ and $\b q$.
The faces of a graphical zonotope are indexed by ordered partitions (see \cite[Prop.~2.5]{Stanley2007} or \cite[Sec.~1.1]{OrientedMatroids}).
Namely, the ordered partition $(\mu,\omega)$ is associated to the face:
\[\polZ_{(\mu,\omega)}\coloneqq
\sum_{u\to v \text{ in }\omega} \{\b e_v\}+\sum_{\substack{uv \in E \text{ in the}\\ \text{ same part of }\mu}} [\b e_u, \b e_v]
\]
In particular:
\begin{compactenum}
    \item The vertices of $\ZG$ are in bijection with the acyclic orientations $\AO(G)$. Namely, 
   each $\rho\in \AO(G)$ is in correspondence with the vertex $\b v_\rho = \sum_{i\in V} d_{in}(i, \rho) \,\b e_i$ where $d_{in}(i, \rho)$ is the in-degree of the node $i\in V$ in the acyclic orientation $\rho$.
    \item The edges $\pol[e]_{g, \rho}$ of $\ZG$ are in bijection with the couples $(g, \rho)$ where $g\in E$ and $\rho\in \AO (\contrG[e])$, \ie with the couples of acyclic orientations $\rho_1, \rho_2\in \AO(G)$ which differ only on the orientation of the arc $g$. Moreover, if $g =uv$ is oriented $u\to v$ in $\rho_1$, then $\b v_{\rho_1} - \b v_{\rho_2} = \b e_v - \b e_u$.
    \item The $2$-faces of $\ZG$ are:
    \begin{compactenum}
        \item Hexagons: one per $(t, \rho)$ where $t$ is a triangle in $G$ and $\rho\in\AO(\contrG[t])$.
        \item Parallelograms: one per 
        $(s, \rho)$ where $s$ is either the union of two disjoint edges or an induced path of length 2, and $\rho\in \AO(\contrG[s])$
    \end{compactenum}
\end{compactenum}


\subsection{Truncated graphical zonotopes of complete bipartite graphs}\label{ssec:TruncatedGraphicalZonotope}
For $n, m\geq 1$, let \defn{$\Knm$} be the complete bipartite graph with nodes $a_1,\dots, a_n$ and $b_1, \dots, b_m$ and arcs $a_ib_j$ for all $i\in [n]$ and $j\in [m]$.
Note that $\Knm$ is triangle-free.
We abbreviate its graphical zonotope by $\mathdefn{\ZGKnm}\eqdef\polZ_{K_{n, m}}$.

We define two new polytopes by intersecting $\ZGKnm$ with some half-spaces. 
Precisely, recall that $\ZGKnm$ is embedded in $\R^{n + m}$ whose coordinates are labeled by $a_i$, $i\in [n]$ and $b_j$, $j\in [m]$, and define:
\begin{align*}
\mathdefn{\Pnm} &\eqdef \ZGKnm \cap \Big\{\b x ~;~ \sum\nolimits_{j=1}^m x_{b_j} \leq nm-1\Big\},&&\text{and } &\mathdefn{\Qnm} &\eqdef \Pnm \cap \Big\{\b x ~;~ \sum\nolimits_{i=1}^n x_{a_i} \leq nm - 1\Big\}.
\end{align*}
Note that they are obtained by deeply truncating one and two opposite vertices of $\ZGKnm$, respectively. (The deep truncation of polytope at a vertex~$\b v$ is defined as the intersection with the halfspace that does not contain~$\b v$ but contains all the neighbors of~$\b v$ in its boundary, this notion is explored in details in \Cref{ssec:DeepTruncationsOfParallelogramicZonotopes}.)

\begin{example}\label{exmp:SmallDimPandQ}
	We describe $\Pnm$ and $\Qnm$ for $n + m \leq 4$.
	Since $\dim\ZGKnm = n + m - 1$, these are the ones we can picture in 3 dimensions.
	We summarized their properties in the following table, see \Cref{fig:SmallDimPandQ}.
	\begin{center}

	\begin{tabular}{c|cccc|ccc}
     & \multicolumn{4}{c|}{$\Pnm$} & \multicolumn{3}{c}{$\Qnm$} \\
		$ (m,n)$ & (1,1) & (1,2) & (1,3) & (2,2) & (1,2) & (1,3) & (2,2)  \\ \hline
		name  & point & triangle & \emph{strawberry}
		&\emph{persimmon} & segment & octahedron & cuboctahedron \\
		\text{indecomposable} & \yes & \yes & \yes & \yes & \yes & \yes & \no \\
		\text{matroid polytope} & \yes & \yes & \no  &\no & \yes & \yes & \no
	\end{tabular}
\end{center}
	
The first row answers ``Is this polytope indecomposable?'', and the second ``Is this polytope (normally equivalent to) a matroid polytope (\ie does it have $0/1$-coordinates)?'' (\yes~ is ``yes'', and \no~ is ``no'').

The polytopes $\Pnm[3][1]$, $\Qnm[3][1]$, and $\Pnm[2][2]$, $\Qnm[2][2]$ are depicted in \Cref{fig:SmallDimPandQ}. The polytope $\Pnm[3][1]$ (nicknamed \emph{strawberry}) is obtained by truncating a vertex of a $3$-cube and has $f$-vector $(7, 12, 7)$; and $\Pnm[2][2]$ (nicknamed \emph{persimmon}) has $f$-vector $(13, 24, 13)$.
As illustrated in the figure, $\ELD[{\Pnm[3][1]}]$, $\ELD[{\Qnm[3][1]}]$ and $\ELD[{\Pnm[2][2]}]$ contain connected subgraphs whose associated edges touch all the vertices: by \Cref{thm:mainthmframeworks}, $\Pnm[3][1]$, $\Qnm[3][1]$ and $\Pnm[2][2]$ are indecomposable.
The cuboctahedron $\Qnm[2][2] $ is decomposable as the Minkowski sum of a regular tetrahedron with its central reflection.
	
		\begin{figure}[htpb]
		\centering
		\includegraphics[width=0.2\linewidth]{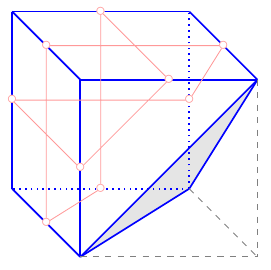}
        \includegraphics[width=0.2\linewidth]{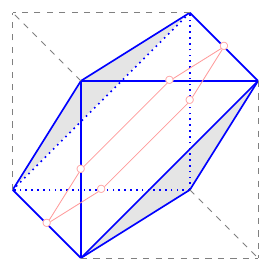}
		\includegraphics[width=0.28\linewidth]{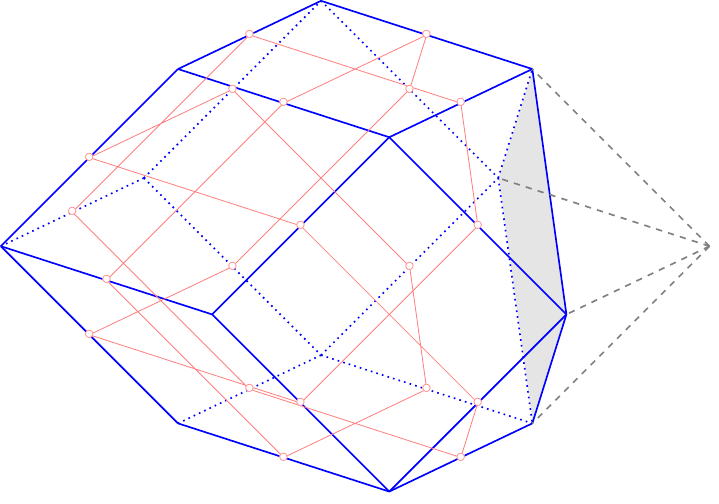}
		\includegraphics[width=0.28\linewidth]{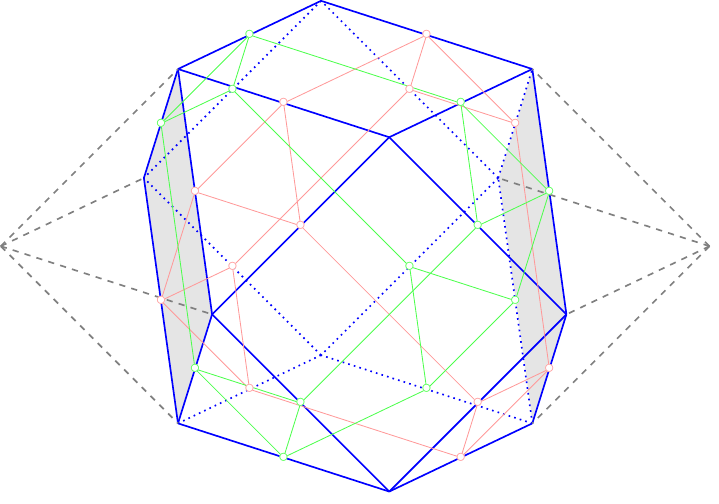}
        \caption[Truncated graphical zonotopes of complete bipartite graphs $\Pnm$ and $\Qnm$]{(Left to right) strawberry $\Pnm[3][1]$, octahedron $\Qnm[3][1]$, persimmon $\Pnm[2][2]$, cuboctahedron $\Qnm[2][2]$.
        Each polytope is in \textcolor{blue}{blue}, while the \textcolor{gray}{gray} dashed edges outline $\ZGKnm$.
        In \textcolor{OrangeRed}{red}, the subgraphs of $\ELD$ obtained by linking opposite edges of (some) parallelograms and edges of (some) triangles (these are the subgraphs from the proof of \Cref{thm:PandQareIndecomposable}).
        The graph $\ELD[{\Qnm[2][2]}]$ has two connected components: the clique on the \textcolor{OrangeRed}{red} subgraph, and the clique on the \textcolor{green}{green} one.
        \vspace{-0.75cm}}
		\label{fig:SmallDimPandQ}
	\end{figure}
    
	Note that $\Pnm[3][1]$ and $\Pnm[2][2]$ are self-dual and have exactly 4 triangular faces.
	Smilansky proved in \cite[Cor. 6.7 \& 6.8]{Smilansky1987} that every indecomposable $3$-polytope has at least 4 triangular faces, and has (weakly) more facets than vertices: $\Pnm[3][1]$ and $\Pnm[2][2]$ are extremal in this sense.
\end{example}

We denote by \defn{$\b v_{n\to m}$} the vertex of $\ZGKnm$ associated to the acyclic orientation \defn{$\rho_{n\to m}$} given by $a_i \to b_j$ for all $i\in [n]$ and $j\in [m]$; and by \defn{$\b v_{n \leftarrow m}$} that associated to \defn{$\rho_{n \leftarrow m}$} given by $b_j \to a_i$ for all $i\in [n]$ and $j\in [m]$.
An acyclic orientation of $\Knm$ is \defn{almost left-right} (respectively \defn{almost right-left}) if it is obtained from $\rho_{n \to m}$ (respectively from $\rho_{n \leftarrow m}$) by reversing the orientation of one arc called its \defn{reversed arc}.

Note that, by construction, the supporting hyperplanes for facets of $\Pnm$ and $\Qnm$ are either supporting hyperplanes for facets of $\ZGKnm$ or the hyperplane $\bigl\{\sum_{j=1}^m x_{b_j} = nm-1\bigr\}$ (and the hyperplane $\bigl\{\sum_{i=1}^n x_{a_i} = nm - 1\bigr\}$ for $\Qnm$). 
We distinguish the facet \defn{$\pol[F]_{n \to m}$} of $\Pnm$ and $\Qnm$ supported by the hyperplane $\bigl\{\sum_{j=1}^m x_{b_j} = nm-1\bigr\}$, and the facet \defn{$\pol[F]_{n \leftarrow m}$} of $\Qnm$ supported by $\bigl\{\sum_{i=1}^n x_{a_j} = nm-1\bigr\}$.
The facet $\pol[F]_{n \to m}$ is the convex hull of the vertices indexed by almost left-right orientations, and the facet $\pol[F]_{n \leftarrow m}$ by those indexed by almost right-left orientations. They are affinely isomorphic to the Cartesian product of simplices $\simplex[n-1]\times \simplex[m-1]$ (because they are the vertex figure of $\ZGKnm$ at $\b v_{n\to m}$ and $\b v_{n \leftarrow m}$, respectively).

If $\b v$ is a vertex of a polytope $\pol$, we denote \defn{$\pol\ssm \b v$} the polytope obtained as the convex hull of all the vertices of $\pol$ except $\b v$.
For notational convenience, we also define $\pol\ssm \b v = \pol$ if $\b v\notin \pol$.

\begin{lemma}\label{lem:VerticesFacesOfPandQ}
Let $n, m\geq 1$ with $nm > 2$. The vertices of $\Pnm$ are $\Qnm$ are resp. $\verts(\Pnm)=\verts(\ZGKnm) \ssm \{\b v_{n \to m}\}$ and $\verts(\Qnm)=\verts(\ZGKnm) \ssm \{\b v_{n \to m}, \b v_{n \leftarrow m}\}$. That is, $\Pnm = \ZGKnm \ssm \b v_{n\to m}$ and $\Qnm = \ZGKnm \ssm (\b v_{n\to m}, \, \b v_{n \leftarrow m})$.
\end{lemma}	
\begin{proof}
	This follows from the fact that the hyperplane $\bigl\{\sum_{j=1}^m x_{b_j} = nm-1\bigr\}$ (resp. $\bigl\{\sum_{i=1}^m x_{a_i} = nm-1\bigr\}$) goes through all the neighbors of $\b v_{n \to m}$ (resp. $\b v_{n \leftarrow m}$) and does not intersect any edge of $\ZGKnm$ in the interior. 
\end{proof}

\begin{lemma}\label{lem:OtherFacesOfPandQ}
	Let $n, m\geq 1$ with $nm > 2$.	The faces of $\Pnm$ and $\Qnm$ are
	 \begin{compactenum}[(i)]
	 	\item the faces of $\pol[F]_{n \to m}$ (resp. of $\pol[F]_{n \to m}$ or $\pol[F]_{n \leftarrow m}$) or 
	 	\item the faces of the form $\pol[F]\ssm \b v_{n \to m}$ (resp. $\pol[F]\ssm (\b v_{n \to m},\b v_{n \leftarrow m})$) for a face $\pol[F]$ of $\ZGKnm$.
 	\end{compactenum}
In particular, the edges of $\Pnm$ (resp. $\Qnm$) are either:
		\begin{compactenum}[(i)]
				\item edges of $\ZGKnm$  not containing $\b v_{n \to m}$ (resp. $\b v_{n \to m}$ nor $\b v_{n \leftarrow m}$), or 
				\item edges $\b v_{\rho_1} \b v_{\rho_2}$ for two almost left-right orientations (resp. two almost left-right or almost right-left orientations) $\rho_1$, $\rho_2$ whose reversed arcs share an endpoint.
			\end{compactenum}
\end{lemma}
\begin{proof}
	The faces of $\Pnm$ are either contained in the hyperplane $\bigl\{\sum_{j=1}^m x_{b_j} = nm-1\bigr\}$, in which case we get a face of $\pol[F]_{n \to m}$, or a face $\pol[F]$ of $\Pnm$ intersected with $\bigl\{\sum_{j=1}^m x_{b_j} \leq nm-1\bigr\}$, yielding $\pol[F]\ssm \b v_{n \to m}$ by \Cref{lem:VerticesFacesOfPandQ}. 
	
		For the description of the edges of the distinguished facets, note that the only edges of $\Pnm$ that are not edges of $\ZGKnm$ come from the intersection of the hyperplane $\bigl\{\sum_j x_{b_j} = nm-1\bigr\}$ with $2$-faces of $\ZGKnm$.
		An edge will appear between two neighbors $\b v_{\rho_1}, \b v_{\rho_2}$ of $\b v_{n\to m}$ that share a common $2$-face. The vertex $\b v_\rho$ is neighbor of $\b v_{n\to m}$ if and only if $\rho$ is an almost left-right orientation.
		Denoting $g_1, g_2$ the reversed arcs of $\rho_1$, $\rho_2$, then $\b v_{\rho_1}$, $\b v_{\rho_2}$ and $\b v_{n\to m}$ belong to a common $2$-face when $\rho_1$ and $\rho_2$ induce the same orientation on $\contrG[g_1, g_2]$, which happens precisely when $g_1$ and $g_2$ share an endpoint.
        The case of $\Qnm$ is similar.
\end{proof}

Now, we focus on $\Pnm$ and $\Qnm$ for $n, m\geq 1$, $(n, m) \ne (1, 1)$. Recall from the introduction that a \defn{deformed $N$-permutahedron} is a deformation of the permutahedron $\permuto[N]$, equivalently, it is a polytope $\pol~\subset~\R^N$ whose edges are parallel to $\b e_i - \b e_j$ for some $i, j\in [N]$, see \cite[Definition 6.1]{Postnikov2009}.
In the upcoming sections, we will prove theorems that combined give the following key result.

\begin{theorem}\label{thm:NewRAYS}
For integers $n, m\geq 1$ with $n+m \geq 5$, the polytopes $\Pnm$ and $\Qnm$ are indecomposable deformed $(n+m)$-permutahedra that are not isomorphic to matroid polytopes.
\end{theorem}

\begin{proof}
Combine \Cref{thm:PandQareGP} (deformed permutahedra), \Cref{thm:PandQareIndecomposable} (indecomposable), and \Cref{thm:PandQnotMatroidPolytopes} (not matroid polytopes).
\end{proof}

As a consequence, we get an infinite new family of rays of the submodular cone.
\begin{corollary}\label{cor:NewMIGP}
For $N\geq 4$, there exist (at least) $2\lfloor\frac{N-1}{2}\rfloor$ non-isomorphic indecomposable $N$-deformed permutahedra which are not matroid polytopes.
\end{corollary}

\begin{proof}
For $N = 4$, see \Cref{exmp:SmallDimPandQ}.
According to \Cref{thm:NewRAYS}, both $\Pnm$ and $\Qnm$ for $n + m = N \geq 5$ are indecomposable deformed $N$-permutahedra which are not matroid polytopes.
These polytopes are not pairwise isomorphic for $1\leq n \leq \lfloor\frac{N-1}{2}\rfloor$ and $m = N - n$, as the number of facets of $\ZGKnm$ is the number of connected subgraphs of $\Knm$ whose complement is also connected, \ie $2^N + N + 2 - (2^n + 2^{N-n})$, which is different for each $n$.
\end{proof}



\begin{remark}\label{rmk:NotGeneratedByLPPmethod}
Another construction for a large family of indecomposable deformed permutahedra is given in \cite{LohoPadrolPoullot2025RaysSubmodularCone}.
We show that the present construction does not boil down to an application of the latter.

Fix $n, m\geq 3$ and some $i \in [n]$.
The vertices of the face $\left(\ZGKnm\right)^{-\b e_{a_i}}$ are in bijection with acyclic orientations of $\Knm$ such that all the arcs $ix$ for $x\in [m]$ are oriented from $i$ towards $x$.
Hence, the face $\left(\ZGKnm\right)^{-\b e_{a_i}}$ is isomorphic to $\ZGKnm[n-1]$ (more precisely, it is a translation of an embedding of $\ZGKnm[n-1]$ in $\R^{n+m}$).
Accordingly, for any $i\in [n]$, the face $\left(\Pnm\right)^{-\b e_{a_i}}$ is isomorphic to $\ZGKnm[n-1]$, while the face $\left(\Pnm\right)^{+\b e_{a_i}}$ is isomorphic to $\Pnm[n-1]$.
Similarly, for any $j\in [m]$, the face $\left(\Pnm\right)^{-\b e_{b_j}}$ is isomorphic to $\Pnm[n][m-1]$, while the face $\left(\Pnm\right)^{+\b e_{b_j}}$ is isomorphic to $\ZGKnm[n][m-1]$.
Finally, recall that $\ZGKnm$ is \textbf{not} indecomposable for $n, m\geq 2$, since it is a Minkowski sum.

This shows that $\Pnm$ cannot be constructed with the tools from \cite{LohoPadrolPoullot2025RaysSubmodularCone}, which always give rise to polytopes that have a pair of antipodal indecomposable facets.

Now, let \defn{$t_N$} denote the number of indecomposable deformed $N$-permutahedra which are not pairwise normally equivalent.
Let \defn{$a_N$} be the number of these polytopes $\pol$ such both $\pol^{+\b e_N}$ and $\pol^{-\b e_N}$ are indecomposable.
In \cite[Question 3.32]{LohoPadrolPoullot2025RaysSubmodularCone}, the authors ask about the behavior of the ratio $\frac{a_N}{t_N}$ when $N \to+\infty$.
In particular, by \Cref{thm:NewRAYS} and what we just argued, for all $n, m\geq 3$, the polytopes $\Pnm$ are deformed $(n+m)$-permutahedra, but 
either $\pol^{+\b e_N}$ or $\pol^{-\b e_N}$ is not indecomposable.
Hence, we proved: $t_N - a_N \geq \left\lfloor\frac{N-1}{2}\right\rfloor$.

For $\Qnm$, the opposite faces $\left(\Qnm\right)^{-\b e_i}$ and $\left(\Qnm\right)^{+\b e_i}$ are equivalent to $\Pnm[n-1]$ or to $\Pnm[n][m-1]$, depending whether $i \in [n]$ or $i \in [m]$.
Hence, $\Qnm$ could be constructed by the inductive methods of \cite[Proposition 3.8]{LohoPadrolPoullot2025RaysSubmodularCone}. However, they do not belong to the families constructed in \cite[Corollary 3.20]{LohoPadrolPoullot2025RaysSubmodularCone} because the polytopes $\Pnm$ cannot be constructed by the aforementioned inductive methods.
\end{remark}

\subsubsection{Deformed permutahedra}
The deformations of $\permuto[N]$ are characterized in~\cite[Definition~6.1]{Postnikov2009} by having all its edges parallel to $\b e_i - \b e_j$.
Using \Cref{lem:OtherFacesOfPandQ}, we verify that $\Pnm$ and $\Qnm$ satisfy this condition.
	\begin{theorem}\label{thm:PandQareGP}
		The polytopes $\Pnm$ and $\Qnm$ are deformed permutahedra.
	\end{theorem}
	
	\begin{proof}
		
		The edges of $\Pnm$ are given in \Cref{lem:OtherFacesOfPandQ}. 
		$\ZGKnm$ is a deformed permutahedron, and hence the edges of $\Pnm$ of type \emph{(i)} are in direction $\b e_i - \b e_j$ for some $i, j \in [n]$.
		
		For an edge of the form \emph{(ii)}, let $\rho_1, \rho_2$ be two almost left-right orientations.
		We have $\b v_{\rho_1} - \b v_{\rho_2} = (\b v_{\rho_1} - \b v_{n\to m}) - (\b v_{\rho_2} - \b v_{n\to m})$. 	The reversed arcs of $\rho_1$ and $\rho_2$ share an endpoint, say they are of the form $a_ib_k$ and $a_jb_k$.
		Hence we have $\b v_{\rho_1} - \b v_{n\to m} = \b e_{a_i} - \b e_{b_k}$ and $\b v_{\rho_1} - \b v_{n\to m} = \b e_{a_j} - \b e_{b_k}$.
		Consequently, $\b v_{\rho_1} - \b v_{\rho_2} = \b e_{a_i} - \b e_{a_j}$ is indeed in the correct direction.
		We have proven that all the edges of $\Pnm$ are parallel to $\b e_i - \b e_j$ for some $i, j$, thus $\Pnm$ is a deformed permutahedron.
		The edges of $\Qnm$ are either edges of $\Pnm$ or of $\Pnm[m][n]$, thus $\Qnm$ is also a deformed permutahedron.
	\end{proof}

Unfortunately, our construction only gives deformed permutahedra for complete bipartite graphs.

\begin{theorem}
For a connected graph $G$ and a vertex $\b v$ of $\ZG$, if $\ZG\ssm \b v$ is a deformed permutahedron, then $G$ is a complete bipartite graph $\Knm$ and $\b v = \b v_{n\to m}$ or $\b v = \b v_{m\to n}$.
\end{theorem}

\begin{proof}
Let $G = (V, E)$ be a connected graph,
and $\rho$ an acyclic orientation of $G$, such that $\b v_\rho$ admits a deep truncation and $\ZG\ssm \b v_\rho$ is a deformed permutahedra.
Equivalently, for all acyclic orientations $\rho_1, \rho_2$ such that $\rho_1$ and $\rho_2$ are obtained by reversing the direction of one arc in $\rho$, if $\b v_\rho$, $\b v_{\rho_1}$ and $\b v_{\rho_2}$ belong to a common $2$-face of $\ZG$, then $\b v_{\rho_1} - \b v_{\rho_2} = \b e_i - \b e_j$ for some $i, j\in V$.

We first prove that there is no directed path of length 2 in $\rho$, which implies that $G$ is bipartite (between sources and sinks in $\rho$), and then prove that all sources of $\rho$ are linked to all sinks of $\rho$, which concludes.
We assume that $V$ is labeled according to a topological order of $\rho$, \ie if $u\to v$ in $\rho$, then $u\leq v$.

Suppose there is a directed path of length 2, namely $u\to v \to w$.
We can assume that $v$ is maximal among the incoming neighbors of $w$ with in-degree at least 1, and that $u$ is maximal among the incoming neighbors of $v$.
In particular, there is no directed path from $u$ to $v$ (other than the arc $u\to v$ itself), otherwise the last vertex arc of this path would be $x \to v$ with $u < x$.
Similarly, there is no directed path from $v$ to $w$ in~$\rho$.
Consider $\rho_{uv}$ the orientation obtained from~$\rho$ by changing the orientation of the arc $u\to v$, and similarly~$\rho_{vw}$.
Both are acyclic orientations of $G$ (otherwise there would be two paths from $u$ to $v$ in $\rho$, respectively from $v$ to $w$), so $\b v_{\rho_{uv}}$ and $\b v_{\rho_{vw}}$ are vertices of $\ZG$ adjacent to $\b v_{\rho}$.
If $\b v_{\rho}$, $\b v_{\rho_{uv}}$ and $\b v_{\rho_{vw}}$ belong to a common 2-face of $\ZG$, then $\b v_{\rho_{uv}}\b v_{\rho_{vw}}$ is an edge of $\ZG\ssm \b v_{\rho}$.
However, the direction of $\b v_{\rho_{uv}}\b v_{\rho_{vw}}$ is $\b e_u - 2 \b e_v + \b e_w$, so in this case, $\ZG\ssm\b v_{\rho}$ is not a deformed permutahedron.

Consequently, $\b v_{\rho}$, $\b v_{\rho_{uv}}$ and $\b v_{\rho_{vw}}$ do not belong to a common 2-face of $\ZG$, meaning there is a path (of length at least 2) from $u$ to $w$ in $\rho$.
Let $x$ be minimal among the vertices such that $x\ne v$ and there is a path from $u$ to $w$ in $\rho$ whose first arc is $u\to x$.
Let $\rho_{ux}$ be obtained by reversing the direction of the arc $u\to x$ in $\rho$.
There is no path from $u$ to $x$ other than the edge $u\to x$, as $x$ in minimal: $\rho_{ux}$ is acyclic and $\b v_{\rho_{ux}}$ is a neighbor of $\b v_\rho$.
There is no directed path from $v$ to $x$ in $\rho$ (otherwise, there will be a directed path from $v$ to $w$ going through $x$).
There is no directed path from $w$ to $u$ (otherwise there would be a cycle using $u\to v\to w$).
Thus, contracting both $u\to x$ and $v\to w$ leads to an acyclic orientation, implying that $\b v_\rho$, $\b v_{\rho_{ux}}$ and $\b v_{\rho_{v w}}$ belong to a common 2-face of $\ZG$.
But, this implies that $\b v_{\rho_{ux}}\b v_{\rho_{vw}}$ is an edge of $\ZG\ssm \b v_\rho$: as its direction is $\b e_x - \b e_u + \b e_v - \b e_w$, this contradicts the fact that $\ZG\ssm\b v_\rho$ is a deformed permutahedron.

Therefore, there is no directed path of length 2 in $\rho$: each vertex is either a source or a sink.
Let $V = A\sqcup B$ be this bi-partition, and suppose that $G$ is not complete bipartite.
Let $a\in A$, $b\in B$ such that $ab \notin E$.
As $G$ is connected, let $b'\in B$ be a neighbor of $a$, and $a'\in A$ a neighbor of $b$.
The orientations $\rho_{ab'}$ and $\rho_{a'b}$ are acyclic because $a$ and $a'$ are sources.
Contracting the arcs $ab'$ and $a'b$ in $\rho$ does not lead to a cycle because $ab\notin E$ (so there is no path from $a$ to $b$).
Hence, $\b v_{\rho_{ab'}}$, $\b v_{\rho_{a'b}}$ and $\b v_{\rho}$ lie in a common 2-face, so $\b v_{\rho_{ab'}}\b v_{\rho_{a'b}}$ is an edge in $\ZG\ssm\b v_\rho$, in direction $\b e_b - \b e_{a'} + \b e_a - \b e_{b'}$.
This contradicts the fact that $\ZG\ssm\b v_{\rho}$ is a deformed permutahedron: $G$ is a complete bipartite graph $\Knm$, and $\rho\in\{\rho_{n\to m},\, \rho_{m\to n}\}$.
\end{proof}

\subsubsection{Indecomposability}
We will use \Cref{thm:mainthmframeworks} to prove indecomposability.

For an arc $uv$ of $\Knm$, we call \defn{$uv$-edges} the edges of $\ZGKnm$, $\Pnm$ or $\Qnm$ in direction $\b e_u-\b e_v$, and \defn{$a$-edge} an $uv$-edge for some arc $uv$ in $\Knm$, see \Cref{fig:ProofIndecomposabilityPnmQnm} (note that the edges of the distinguished facets of $\Pnm$ and $\Qnm$ are precisely their non-$a$-edges). 
Let \defn{$\zELD[uv]$} (resp. \defn{$\zELD[uv]^{\dashcircle}$} and \defn{$\zELD[uv]^{\circledashcircle}$}) be the subgraph of $\ELD[\ZGKnm]$ (resp. $\ELD[\Pnm]$ and $\ELD[\Qnm]$) whose nodes are the $uv$-edges and where two $uv$-edges are linked by an arc if they are opposite in a parallelogram of $\ZGKnm$ (resp. of $\Pnm$ and $\Qnm$).

\begin{lemma}\label{lem:EDuv}
The graph \defn{$\zELD[uv]$} is isomorphic to the $1$-skeleton of the graphical zonotope of the contracted graph ${\contrG[uv][\Knm]}$.
\end{lemma}

\begin{proof}
Following the combinatorial description of faces of $\ZGKnm$ in terms of ordered partitions from \Cref{sec:graphicalzonotopes}, we see that $uv$-edges of $\ZGKnm$ are in bijection with 
vertices of $\ZG[{\contrG[uv][\Knm]}]$, and the $2$-faces of $\ZGKnm$ containing $uv$-edges (which are parallelograms) are in bijection with the edges of $\ZG[{\contrG[uv][\Knm]}]$.
\end{proof}

\begin{figure}
    \centering
    \includegraphics[width=0.99\linewidth]{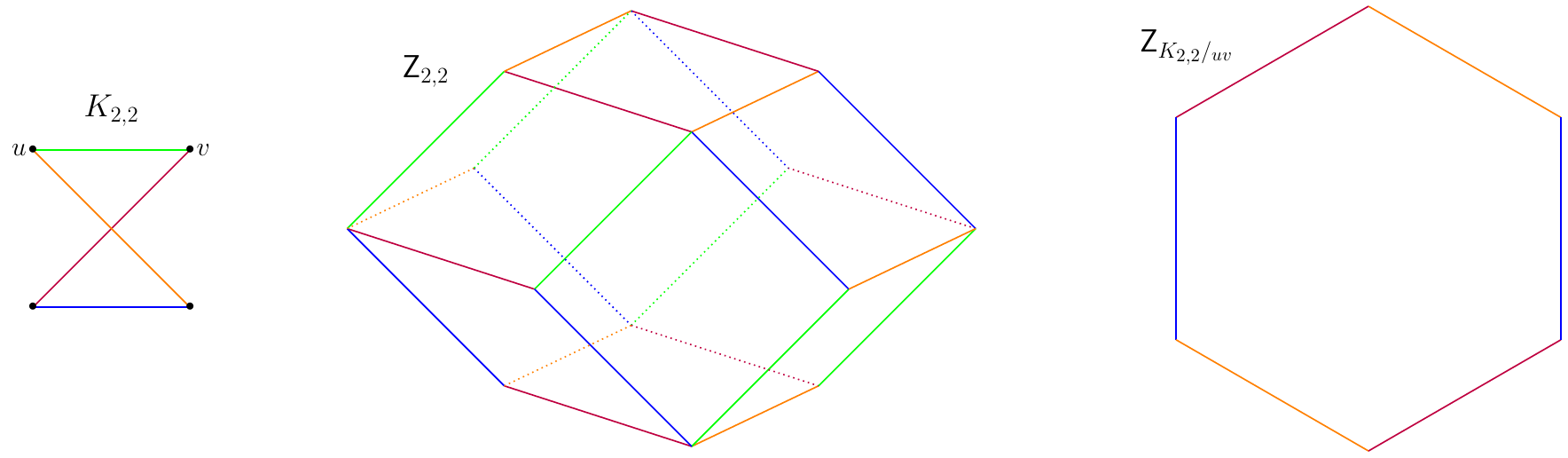}
    \caption{(Left) The complete bipartite graph $\Knm[2][2]$.
    (Middle) The corresponding graphical zonotope $\ZGKnm[2][2]$, the \textcolor{green}{green} edges are the $uv$-edges for the arc $uv$ drawn in \textcolor{green}{green} on $\Knm[2][2]$ (and reciprocally for the other colors).
    (Right) As proven in \Cref{lem:EDuv}, the graph $\zELD[uv]$ is isomorphic to the $1$-skeleton of the graphical zonotope $\ZG[{\contrG[uv][{\Knm[2][2]}]}]$ which can be obtained by projecting the zonotope $\ZGKnm[2][2]$ along the direction of its \textcolor{green}{green} edges.
    Note that the graph $\contrG[uv][{\Knm[2][2]}]$ is a triangle: its graphical zonotope is indeed a regular hexagon.}
    \label{fig:ProofIndecomposabilityPnmQnm}
\end{figure}

\begin{theorem}\label{thm:PandQareIndecomposable}
$\Pnm$ is indecomposable for all $n, m\geq 1$, and $\Qnm$ is indecomposable  for all $n, m\geq 1$ with $(n, m) \neq (2, 2)$.
\end{theorem}

\begin{proof}
We fix $n, m$ with $n+m \geq 5$. For the lower dimensional cases, see \Cref{exmp:SmallDimPandQ}.

First note that, as all the 2-faces of $\ZGKnm$ are parallelograms, by \Cref{ex:parallelogram}, the graphs $\zELD[uv]$, $\zELD[uv]^{\dashcircle}$, and $\zELD[uv]^{\circledashcircle}$
are subgraphs of $\ELD[\ZGKnm]$, $\ELD[\Pnm]$, and $\ELD[\Qnm]$, respectively. 

By \Cref{lem:EDuv}, the graph $\zELD[uv]$ is isomorphic to the $1$-skeleton of a $(n+m-2)$-dimensional polytope, see \Cref{fig:ProofIndecomposabilityPnmQnm}.
In particular, this graph is $(n+m-2)$-connected, by Balinksi's theorem (see \cite[Section 3.5]{Ziegler-polytopes}).
Moreover, as $n+m-2\geq 3$, removing 1 or 2 vertices to this graph does not break its connectivity, and hence, $\zELD[uv]^{\dashcircle}$ and $\zELD[uv]^{\circledashcircle}$ are also connected.

Now, let $uv$, $uw$ be two arcs of $\Knm$ sharing an endpoint, \ie they form an arc in the line graph of~$\Knm$.
Let $\rho_1$, $\rho_2$ be their associated almost left-right orientations.
The edge between $\b v_{\rho_1}$ and $\b v_{\rho_2}$ shares a triangular $2$-face with an $uv$-edge and a $uw$-edge, see \Cref{fig:SmallDimPandQ}.
As any triangle is indecomposable, \Cref{cor:indecsubframeclique} ensures that $\zELD[uv]^{\dashcircle}$ and $\zELD[uw]^{\dashcircle}$ are connected together.
As $\Knm$ is connected, so is its line-graph, thus the subgraph of $\ELD[\Pnm]$ induced on  all $a$-edges is connected.
The $a$-edges form a dependent subset, and every vertex of $\Pnm$ appears as endpoint of one of these edges in $\G(\Pnm)$ (because any $\rho\in\AO(G)$ has a re-orientable arc): by \Cref{thm:mainthmframeworks}, we get that $\Pnm$ is indecomposable.

The same reasoning holds for $\Qnm$.
\end{proof}




This allows for refuting a conjecture by Smilansky from 1987:

\begin{conjecture}[{\cite[Conjecture~6.12]{Smilansky1987}}]
If a $d$-polytope $\pol$ satisfies 
\begin{equation}\label{eq:SmilanskyConj}\verts > 1 +\, \binom{\facets - 1-\lfloor{d}/{2}\rfloor }{\facets-d} \,+\, \binom{\facets - 1-\lfloor({d+1})/{2}\rfloor }{\facets-d},\end{equation}
where $\verts\eqdef|\verts(\pol)|$ and $\facets\eqdef|\facets(\pol)|$ are the number of vertices and facets of $\pol$, then $\pol$ is decomposable.
\end{conjecture}

For $d = 3$, condition \eqref{eq:SmilanskyConj} translates to $V > F$.
Smilansky indeed proved that if a polytope has strictly more vertices than facets, then it is decomposable, see \cite[Theorem 6.7 \& Figure 1]{Smilansky1987}.

For $d = 4$, condition \eqref{eq:SmilanskyConj} translates to $\verts\geq 2\facets -4$.
Both $\Pnm[1][4]$ and $\Pnm[2][3]$ are indecomposable and $4$-dimensional.
Yet their $f$-vectors are respectively $(\textbf{15}, 34, 28, \textbf{9})$ and $(\textbf{45}, 111, 89, \textbf{23})$, both fulfilling \eqref{eq:SmilanskyConj}.

Our construction does not directly give counterexamples to Smilansky's conjecture for $d \geq 5$, but we believe they exist.
The interesting question of constructing counterexamples in dimensions $\geq 5$ is left open.

\subsubsection{Not matroid polytopes}
The \defn{matroid polytope} of a matroid is the convex hull of all indicator vectors of its bases. Matroid polytopes are characterized by being the deformed permutahedra with $0/1$-coordinates \cite{GelfandGoreskiMacPhersonSerganova}.
We can show that $\Pnm$ and $\Qnm$ are not normally equivalent to matroid polytopes (except for some sporadic cases) by providing vertices whose first coordinate takes at least $3$ different values.

\begin{lemma}\label{lem:NotMatroidPolytope}
If an indecomposable polytope $\pol\subset\R^d$ 
is normally equivalent to a matroid polytope then for all $1\leq i\leq d$, the $i^\text{th}$ coordinates of its vertices take at most 2 different values.
\end{lemma}

\begin{proof}
As $\pol$ is indecomposable, its normally equivalent polytopes are of the form $\b\lambda\pol + \b t$ for $\b\lambda > 0$ and $\b t\in\R^d$.
The $i^\text{th}$ coordinates of the vertices of $\b\lambda\pol + \b t$ are $\b\lambda v_i + t_i$ for $\b v\in \verts(\pol)$.
As $\b\lambda \ne 0$, the application $x \mapsto \b\lambda x + t_i$ is bijective on $\R$, thus $\{v_i ~;~ \b v \text{ vertex of } \pol\}$ has the same cardinal as $\{w_i ~;~ \b w \text{ vertex of } \b\lambda\pol + \b t\}$.
If $\pol$ is normally equivalent to a matroid polytope, there exist $\b\lambda, \b t$ such that the second set is of cardinal $\leq 2$.
\end{proof}

\begin{theorem}\label{thm:PandQnotMatroidPolytopes}
For $1\leq n\leq m$, $\Pnm$ (resp. $\Qnm$) is normally equivalent to a matroid polytope if and only if $n = 1$ and $m \in \{1, 2\}$  (resp. $n = 1$ and $m \in \{1, 2, 3\}$).
\end{theorem}
	
\begin{proof}		
\Cref{exmp:SmallDimPandQ} handles $n = 1, m \leq 3$ and $(n, m) = (2, 2)$: note that the cuboctahedron $\Qnm[2][2]$ does not have any pair of parallel faces containing all the vertices, hence is not a deformation of a matroid polytope.

For the remaining cases, $\Pnm$ and $\Qnm$ are indecomposable deformed permutahedra by \Cref{thm:PandQareGP,thm:PandQareIndecomposable}.
Recall that the $i^\text{th}$ coordinate of $\b v_\rho$ for $\rho\in\AO(G)$ is $d_{in}(i, \rho) = \bigl|\{j ~;~ j\to i \in \rho\}\bigr|$.

Fix $n, m \geq 2$ and $(n, m)\ne (2, 2)$.
We construct 3 acyclic orientations $\rho_1, \rho_2, \rho_3 \in \AO(\Knm)\ssm\{\rho_{n\to m}, \rho_{n\leftarrow m}\}$, such that $a_1$ has respective in-degree $0$, $1$ and $2$.
\Cref{lem:NotMatroidPolytope} thus concludes.
Consider the subgraph of $\Knm$ induced on $\{a_1, a_2, b_1, b_2\}$: it is isomorphic to $\Knm[2][2]$.
The orientations $\rho_1$, $\rho_2$, $\rho_3$ are obtained constructing $\rho'_1$, $\rho'_2$, $\rho'_3$ as drawn, and orienting other arcs $a_i \to b_j$ in $\Knm$.
		
		\begin{center}
			\begin{tikzpicture}[decoration={markings, mark=at position 0.3 with {\arrow{>}}}]
				\begin{scope}[shift = {(0, 0)}]
					\coordinate (1a) at (0, 0);
					\coordinate (2a) at (0, -1);
					\coordinate (1b) at (1, 0);
					\coordinate (2b) at (1, -1);
					
					\draw (1a) node{$\bullet$};
					\draw (2a) node{$\bullet$};
					\draw (1b) node{$\bullet$};
					\draw (2b) node{$\bullet$};
					
					\draw (1a) node[left]{$a_1$};
					\draw (2a) node[left]{$a_2$};
					\draw (1b) node[right]{$b_1$};
					\draw (2b) node[right]{$b_2$};
					
					\draw[postaction={decorate}] (1a) -- (1b); 
					\draw[postaction={decorate}] (1a) -- (2b); 
					\draw[postaction={decorate}] (1b) -- (2a); 
					\draw[postaction={decorate}] (2a) -- (2b); 
					
					\draw (0.5, -1.5) node{$\rho'_1$};
				\end{scope}
				
				\begin{scope}[shift = {(4, 0)}]
					\coordinate (1a) at (0, 0);
					\coordinate (2a) at (0, -1);
					\coordinate (1b) at (1, 0);
					\coordinate (2b) at (1, -1);
					
					\draw (1a) node{$\bullet$};
					\draw (2a) node{$\bullet$};
					\draw (1b) node{$\bullet$};
					\draw (2b) node{$\bullet$};
					
					\draw (1a) node[left]{$a_1$};
					\draw (2a) node[left]{$a_2$};
					\draw (1b) node[right]{$b_1$};
					\draw (2b) node[right]{$b_2$};
					
					\draw[postaction={decorate}] (1b) -- (1a); 
					\draw[postaction={decorate}] (1a) -- (2b); 
					\draw[postaction={decorate}] (2a) -- (1b); 
					\draw[postaction={decorate}] (2a) -- (2b);
					
					\draw (0.5, -1.5) node{$\rho'_2$};
				\end{scope}
				
				\begin{scope}[shift = {(8, 0)}]
					\coordinate (1a) at (0, 0);
					\coordinate (2a) at (0, -1);
					\coordinate (1b) at (1, 0);
					\coordinate (2b) at (1, -1);
					
					\draw (1a) node{$\bullet$};
					\draw (2a) node{$\bullet$};
					\draw (1b) node{$\bullet$};
					\draw (2b) node{$\bullet$};
					
					\draw (1a) node[left]{$a_1$};
					\draw (2a) node[left]{$a_2$};
					\draw (1b) node[right]{$b_1$};
					\draw (2b) node[right]{$b_2$};
					
					\draw[postaction={decorate}] (1b) -- (1a); 
					\draw[postaction={decorate}] (2b) -- (1a); 
					\draw[postaction={decorate}] (2a) -- (1b); 
					\draw[postaction={decorate}] (2b) -- (2a); 
					
					\draw (0.5, -1.5) node{$\rho'_3$};
				\end{scope}
			\end{tikzpicture}
		\end{center}
        \vspace{-0.33cm}
        
For $n = 1$, the graph $\Knm$ is a star on $m$ vertices.
An acyclic orientation amounts to choosing $X\subseteq[m]$ and fixing $a_1 \to b_j$ for $j\in X$ and $a_1 \leftarrow b_j$ for $j \in [m]\ssm X$.

For $m \geq 3$, taking $X = \emptyset$, $\{1\}$, $\{1, 2\}$ gives 3 orientations which are not $\rho_{n \to m}$, where $a_1$ has in-degree $m$, $m-1$ and $m-2$, proving that $\Pnm[1][m]$ is not normally equivalent to a matroid polytope.

For $m \geq 4$, $X = \{1\}$, $\{1, 2\}$, $\{1, 2, 3\}$ gives 3 orientations which are not $\rho_{n \to m}$ nor $\rho_{n \leftarrow m}$, where $a_1$ has in-degree $m-1$, $m-2$ and $m-3$, thus $\Qnm[1][m]$ is not normally equivalent to a matroid polytope.
\end{proof}

\subsection{Permutahedral wedges}\label{sec:wedges}

In this section, we present a special case of the polyhedral wedges construction (see \cite[Section 2.1]{RorigZiegler-Wedges}), embedded with special coordinates, which we call \defn{permutahedral wedge}.
On the one side, wedges of indecomposable polytopes are always indecomposable (because they have an indecomposable face which shares a vertex with every facet, see McMullen's criterion stated in \Cref{cor:McMullen}).
On the other side, our construction is designed in order to guarantee that the permutahedral wedge of a deformed permutahedron is again a deformed permutahedron.

\begin{definition}\label{def:PermutahedralWedge}
Let $\pol\subset\R^n$ be a polytope and $i\in [n]$, the \defn{$i^{\text{th}}$-minimal permutahedral wedge $\pol_{\wedge i}$} and the \defn{$i^{\text{th}}$-maximal permutahedral wedge $\pol^{\wedge i}$} are defined as the polytopes in $\R^{n+1}$:
\begin{align*}
\pol_{\wedge i}&\eqdef \left(\pol\times(\b e_{n+1} - \b e_i)\R\right) \cap \{\b x \in \R^{n+1}~;~ x_{n+1} \geq  0\} \cap \{\b x \in \R^{n+1}~;~ x_i \geq \min_{\b y\in \pol} y_i\}\\
\pol^{\wedge i}&\eqdef \left(\pol\times(\b e_{n+1} - \b e_i)\R\right) \cap \{\b x \in \R^{n+1}~;~ x_{n+1} \leq  0\} \cap \{\b x \in \R^{n+1}~;~ x_i \leq \max_{\b y\in \pol} y_i\}
\end{align*}
\end{definition}

\begin{figure}[htpb]
    \centering
    \includegraphics[width=0.99\linewidth]{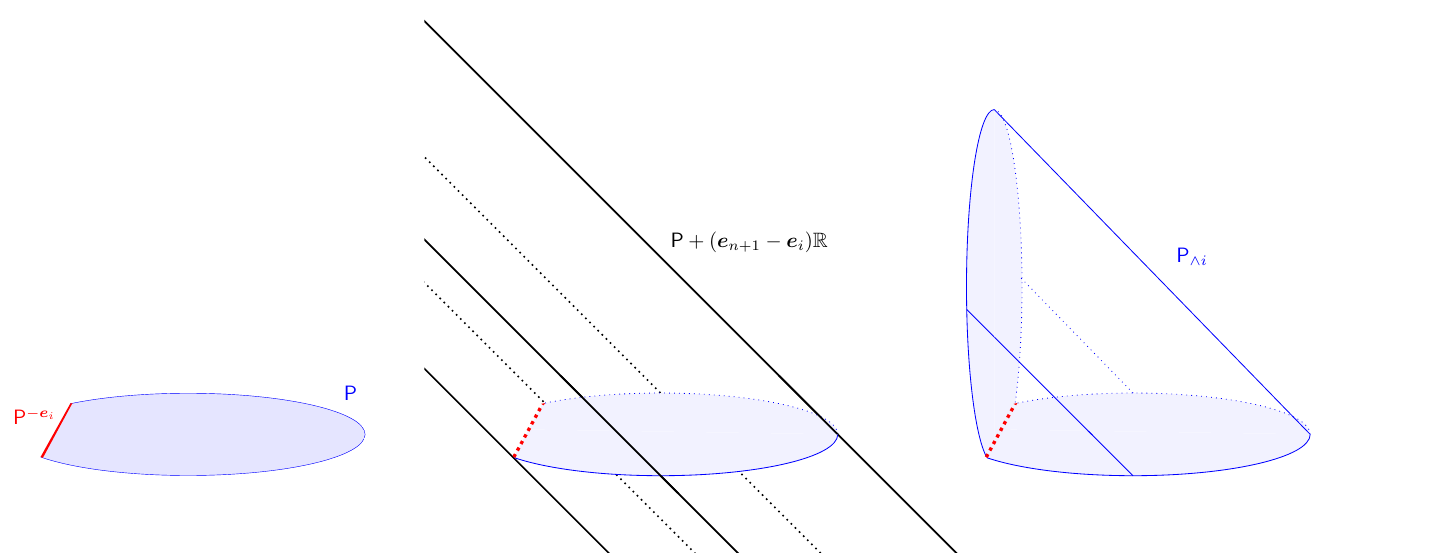}
    \caption[Construction of the (minimal) permutahedral wedge]{Construction of the $i^{\text{th}}$-minimal permutahedral wedge $\pol_{\wedge i}$.
    The vertical direction is $\b e_{n+1}$.
    As we care not only about combinatorics but also about geometry, we specify the axis of the cylinder to be~$\b e_{n+1} - \b e_i$, and we cut by orthogonal hyperplanes.
    This ensures that if $\pol$ is a generalized permutahedron, then so is~$\pol_{\wedge i}$.}
    \label{fig:Wedge}
\end{figure}

\Cref{fig:Wedge} depicts this construction. In~$\R^{n+1}$, we consider the cylinder $\pol\times(\b e_{n+1} - \b e_i)\R$. 
We cut this cylinder with the hyperplanes $\{\b x \in \R^{n+1}~;~ x_{n+1} = 0\}$ and $\{\b x \in \R^{n+1}~;~ x_i = \min_{\b y\in \pol} y_i\}$ (respectively $\{\b x \in \R^{n+1}~;~ x_i = \max_{\b y\in \pol} y_i\}$).
These cuts divide the cylinder into two unbounded parts and a bounded one: the bounded one is the $i^{\text{th}}$-permutahedral wedge $\pol_{\wedge i}$ (respectively $\pol^{\wedge i}$).
Said differently, with the notations of \cite[Section 2.1]{RorigZiegler-Wedges}, the $i^{\text{th}}$-minimal (respectively maximal) permutahedral wedge is the wedge of $\pol$ on its face $\pol^{-\b e_i}$ (resp. $\pol^{+\b e_i}$), but embedded with specific coordinates.

For a vertex $\b v$ of $\pol$, we denote \defn{$\b v_{\wedge i}$} $:= \b v + (v_i - \min_{\b y\in \pol}y_i)(\b e_{n+1} - \b e_i)$, and \defn{$\b v^{\wedge i}$} $:= \b v + (v_i - \max_{\b y\in \pol}y_i)(\b e_{n+1} - \b e_i)$.

\begin{lemma}[{\cite[Lem. 2.5]{RorigZiegler-Wedges}}]\label{lem:WedgeVertices}
For $\pol\subset\R^n$ and $i\in [n]$, the vertices of $\pol_{\wedge i}$ are $\b v$ and $\b v_{\wedge i}$ for $\b v$ a vertex of $\pol$.
Respectively, the vertices of $\pol^{\wedge i}$  are $\b v$ and $\b v^{\wedge i}$ for $\b v$ a vertex of $\pol$. 
\end{lemma}

\begin{proof}
For $\pol_{\wedge i}$, use \cite[Lem. 2.5]{RorigZiegler-Wedges} with $\pol[Q] = [-\min_{\b y\in \pol} y_i, 0]$ and $\b a_0 = \frac{1}{\min_{\b y\in \pol} y_i} \b e_i$.
Then apply the linear transformation given by $\b e_j\mapsto\b e_j$ for $j\leq n$, and $\b e_{n+1} \mapsto \b e_{n+1} - \b e_i$.
A similar result holds for $\pol^{\wedge i}$.
\end{proof}

McMullen's indecomposability criterion directly shows that wedges of indecomposable polytopes are indecomposable, we give here a proof with our vocabulary.

\begin{theorem}\label{thm:PermutahedralWedgesIndecomposable}
For $i\in [n]$, if $\pol\subset\R^n$ is indecomposable, then $\pol_{\wedge i}$ and $\pol^{\wedge i}$ are indecomposable polytopes.
\end{theorem}

\begin{proof}
The facets of $\pol_{\wedge i}$ either come from facets of the cylinder $\pol\times (\b e_{n+1} - \b e_i)\R$, or are supported by one of the hyperplanes $\{\b x \in \R^{n+1}~;~ x_{n+1} = 0\}$ or $\{\b x \in \R^{n+1}~;~ x_i = \min_{\b y\in \pol} y_i\}$.
All facets share a vertex with the facet~$(\pol_{\wedge i})^{-\b e_{n+1}}\cong \pol\times\{0\}$.
As~$\pol$ is indecomposable, its vertices form a dependent subset of vertices, and \Cref{cor:mainthmpolytopes} ensures that $\pol_{\wedge i}$ is indecomposable.
The same holds for $\pol^{\wedge i}$.
\end{proof}

The construction is designed so that it remains in the family of deformed permutahedra.

\begin{theorem}\label{thm:PermutahedralWedgesGP}
For $i\in [n]$, if $\pol\subset\R^n$ is a deformed permutahedron, then $\pol_{\wedge i}$ and $\pol^{\wedge i}$ also are.
\end{theorem}

\begin{proof}
The edges of $\pol_{\wedge i}$ are either (i) edges of the facet $(\pol_{\wedge i})^{-\b e_{n+1}}=\pol\times \{0\}$, (ii) edges of the facet $(\pol_{\wedge i})^{-\b e_{i}}\conv\bigl(\b v_{\wedge i} ~;~ \b v \in \verts(\pol)\bigr)$, or (iii) edges in the direction of the cylinder. 
Edges of the form (i) belong to $\pol\times\{0\}$, which is a deformed permutahedron, and hence are all in direction $\b e_j - \b e_k$ for some $j \ne k$.
For (ii), note that $(\pol_{\wedge i})^{-\b e_{i}}$ is the image of $(\pol_{\wedge i})^{-\b e_{n+1}}$ under the reflection on the hyperplane $x_{n+1}-x_i \geq \min_{\b y\in \pol} y_i$. This reflection exchanges $\b e_{n+1}$ and $\b e_i$ and leaves all the other $\b e_j$ unchanged, and hence the edges $(\pol_{\wedge i})^{-\b e_{i}}$ are also of direction $\b e_j - \b e_k$ for some $j \ne k$. Finally, the edges (iii) of the cylinder are all in direction $\b e_{n+1} - \b e_i$. Consequently, $\pol_{\wedge i}$ is a deformed permutahedron.
The same holds for $\pol^{\wedge i}$.
\end{proof}

In order to create ``a lot'' of indecomposable deformed permutahedra which are not matroid polytopes, we can construct a family of polytopes, closed under taking permutahedral wedges, which contains all (indecomposable and non-matroidal) truncated graphical zonotopes $\Pnm$ and $\Qnm$.
To describe a polytope in this family we can give a starting truncated zonotope $\pol_0$ living in $\R^d$, and a sequence of indices $(j_d, \dots, j_{N-1})$ with $j_k\in \{\pm1, \dots, \pm k\}$. Then, for $d\leq k\leq N$, $\pol_k\eqdef (\pol_{k-1})_{\wedge i}$ if $j_k=+i$, and $\pol_k\eqdef (\pol_{k-1})^{\wedge i}$  if $j_k=-i$.
The number of indecomposable deformed $N$-permutahedra in such a family is bounded above by: $\sum_{d\geq 4} 2\lfloor\frac{d-1}{2}\rfloor\, 2^{N-d}\frac{(N-1)!}{(d-1)!} \,=\, O(N^2\, 2^N\, N!)$.
Note that this is only an upper bound, because we cannot guarantee that we construct non-normally-equivalent polytopes. Actually, computer experiments tend to indicate that the method described here cannot generate $O(N^2\, 2^N\, N!)$ non-normally-equivalent polytopes.

We give an unoptimized simple lower bound by restricting ourselves to a sub-family generated by the strawberry $\Pnm[1][3]$.

\begin{corollary}\label{cor:BetterLowerBound}
For $N\geq 4$, there are at least $\frac{1}{3!}(N-1)!$ non-normally-equivalent indecomposable deformed $N$-permutahedra which are not normally-equivalent to matroid polytopes.
\end{corollary}

To prove this corollary, we will need two short lemmas on the geometry of wedges.

\begin{definition}
For a polytope $\pol\subset\R^d$ and $\b c\in \R^d$, we say that the face $\pol^{\b c}$ is a \defn{co-facet} if $\pol^{-\b c}$ is a facet.
\end{definition}

\begin{lemma}\label{lem:WedgesKeepsFacets}
Fix a polytope $\pol\subset\R^n$ and $i\in [n]$. For $j\in [n]$, $j\ne i$, if the face minimizing (resp. maximizing) $\b e_j$ in $\pol$ is a co-facet, then the faces minimizing (resp. maximizing) $\b e_j$ in both $\pol_{\wedge i}$ and $\pol^{\wedge i}$ are co-facets. 
Moreover, the faces maximizing (resp. minimizing) $\b e_i$ or $\b e_{n+1}$ in $\pol_{\wedge i}$ (resp. in $\pol^{\wedge i}$) are co-facets.
\end{lemma}

\begin{proof}
If $j\ne i$, then $\bigl(\pol_{\wedge i}\bigr)^{\b e_j}$ is the convex hull of two polytopes of dimension $\dim\pol^{\-\b e_j}$ living in orthogonal hyperplanes.
Hence if $\pol^{\b e_j}$ is a facet of $\pol$, then $\bigl(\pol_{\wedge i}\bigr)^{\-\b e_j}$ is a facet of $\pol_{\wedge i}$.
Moreover, $\bigl(\pol_{\wedge i}\bigr)^{-\b e_i}$ and $\bigl(\pol_{\wedge i}\bigr)^{-\b e_{n+1}}$ are affinely equivalent to $\pol$, and hence they are facets of $\pol_{\wedge i}$.
Similar arguments solve the remaining cases.
\end{proof}

\begin{lemma}\label{lem:WedgesAreRecognizable}
For a polytope $\pol$, and $i, j\in [n]$ with $i\ne j$, if the face maximizing (respectively minimizing) $\b e_i$ in $\pol$ is a co-facet, then $\pol_{\wedge i}$ (resp. $\pol^{\wedge i}$) is not normally equivalent to $\pol_{\wedge j}$ nor to $\pol^{\wedge j}$.
\end{lemma}

\begin{proof}
On the one hand, note that the face in direction $\b e_{n+1}$ of $\pol_{\wedge i}$ is normally equivalent to $\pol$, which is not the case for $\pol^{\wedge j}$ for all $j\in [n]$: hence $\pol_{\wedge i}$ and $\pol^{\wedge j}$ are not normally equivalent (even for $i = j$).

On the other hand, consider the face $\bigl(\pol_{\wedge i}\bigr)^{\b e_{n+1}}$: by construction, it lives in an hyperplane orthogonal to~$\b e_i$, and is isomorphic to $\pol^{\b e_i}$.
If $\pol^{\b e_i}$ is a facet, then $\bigl(\pol_{\wedge i}\bigr)^{\b e_{n+1}}$ is of dimension $d-1$.
Hence, $\bigl(\pol_{\wedge i}\bigr)^{\b e_{n+1}}$ does not live in any hyperplane orthogonal to $\b e_j$ for $j\notin \{i, n+1\}$.
Consequently, if $i\ne j$, then $\bigl(\pol_{\wedge i}\bigr)^{\b e_{n+1}}$ and $\bigl(\pol_{\wedge j}\bigr)^{\b e_{n+1}}$ live in different hyperplanes, so $\pol_{\wedge i}$ and $\pol_{\wedge j}$ are not normally equivalent.
\end{proof}

\begin{proof}[Proof of \Cref{cor:BetterLowerBound}]
First, note that performing a permutahedral wedge preserves the properties of:
\begin{compactenum}
\item[$\bullet$] being indecomposable (by \Cref{thm:PermutahedralWedgesIndecomposable});
\item[$\bullet$] being a deformed permutahedron (by \Cref{thm:PermutahedralWedgesGP});
\item[$\bullet$] being not normally equivalent to a matroid polytope (by \Cref{lem:NotMatroidPolytope}, because the number of different values taken on a given coordinate can only increase).
\end{compactenum}
Besides, $\dim \pol_{\wedge i} = \dim \pol^{\wedge i} = \dim \pol + 1$.

To each sequence of indices $\b j = (j_4, \dots, j_{N-1})$ with $1 \leq j_k \leq k$, we associate a polytope $\pol_N^\b j$ obtained recursively as follows: 
$$\pol_4^\b j = \Pnm[1][3] ~~~\text{ and }~~~
\pol_{k+1}^\b j = \begin{cases}
		(\pol_k)^{\wedge j_k} & \text{ if } \varepsilon_k = 1, \\
	(\pol_k)_{\wedge j_k} & \text{ if } \varepsilon_k = -1\end{cases},$$
where $\varepsilon_1 = \varepsilon_2 = \varepsilon_3 = \varepsilon_4 = -1$ and $\varepsilon_{k} = \varepsilon_{j_{k-1}} \times (-1)^{\text{number of copies of } j_k \text{ in } (j_4, \dots, j_{k-1})}$ for $5\leq k\leq N-1$.

By the above discussion, then $\pol_N^\b j$ is an indecomposable $N$-deformed permutahedron which is not normally equivalent to a matroid polytope.
There are $\frac{1}{3!}(N-1)!$ sequences $\b j = (j_4, \dots, j_{N-1})$ with $1 \leq j_k \leq k$.
It remains to prove that if $\b j\ne\b j'$, then $\pol_N^{\b j}$ and $\pol_N^{\b j'}$ are not normally equivalent.

Firstly, the faces of the strawberry $\Pnm[1][3]$ minimizing $\b e_1$, $\b e_2$, $\b e_3$ and $\b e_4$ are co-facets:
thanks to the choice of $(\varepsilon_1, \dots, \varepsilon_{N-1})$, \Cref{lem:WedgesKeepsFacets} guarantees, by induction, that all the wedges we perform are on co-facets (\ie $(\pol_k)^{-\varepsilon_k\b e_{j_k}}$ is a facet).

Next, as all the wedges are taken on co-facets, we can apply \Cref{lem:WedgesAreRecognizable} inductively: if all $\pol_N^\b j$ are not normally equivalent for sequences $\b j$ of length $\ell$, then there all $\pol_N^{\b j'}$ cannot be normally equivalent for sequences $\b j'$ of length $\ell + 1$.
This concludes the proof.
\end{proof}

In a concomitant article with Georg Loho~\cite{LohoPadrolPoullot2025RaysSubmodularCone}, we provide a much larger new family of $2^{2^{N-2}}$ indecomposable deformed $N$-permutahedra.

\section{Uniquely decomposable frameworks}\label{sec:edgedecomposable}

In this section, we study edge decomposable frameworks, an important family of uniquely decomposable frameworks defined by the fact that the contribution on each of the edges can be attributed to a single Minkowski summand.

\subsection{Contractible sets of edges}

\begin{definition}
	For a subset of edges $X\subseteq \edges $ of a framework $\fw = (\verts, \edges, \p)$, the \defn{characteristic vector} associated to $X$ is the vector $\mathdefn{\cv}\in\R^{E}_+$ defined by $\cve_{e} = \left\{\begin{array}{lr}
		1 & \text{ if } e\in X \\
		0 & \text{ if } e\notin X
	\end{array}\right.$.
\end{definition}

We study the cases where characteristic vectors are edge-deformation vectors of Minkowski summands.

\begin{definition}
	Let $\fw = (\verts, \edges, \p)$ be a framework. A subset of non-degenerate edges $X\subseteq \nde$ is \defn{contractible} if $\cv[\nde \ssm X]\in \DefoCone[\fw]$ (\ie if $\cv[\nde \ssm X]$ satisfies all cycle equations of $\fw$).
\end{definition}

This definition can be interpreted geometrically as follows: $X$ is contractible if shrinking all the edges in $X$ to a point while keeping the length of the edges in $\nde\ssm X$ yields a deformation of $\fw$.

\begin{lemma}
	A subset $X\subseteq\nde$ is contractible if and only if its complement $\nde \ssm X$ is contractible.
\end{lemma}

\begin{proof}
	Note that $\cv[\nde]$ is the edge-deformation vector of $\fw$, and hence it belongs to the deformation space. If $X$ is contractible, then $\cv[\nde \ssm X]$ also belongs to this subspace, and so does their difference $\cv[\nde]-\cv[\nde \ssm X]=\cv$, which is the characteristic vector  of $X$. This is the definition of $\nde \ssm X$ being contractible.
\end{proof}
This means that a subset of edges $X$ is contractible if the remaining edges can be contracted to a point, leaving $X$ as the edge set of a deformation. That is, $X$ is contractible if and only if $\cv$ belongs to the deformation cone. This is the characterization that we will use more often from now on.

\begin{definition}
	If $X\subseteq \nde$ is a contractible subset of edges, we define by $\mathdefn{\restr}$ and $\mathdefn{\contr}$ the deformations of $\fw$ with edge-deformation vectors $\cv$ and $\cv[\nde \ssm X]$, respectively.
\end{definition}

Note that the union of two disjoint contractible sets of edges is contractible because $\cv[X\sqcup Y] = \cv[X] + \cv[Y]$, and $\DefoCone[\fw]$ is a convex cone.
Furthermore, all contractible sets can be written as a union of dependent sets, because dependent sets partition the edges, and it is not possible to contract an edge without contracting its complete dependency class.

\begin{figure}
	\centering
	\includegraphics[width=0.75\linewidth]{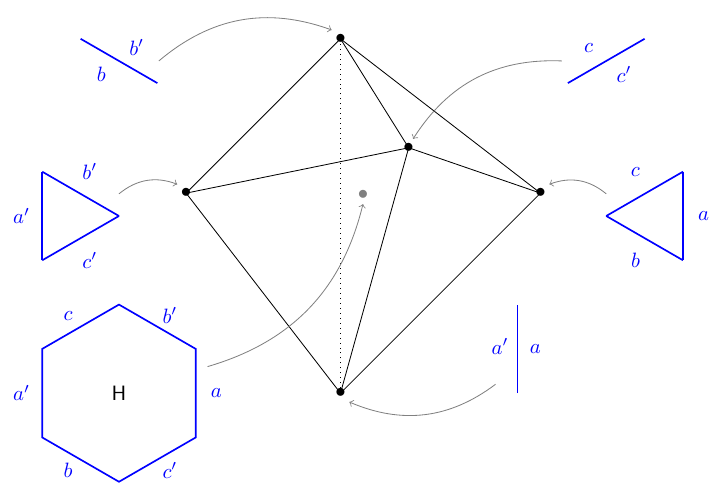}
	\caption[Deformation cone of the regular hexagon]{In \textcolor{blue}{blue} are the regular hexagon $\pol[H]$ and (the polytopes associated to) the rays of $\DefoCone[{\pol[H]}]$ obtained from characteristic vectors.
		The cone $\DefoCone[{\pol[H]}]$ is a cone over the \textbf{black} bi-pyramid in the center.}
	\label{fig:DChexagon}
\end{figure}

\begin{example}[Regular hexagon]\label{ex:hexagon}
	We label the edges of the regular hexagon $\pol[H]$ by $\pol[a], \pol[b]', \pol[c], \pol[a]', \pol[b], \pol[c]'$, in cyclic order, see \Cref{fig:DChexagon} (bottom left).
	There is only one 2-face, whose cycle equation is given by $\b\lambda_{\pol[a]} - \b\lambda_{\pol[a']} = \b\lambda_{\pol[b]} - \b\lambda_{\pol[b']} = \b\lambda_{\pol[c]} - \b\lambda_{\pol[c']}$.
	The contractible sets of edges with minimal support are $(\pol[a], \pol[a]')$, $(\pol[b], \pol[b]')$, $(\pol[c], \pol[c]')$, $(\pol[a], \pol[b], \pol[c])$ and $(\pol[a]', \pol[b]', \pol[c]')$.
	Each of these sets of edges $X$ is associated to its characteristic vector $\cv$, and to the framework $\restr$ with this edge-deformation vector (which will be the framework of a polytope $\restr[\pol]$ that is a deformation of the hexagon): up to translation, these polytopes are respectively the segments $\pol[a]$, $\pol[b]$, $\pol[c]$; and the two equilateral triangles with edges parallel to $\pol[a]$, $\pol[b]$ and $\pol[c]$ (which are related by a central symmetry).
	
	Their characteristic vectors are rays of $\DefoCone[{\pol[H]}]$.
	One checks that these 5 rays are the only rays:
	$\DefoCone[{\pol[H]}]$ is a cone over a bi-pyramid over a triangle (the apices of the bi-pyramid are associated to the two triangles; the vertices of the triangle to the segments), see \Cref{fig:DChexagon}.
	
	Note however that $\ELD[{\pol[H]}]$ is the graph on node set $E(\pol[H]) = (\pol[a], \pol[b], \pol[c], \pol[a]', \pol[b]', \pol[c]')$ composed of 6 isolated nodes.
	Indeed, for each pair of edges $\pol[e], \pol[f]\in E(\pol[H])$ there exists some $\b\lambda$ with $\b\lambda_{\pol[e]} \ne \b\lambda_{\pol[f]}$ that satisfies the above cycle equations.
	There are no dependent edges.
\end{example}

We now focus on contractible \emph{dependent} sets of edges (recall that we say that a set of edges is \emph{dependent} if the edges it contains are pairwise dependent).

\begin{theorem}\label{thm:CharacteristicLengthsGiveRays}
	If a non-empty set of edges $X\subseteq \nde$ is contractible and dependent, then $\cv$ is a ray of $\DefoCone[\fw]$, and $X$ is a connected component of $\ELD$.
\end{theorem}

\begin{proof}
	Let $C$ be the connected component of $\ELD$ containing $X$.
	As $C$ is dependent, we have $\b\lambda_{e} = \b\lambda_{f}$ for all $e, f\in C$ and $\b\lambda\in\DefoCone[\fw]$.
	Now, $X$ is contractible, which means that $\cv[X]\in \DefoCone[\fw]$. This implies that $\cve_{e} = 1$ for all $e\in C$, and hence $C\subseteq X$.
	As $X$ dependent (\ie is connected in $\ELD$), we have $C = X$.
	
	Now, $\restr[\fw][X]$ is indecomposable, because its non-degenerate edges are pairwise dependent. Thus its edge-deformation vector $\cv[X]$ is a ray of $\DefoCone[\fw]$.
\end{proof}

The existence of a contractible dependent set of edges yields an even stronger consequence: it induces a product structure on the deformation cones.

\begin{theorem}\label{thm:contractibledependentdecomposition}
	For a framework $\fw = (\verts, \edges, \p)$, if $X\subseteq \nde$ is a contractible set of dependent edges, then 
	$\DefoCone[\fw] = \DefoCone[\contr] \times \DefoCone[\restr]$ where $\DefoCone[\fw_X] = \cone(\cv[X])$ is a ray. 
\end{theorem}

\begin{proof}
The fact that $\DefoCone[\fw_X] = \cone(\cv[X])$ is a ray is immediate from \Cref{thm:CharacteristicLengthsGiveRays}.

	We show that $f : (\b\mu, \b\mu') \mapsto \b\mu + \b\mu'$ is a bijection sending $\DefoCone[\contr]\times \DefoCone[\restr]$ to $\DefoCone[\fw]$.
	Note that, as  $\DefoCone[\contr],\DefoCone[\restr]\subseteq\DefoCone[\fw]$, the linear map $f$ is well-defined (because $\DefoCone[\fw]$ is closed under taking sums).
	Besides, as all $e\in X$ are degenerate edges in $\contr$, the cones $\DefoCone[\contr]$ and $\DefoCone[\restr]$ lie in orthogonal sub-spaces, hence $f$ is injective.
	It remains to prove that $f$ is surjective.
	
	Consider a deformation $\b\lambda\in \DefoCone[\fw]$, and let $\alpha = \b\lambda_e$ for some $e\in X$.
	As $X$ is dependent, $\alpha$ does not depend on the choice of $e\in X$, and $\alpha\cv[X]\in \R_+\b\cv=\DefoCone[\restr]$.
	Moreover, $\b\mu \coloneqq \b\lambda - \alpha\cv[X]$ has all non-negative coordinates, it lies in the dependence space, because it is a difference of vectors in the dependence space, and lies in the non-negative orthant. Hence $\b\mu\in \DefoCone[\fw]$.
    Note that, by definition of $\contr$, we have: $\DefoCone[\contr] = \DefoCone[\fw] \cap \{\b\lambda\in \R^\edges_+ ~;~ \b\lambda_e = 0 \text{ for all } e\in X\}$.
    Thus $\b\mu\in \DefoCone[\contr]$, which proves the surjectivity of~$f$, ending the proof.
\end{proof}

Note that if $X$ contractible but not dependent, then we have the inclusion $\DefoCone[\fw] \supseteq \DefoCone[\contr] \times \DefoCone[\restr]$, but it can be strict, like in the regular hexagon of \Cref{ex:hexagon} (\eg with $X = \{\pol[a], \pol[a]'\}$).

\subsection{Edge decomposable frameworks}

\begin{definition}
We say that a framework $\fw = (\verts, \edges, \p)$ is \defn{edge decomposable} if its support can be partitioned into non-empty contractible dependent sets of edges $\nde = C_1\sqcup \cdots\sqcup C_r$.
\end{definition}

We can use contractible dependent sets of edges to prove that certain deformation cones are simplicial.

\begin{theorem}\label{thm:NiceConnectedComponentImplySimplicialCone}
If a framework $\fw = (\verts, \edges, \p)$ is edge decomposable then it is uniquely decomposable. More precisely, if $\nde = C_1\sqcup \cdots\sqcup C_r$ is a partition into non-empty contractible dependent sets of edges, then $\DefoCone[\fw]$ is a simplicial cone of dimension $r$ whose rays are spanned by $\cv[C_1],\dots,\cv[C_r]$, and $\fw=\restr[\fw][C_1]+\cdots+\restr[\fw][C_r]$ is the unique decomposition into indecomposables with distinct support.
	%
\end{theorem}

\begin{proof}
Note that each $C_i$ is also a contractible dependent set of edges for any $\contr[\fw][C_j]$ with $j\neq i$. Therefore, we can iteratively apply \Cref{thm:contractibledependentdecomposition} to get that 
$\DefoCone[\fw] =\DefoCone[{\restr[\fw][C_1]}] \times \cdots \times \DefoCone[{\restr[\fw][C_r]}]$, which is a simplicial cone because $\DefoCone[{\restr[\fw][C_i]}]$ is the ray spanned by $\cv[C_i]$. By \Cref{lem:uniquelydecomposable}, the decomposition $\fw=\restr[\fw][C_1]+\cdots+\restr[\fw][C_r]$ must be unique.
\end{proof}

\begin{remark}
	If $\fw$ is edge decomposable, then any deformation $\fw[G]\in \DefoCone[\fw]$ is too. Indeed, every contractible dependent set of edges $C_i$ is either completely contracted in $\fw[G]$, in which case they are degenerate edges of $\fw[G]$, or is still a contractible dependent set of edges of $\fw[G]$. Therefore, if $\fw$ is edge decomposable, the conclusion of \Cref{thm:NiceConnectedComponentImplySimplicialCone} holds for any of its deformations.
\end{remark}

\begin{remark}
We have proven that edge decomposability implies unique decomposability, but the reciprocal is not true.
Indeed, every quadrilateral in the plane is uniquely decomposable (by  \cite{CDGRY2020}, all polytopes combinatorially equivalent to a product of simplices are), but a scalene quadrilateral like in \Cref{sfig:ScaleneQuadrilateral} is not edge decomposable.
\end{remark}

%
%
%

As a first consequence, we show that polytopes whose $2$-faces are either triangles or parallelograms are always edge decomposable and thus uniquely decomposable.

\begin{lemma}\label{lem:EdgeDecompositionFor3and4cycles}
	Let $\pol$ be a polytope all whose $2$-faces are either triangles or parallelograms. Define the equivalence relation $\mathdefn{\sim}$ on the set of edges $E$ induced by $e\sim f$ if they belong to a common triangle, or are opposite edges of a parallelogram.	Then the equivalence classes of $\sim$ are non-empty contractible dependent sets that partition $\edges$.
\end{lemma}

\begin{proof}
It follows from observing that the characteristic vector of each equivalence class $C$ fulfills the face equations of $\pol$, and thus belongs to $\DefoCone$.
Indeed, for the triangular faces with edges $e,f,g$, the face equation is $\b\lambda_e = \b\lambda_f = \b\lambda_g$, and for parallelogram faces with cyclically ordered edges $e,f,g,h$, the face equations are $\b\lambda_e = \b\lambda_g$ and $\b\lambda_g = \b\lambda_h$.
By construction of our equivalence classes, their characteristic vectors fulfill these equations.
By \Cref{ex:triangle,ex:parallelogram}, they are dependent too.
\end{proof}

The combination of this with \Cref{thm:NiceConnectedComponentImplySimplicialCone} yields:

\begin{corollary}\label{cor:triangleparallelogram}
Every polytope $\pol$ all whose $2$-faces are either triangles or parallelograms is edge decomposable, and therefore uniquely decomposable. Precisely, if $C_1, \dots, C_r$ are the equivalence classes of~$\sim$, then $\DefoCone$ is a simplicial cone of dimension $r$ whose rays are spanned by $\restr[\pol][C_1],\dots,\restr[\pol][C_r]$, and where $\pol = \restr[\pol][C_1]+\cdots+\restr[\pol][C_r]$ is the unique decomposition into indecomposable polytopes with distinct normal fans.
\end{corollary}

\begin{proof}
Combine \Cref{lem:EdgeDecompositionFor3and4cycles,thm:NiceConnectedComponentImplySimplicialCone}.
\end{proof}


\begin{example}
As direct examples of applications of \Cref{cor:triangleparallelogram}, we get:
\begin{compactitem}
\item For a polytope $\pol$ all whose 2-faces are triangles or parallelograms, $\pol$ is indecomposable if and only if the equivalence relation $\sim$ has a unique equivalence class.
\item The two polytopes of \Cref{fig:DiminishedTrapezohedronAndGyrobifastigium} have a simplicial deformation cone of dimension 2, and the unique way to write them as a Minkowski sum is as a sum of two triangles.
\item
The \emph{rhombicuboctahedron} is an Archimedean solid arising as the Minkowski sum of the standard cube $[0, 1]^n$ with the standard octahedron $\frac{1}{\sqrt{2}}\conv(\pm\b e_i ~;~ i\in [3])$.
All the 2-faces of a rhombicuboctahedron are triangles and parallelograms, and there are 4 equivalent classes, hence its deformation cone is a simplicial cone of dimension 4, whose 4 rays correspond to the chosen octahedron and the 3 segments which sum to the chosen cube.
\item If $\polZ = \sum_{\b v\in V} [\b 0, \b v]$ is a zonotope whose 2-faces are parallelograms, \ie there are no 3 vectors in $V$ lying in the same plane, then $\DefoCone[\polZ]$ is simplicial of dimension $r = |V|$, and the deformations of $\polZ$ are normally equivalent to the polytopes $\sum_{\b v\in V'} [\b 0, \b v]$ for $V'\subseteq V$. The converse is also true, if a zonotope $\polZ$ has a non-parallelogram $2$-face $\polF$, then its deformation cone its not simplicial because it contains $\DefoCone[\polF]$ as a face since $\polF$ is a deformation of $\polZ$.
A particular case was observed in \cite[Cor.~2.10]{PPP2023gZono}, which states that a graphical zonotope has a unique decomposition if and only if the underlying graph is triangle-free.
\item If $\pol = \pol[S]_{n_1}\times\dots\times \pol[S]_{n_r}$ is a product of simplices, then all the 2-faces of $\pol$ are product of simplices, hence triangles or parallelograms. It is straightforward to see that $\sim$ admits $r$ equivalence classes, so $\DefoCone$ is a simplicial cone of dimension $r$. This was generalized in \cite{CDGRY2020} to polytopes combinatorially equivalent to products of simplices.
\item Matroid polytopes have this property, and this provides an alternative proof for \Cref{cor:matroidpolytopesimplicial}, as the equivalence classes of edges correspond to the basis exchanges within each of the connected components of the matroid.
\item In fact, this extends to all $0/1$-polytopes,
see~\cite{HigashitaniPadrolSanyal} where it is proven that all $0/1$-polytopes that are not a Cartesian product are indecomposable.
\end{compactitem}
\end{example}

\subsection{Parallelogramic Minkowski sums}\label{ssec:parallelogramicMinkowskisums}
\Cref{thm:products} states that the deformation cone of a product of frameworks is the product of their deformation cones.
Yet, this property is not specific to Cartesian products, as we showcase now. A parallelogramic Minkowski sum of polytopes $\pol_1,\dots,\pol_k$ will be a polytope of dimension as low as $\max_i (\dim \pol_i)$, but whose deformation cone is isomorphic to $\DefoCone[\pol_1] \times \cdots \times \DefoCone[\pol_k]$.
When $\pol_1,\dots,\pol_k$ are indecomposable, then any of their parallelogramic Minkowski sums will be uniquely decomposable, but parallelogramic Minkowski sums need not to be edge decomposable in general.
These results are stated directly at the level of polytopes instead of frameworks, because to add frameworks one needs to have beforehand fixed a common underlying graph, and we will perform Minkowski sums of unrestricted families polytopes.

\begin{definition}
	For a polytope $\pol$, we say that a vector $\b v$ is \defn{2-generic} with respect to $\pol$ if $\b v$ is not parallel to any of the $2$-faces of $\pol$, \ie there do not exist two edges $\b x\b y$ and $\b a\b b$ in a common $2$-face of $\pol$ such that $\b v = \alpha(\b y - \b x) + \beta(\b b - \b a)$ for some $\alpha, \beta\in \R$.
	
	Two polytopes $\pol$ and $\polQ$ are in \defn{parallelogramic position} if every edge direction of $\pol$ is $2$-generic with respect to $\polQ$, and every edge direction of $\polQ$ is $2$-generic with respect to $\pol$.
\end{definition}

Recall that, for two polytopes $\pol$ and $\polQ$, the faces of the Minkowski sum $\pol + \polQ$ are of the form $\polF + \pol[G]$ for some face $\polF$ of $\pol$ and some face $\pol[G]$ of $\polQ$.
In particular, if $\pol$ and $\polQ$ are in parallelogramic position (or more generally if no edge of $\pol$ is parallel to an edge of $\polQ$), then all the edges of $\pol+\polQ$ are of the form $\pol[e] + \b x$ for some edge $\pol[e]$ and some vertex $\b x$ with $\pol[e] \subseteq \pol$ and $\b x\in \polQ$ or with $\b x \in \pol$ and $\pol[e]\subseteq\polQ$.

For an edge $\pol[e]$ of $\pol$ or $\polQ$, its associated \defn{edge class $C_{\pol[e]}$} in $\pol+\polQ$ is the set of edges of $\pol + \polQ$ of the form $\pol[e] + \b x$ for a vertex $\b x$ of $\polQ$ or $\pol$.

\begin{theorem}\label{thm:GeneralParallelogramicPositionDC}
	If $\pol$ and $\polQ$ are in parallelogramic position, then $\DefoCone[\pol+\polQ] \simeq \DefoCone \times \DefoCone[\polQ]$.
	
	Moreover, the graph $\ELD[\pol + \polQ]$ is the lift of $\ELD[\pol]\cup\ELD[\polQ]$ for the equivalence relation whose equivalence classes are the edge classes.
\end{theorem}

\begin{proof}
	Pick an edge $\pol[e]$ of $\pol$.
	Consider the fan obtained by intersecting the normal cone of $\pol[e]$ in $\pol$ with the normal fan of $\polQ$, \ie the fan whose collection of cones are $\c N_{\pol[e]}(\b v) \coloneqq \{\b c \in \R^d ~;~ (\pol+\polQ)^{\b c} \supseteq \pol[e] + \b v\}$ for each vertex $\b v$ of $\polQ$.
	For every vertex $\b v$ of $\polQ$, if $\dim\bigl(\c N_{\pol[e]}(\b v)\bigr) = d - 1$, then for $\b c\in \interior \bigl(\c N_{\pol[e]}(\b v)\bigr)$, we get $(\pol+\polQ)^{\b c} = \pol[e] + \b v$.
	As $\pol[e]$ is 2-generic with respect to $\polQ$, there are no vertices $\b v, \b v',\b v''$ of $\polQ$ such that $\dim\bigl(\c N_{\pol[e]}(\b v)\cap \c N_{\pol[e]}(\b v')\cap \c N_{\pol[e]}(\b v'')\bigr)  = d - 2$, because otherwise $\pol[e]$ would be parallel to the $2$-face of $\polQ$ containing $\b v, \b v',\b v''$ (there is such a face if their normal cones intersect in co-dimension~$2$).
	Thus, if $\b v, \b v'$ are such that $\dim\bigl(\c N_{\pol[e]}(\b v)\cap \c N_{\pol[e]}(\b v')\bigr) = d - 2$, then for $\b c\in \interior\bigl(\c N_{\pol[e]}(\b v)\cap \c N_{\pol[e]}(\b v')\bigr)$, we have that $(\pol+\polQ)^{\b c} = \pol[e] + [\b v, \b v']$ is a 2-face of $\pol+\polQ$:
	as the latter is a parallelogram, the edges $\pol[e] + \b v$ and $\pol[e] + \b v'$ of $\pol + \polQ$ are dependent.
	Consequently, the following graph is a sub-graph of $\ELD[\pol+\polQ]$:
	its nodes are the vertices $\b v$ of $\polQ$ with $\dim\bigl(\c N_{\pol[e]}(\b v)\bigr) = d - 1$, and arcs are between $\b v$ and $\b v'$ for which $\dim\bigl(\c N_{\pol[e]}(\b v)\cap \c N_{\pol[e]}(\b v')\bigr) = d - 2$.
	This graph is connected because it is the dual graph of the above-defined fan.
	All the edges of $\pol+\polQ$ of the form $\pol[e] + \b v$ are nodes of this graph.
	
	Thus, for any fixed edge $\pol[e]$ of $\pol$, all the edges of $\pol+\polQ$ in the edge class $C_{\pol[e]}$ are dependent
	(and symmetrically inverting the roles of $\pol$ and $\polQ$).
	For any two face $\polF$ of $\pol$, taking $\b c$ with $\pol^{\b c} = \polF$ yields $(\pol+\polQ)^{\b c} = \polF + \b v$ for a vertex $\b v$ of $\polQ$, as otherwise there would be an edge of $\polQ$ which is not $2$-generic with respect to $\pol$.
	Hence, the cycle equations of $\pol$ are also valid in $\pol+\polQ$.
	Consequently, for two edges $\pol[e]$ and $\pol[e]'$ of $\pol$, the edge sets $C_{\pol[e]}$ and $C_{\pol[e]'}$ are dependent in $\ELD[\pol+\polQ]$ if and only if $\pol[e]$ and $\pol[e]'$ are dependent in $\ELD[\pol]$.
	The same holds for $\polQ$.
	This proves the $\ELD[\pol+\polQ]$ is the lift of $\ELD[\pol] \cup \ELD[\polQ]$ with respect to the equivalence relation induced by the partition into edge classes.
	
	Finally, for each edge $\pol[e]$ of $\pol$ or of $\polQ$, fix some representative $\pol[e]_\circ$ in $C_{\pol[e]}$.
	Consider the projection $\pi~\!:~\!\DefoCone[\pol+\polQ] \to \DefoCone\times\DefoCone[\polQ]$ with $\b\lambda \mapsto \bigl(\b\lambda_{\pol[e]_\circ} ~;~ C_{\pol[e]} \text{ edge class of } \pol \text{ or of } \polQ\bigr)$.
	As the cycle equations of $\pol$ and of $\polQ$ hold for $\pol+\polQ$, we indeed have that $\pi(\b\lambda) \in \DefoCone\times\DefoCone[\polQ]$.
	Moreover, as the edges in any $C_{\pol[e]}$ are dependent, this projection is injective.
	Lastly, this projection is surjective because the sum of any deformation $\pol'$ of $\pol$ with a deformation $\polQ'$ of $\polQ$ yields a deformation $\pol' + \polQ'$ of $\pol + \polQ$.
	This ensures that $\DefoCone[\pol+\polQ] \simeq \DefoCone[\pol]\times\DefoCone[\polQ]$.
\end{proof}

We say that a collection of polytopes $\pol_1, \dots, \pol_r$ is in \defn{parallelogramic position} if for every $1\leq i\leq r$, the polytopes $P_i$ and $\sum_{j\neq i}\pol_j$ are in parallelogramic position.
For example, the collection of polytopes $\pol_1, \dots, \pol_r$ is in parallelogramic position if no edge of one of the $\pol_i$ is parallel to the plane spanned by an edge of $\pol_j$ and an edge of $\pol_k$ (for $j\ne i$ and $k\ne i$ but possibly $j = k$).

\begin{corollary}\label{cor:ParallelogramicSumIndecomposablePolytopes}
	If $\pol_1, \dots, \pol_r$ are indecomposable polytopes in parallelogramic position, then $\pol_1 + \dots + \pol_r$ is edge decomposable, and therefore uniquely decomposable. 
	The deformation cone of $\pol_1 + \dots + \pol_r$ is simplicial of dimension $r$ whose rays are associated to the polytopes $\pol_1, \dots, \pol_r$.
	In particular every deformation of $\pol_1 + \dots + \pol_r$ is of the form $\alpha_1\pol_1 + \dots + \alpha_r\pol_r$ with $\alpha_i\geq 0$, and this decomposition is unique up to translation (with the condition that the summands are not normally equivalent).
\end{corollary}

\begin{proof}
For every $i$, the union of all edges that are copies of edges of $\pol_i$, \ie $\bigcup_{\pol[e]\in \edges(\pol_i)} C_{\pol[e]}$, is contractible (its characteristic vector is the edge-deformation vector of $\pol_i$ itself) and dependent, indeed, we have seen that copies of the same edge are dependent in the the proof of \Cref{thm:GeneralParallelogramicPositionDC}), and those that belong to a same copy of $\pol_i$ are pairwise dependent by \Cref{lem:ELDindecomposable}. The rest follows from  \Cref{thm:NiceConnectedComponentImplySimplicialCone}.
\end{proof}

\begin{remark}
	We denote $\b 0_m \coloneqq (0, \dots, 0)$ the zero vector of length~$m$.
	For two polytopes $\pol\subset\R^d$ and $\polQ\subset\R^e$, the polytopes $\pol\times\{\b 0_e\} \subset \R^{d+e}$ and $\{\b 0_d\}\times\polQ \subset \R^{d+e}$ are in parallelogramic position.
	Moreover, by definition: $\pol\times\polQ = \bigl(\pol\times\{\b 0_e\}\bigr) + \bigl(\{\b 0_d\}\times\polQ\bigr)$.
	Hence, \Cref{thm:GeneralParallelogramicPositionDC} implies \Cref{thm:products} for polytopes.
\end{remark}

\begin{figure}
	\centering
	\begin{subfigure}[b]{0.26\linewidth}
		\centering
		\includegraphics[width=\linewidth]{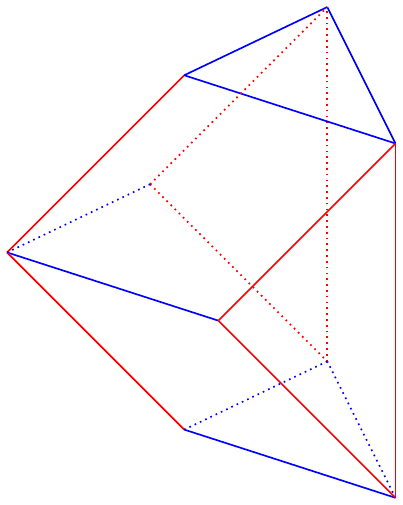}
		\caption{Diminished trapezohedron}
		\label{sfig:DiminishedTrapezohedron}
	\end{subfigure}\hspace{3cm}
	\begin{subfigure}[b]{0.3\linewidth}
		\centering
		\includegraphics[width=\textwidth]{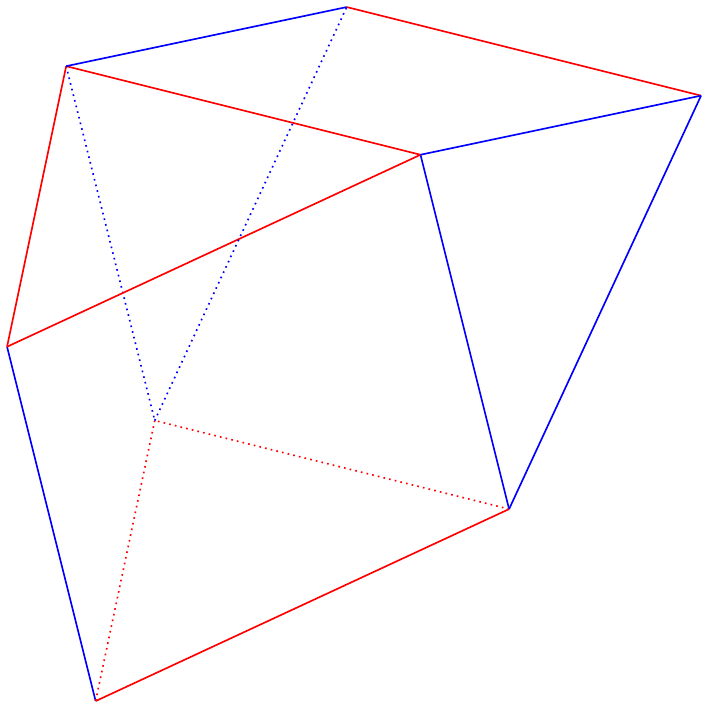}
		\caption{Gyrobifastigium}
		\label{sfig:Gyrobifastigium}
	\end{subfigure}
	\caption[Diminished trapezohedron and gyrobifastigium]{Two 3-dimensional polytopes obtained as the Minkowski sum $\textcolor{red}{\pol[T]} + \textcolor{blue}{\pol[T]'}$ of two triangles in parallelogramic position: (left) the diminished trapezohedron can be seen as half of the graphical zonotope $\ZGKnm[2][2]$, (right) the gyrobifastigium can be seen as gluing two triangular prisms.
		In each polytope, we draw in \textcolor{red}{red} the edges parallel to some edge of $\textcolor{red}{\pol[T]}$, and conversely for \textcolor{blue}{blue} edges.}
	\label{fig:DiminishedTrapezohedronAndGyrobifastigium}
\end{figure}

\begin{example}
	As depicted in \Cref{fig:DiminishedTrapezohedronAndGyrobifastigium}, the Minkowski sum of two triangles in parallelogramic position in dimension 3 gives rise either to a diminished trapezohedron or to a gyrobifastigium (if moreover the two triangles are taken to be equilateral, then this is the Johnson solid $J_{26}$, see \cite{Johnson1966-Solids}), depending on the relative positions of the two triangles.
	By \Cref{cor:ParallelogramicSumIndecomposablePolytopes}, all these polytopes have a simplicial deformation cone of dimension~$2$, and the unique way to decompose them as a sum of indecomposable polytopes is as a sum of two triangles.
\end{example}

\begin{example}\label{exm:ParallelogramicZonotopesLiterature}
	Another typical example of Minkowski sums of polytopes in parallelogramic position are parallelogramic zonotopes: Minkowski sums of segments in parallelogramic position (including cubes, and graphical zonotopes for triangle-free graphs).
	\Cref{cor:ParallelogramicSumIndecomposablePolytopes} ensures that the deformation cone $\DefoCone[\polZ(V)]$ of a parallelogramic zonotope is simplicial.
	Moreover, by \Cref{thm:GeneralParallelogramicPositionDC}, the edge-dependency graph of a parallelogramic zonotope $\polZ(V)$ is formed by cliques on the edge-classes $C_{\b v}$ for each $\b v\in V$.
	

\end{example}

\section{Additional constructions for indecomposable and uniquely decomposable polytopes}\label{sec:constructions}

In this section, we use the graph of (implicit) edge dependencies to analyze two constructions of polytopes with controlled deformation cones. On the one hand, in \Cref{ssec:DeepTruncationsOfParallelogramicZonotopes}, we extend the constructions of \Cref{ssec:TruncatedGraphicalZonotope}
	using deep truncations to create indecomposable polytopes from arbitrary parallelogramic zonotopes. On the other hand, in \Cref{ssec:parallelogramicMinkowskisums}, we describe how to construct indecomposable polytopes by stacking on top of facets of parallelogramic Minkowski sums.

%
%
%

\subsection{Deep truncations on parallelogramic zonotopes}\label{ssec:DeepTruncationsOfParallelogramicZonotopes}

To extend the results of \Cref{thm:PandQareIndecomposable}, we explore (deep) truncations of other zonotopes.
We refer the reader to \cite[Chapter 7.3]{Ziegler-polytopes} for an excellent introduction to zonotopes, and restrict ourselves to a presentation without proofs of the key facts we need.

\begin{definition}
A \defn{zonotope} is a Minkowski sum of segments.
Up to translation, all zonotopes can be written as the Minkowski sum $\mathdefn{\polZ(V)} \coloneqq \sum_{\b v\in V} [\b 0, \b v]$ for some set of vectors $V = (\b v_1, \dots, \b v_r)$.

In $\polZ(V)$, the \defn{edge-class $C_{\b v}$} of $\b v\in V$ is the set of edges which are parallel to $[\b 0, \b v]$, see \Cref{fig:DeepTruncations}: as all the edges of $\polZ(V)$ are parallel to some $[\b 0, \b v]$ for $\b v\in V$, the edge-classes form a partition of $E(\polZ(V))$.
The \defn{zone} of $\b v\in V$ is the collection of all the faces $\polF$ of $\polZ$ which contain and edge parallel to  $\b v$, \ie $C_{\b v} \cap E(\polF) \ne \emptyset$.

A zonotope is \defn{parallelogramic} if all its 2-faces are parallelograms.
In particular, $\polZ(V)$ is parallelogramic if and only if no three vectors in $V$ lie in a common plane (in particular, no two of them have the same direction, and $\b0\notin V$).
\end{definition}

\begin{definition}
	For a $d$-polytope $\pol$ with $d\geq 3$ and $\b x\in V(\pol)$, the vertex $\b x$ admits a \defn{deep truncation} if all the adjacent vertices belong to a common hyperplane $H_{\b x}$, see \Cref{fig:DeepTruncations,fig:SmallDimPandQ}.
	In this case, the deep truncation of $\pol$ at $\b x$ is defined as 
	\[\mathdefn{\pol\ssm \b x} \coloneqq \conv\bigl(\b u ~;~ \b u \in V(\pol),~ \b u\ne\b x\bigr)=\pol\cap H^-_{\b x}\] where $H^+_{\b x}$ is the closed half-space bounded by $H_{\b x}$ containing $\b x$, and $H^-_{\b x}$ its opposite closed half-space (the proof that these two descriptions above are equivalent is analogous to that of \Cref{lem:VerticesFacesOfPandQ}).
	We denote by \defn{$\pol[X]$} the distinguished facet of $\pol\ssm\b x$ that is supported by $H_{\b x}$.	
\end{definition}

\begin{figure}[h]
	\centering
	\includegraphics[width=0.31\linewidth]{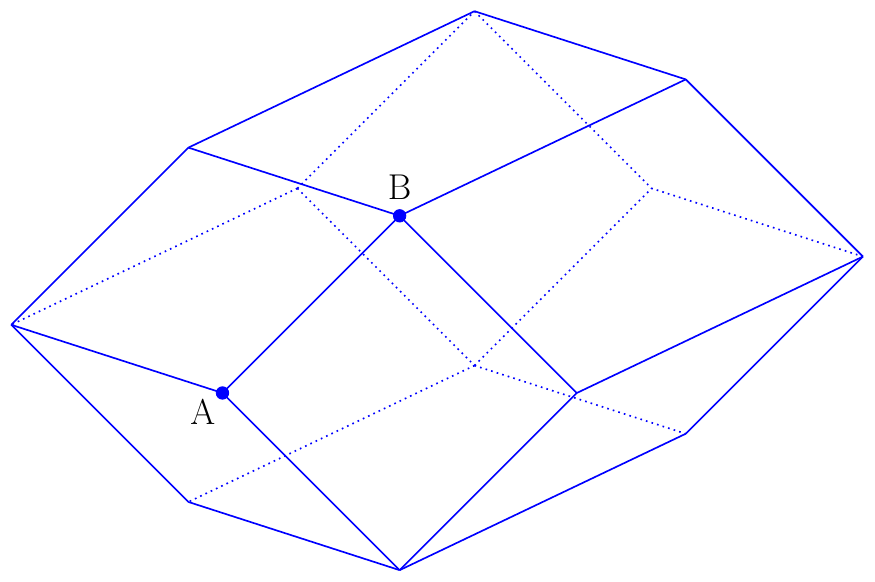}
	\includegraphics[width=0.22\linewidth]{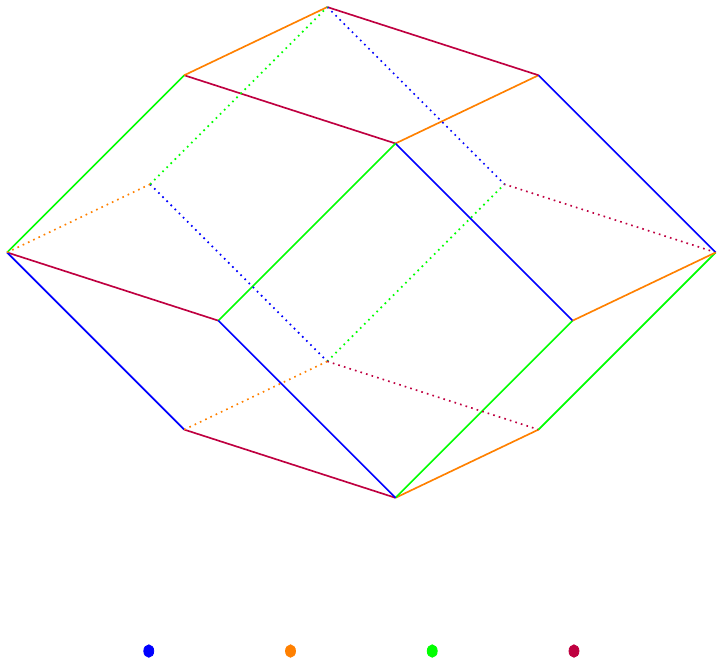}
	\includegraphics[width=0.22\linewidth]{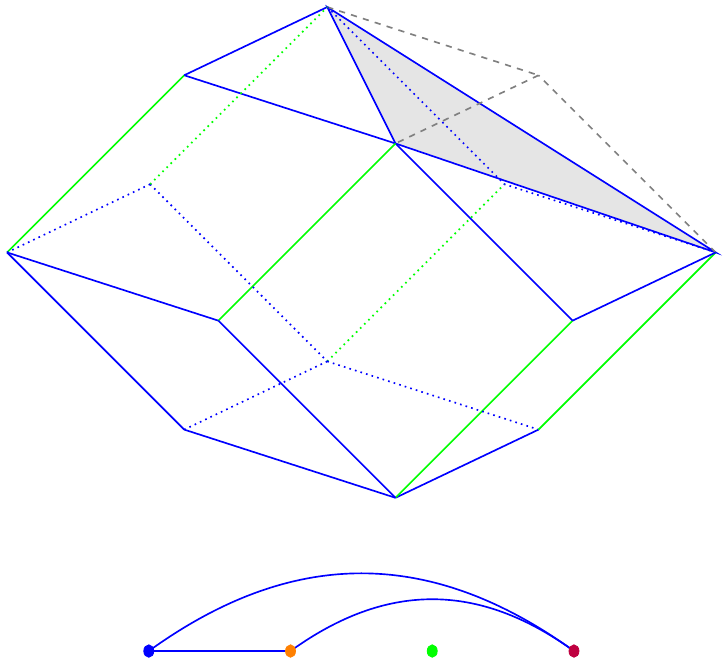}
	\includegraphics[width=0.22\linewidth]{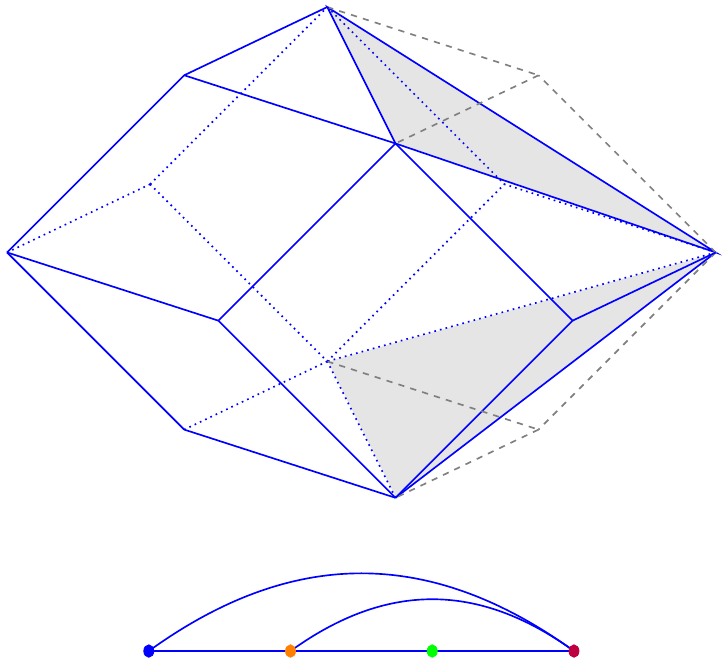}
	\caption[Truncating facets of a zonotope to improve its indecomposability]{(Leftmost) A zonotope where vertex $\b A$ admits a deep truncation, while vertex $\b B$ does not (one of its adjacent edges is too long). (Others) A zonotope with 4 edge-classes and 0, 1 and 2 vertices truncated: (Top) the polytope $\polZ(V)\ssm\c X$; (Bottom) the graph $\Omega_{V, \c X}$.
		Edges drawn in the same color are dependent.
		The number of connected components of $\Omega_{V, \c X}$ is the dimension of the deformation cone $\DefoCone[{\polZ(V) \ssm \c X}]$, hence the rightmost example is indecomposable.}
	\label{fig:DeepTruncations}
\end{figure}

The proof of the following lemma is analogous to that of \Cref{lem:OtherFacesOfPandQ}.
\begin{lemma}
	If $\pol$ admits a deep truncation at $\b x$, then the proper faces of $\pol\ssm \b x$ are either
	\begin{compactenum}[(i)]\item the faces of $\pol[X]$, or \item of the form $\pol[F]\ssm \b x$ for $\pol[F]$ a face of $\pol$ (with the convention $\pol[F]\ssm \b x = \pol[F]$ if $\b x\notin\pol[F]$).\end{compactenum}
\end{lemma}
%

In particular:
\begin{compactitem}
	\item the vertices of $\pol\ssm\b x$ are the vertices of $\pol$ except $\b x$;
	\item the edges $\pol\ssm \b x$ are the edges of $\pol$ that do not contain $\b x$, and the edges $[\b u, \b u']$ where $\b u, \b u'$ are vertices of $\pol$ adjacent to $\b x$ for which there is a 2-face of $\pol$ containing $\b u, \b u', \b x$;
	\item the 2-faces of $\pol\ssm\b x$ are the 2-faces of $\pol$ not containing $\b x$, the 2-faces of $\pol[X]$, and the 2-faces obtained by removing $\b x$ from a (non-triangular) 2-face of $\pol$ which contains it.
\end{compactitem}

Concerning the cycle equations of $\pol\ssm\b x$, note that some cycles of $\pol$ disappear when we truncate a vertex, some appear, and some are modified.
In particular, note that if $\pol$ has a 2-face $\pol[F]$ with $\b x\in\pol[F]$ which is a quadrangle with vertices $\b x$, $\b u$, $\b w$, $\b u'$ in cyclic order, then $\pol[F]\ssm \b x$ is a triangle, so the three edges $[\b u, \b u']$, $[\b u, \b w]$ and $[\b u', \b w]$ are dependent in $\pol\ssm \b x$.
We will apply that to the case of parallelogramic zonotopes, but we first need to recall some property on the edge-classes of such zonotopes.

\begin{definition}
	For a collection of vectors $V$ and $\b v\in V$, we denote by \defn{$V^{\downarrow\b v}$} be the vector configuration obtained by projecting the vectors $\b u\in V$ with $\b u\ne \b v$ orthogonally along the direction spanned by $\b v$.
	\end{definition}

We denote by \defn{$ED_{\b v}$} the graph whose node set is $C_{\b v}$ and where two edges are linked if they are opposite in a parallelogram of $\polZ(V)$. Recall that a graph $G$ is $k$-connected if the graph $G$ cannot be disconnected by the deletion of any set of $k-1$ vertices or less.

\begin{theorem}\label{thm:ConnectivityOfEdgeClasses}
For a parallelogramic zonotope $\polZ(V)$ and any vector $\b v\in V$, the graph $ED_{\b v}$ is a \linebreak$(\dim\polZ(V) - 1)$-connected graph.
\end{theorem}

\begin{proof}
	By \cite[Cor. 7.17]{Ziegler-polytopes}, the face lattice of the zone associated to $\b v$ is isomorphic to the face lattice of $\polZ(V^{\downarrow\b u})$.
	In particular, $ED_{\b v}$ is the 1-incidence graph of the zone, which is isomorphic to the 1-skeleton of $\polZ(V^{\downarrow \b u})$.
	By Balinski's theorem (see \cite[Section 3.5]{Ziegler-polytopes}), the latter is $(\dim\polZ(V^{\downarrow \b u}))$-connected.
\end{proof}

\begin{lemma}\label{lem:deepmultitruncation}
	If $\c X$ is a collection of vertices of $\pol$ such that no two vertices are adjacent and every $\b x\in \c X$ admits a deep truncation, then $(\dots((\pol\ssm\b x_1)\ssm \b x_2)\ssm\dots\ssm \b x_{r-1})\ssm\b x_r$ does not depend on the ordering $\c X = \{\b x_1, \dots, \b x_r\}$, and we denote \defn{$\pol\ssm\c X$} this deep multi-truncation.
\end{lemma}

\begin{proof}
	Fix $\pol$ with two vertices $\b x$, $\b y$ that admit a deep truncation.
	If $\b x$ and $\b y$ are not adjacent, then the neighbors of $\b y$ in $\pol\ssm \b x$ are the neighbors of $\b y$ in $\pol$.
	Hence, $\b y$ admits a deep truncation in $\pol\ssm\b x$, and $(\pol\ssm\b x)\ssm\b y = (\pol\ssm\b x)\cap H_{\b y}^- = \pol\cap H_{\b x}^- \cap H_{\b y}^-$.
	The same holds by exchanging the roles of $\b x$ and $\b y$, yielding $(\pol\ssm\b x)\ssm\b y = (\pol\ssm\b y)\ssm\b x$.
	This gives the claim.
\end{proof}

\begin{definition}
For a zonotope $\polZ(V)$ and a vertex $\b x$, an edge-class $C_{\b v}$ is \defn{incident around $\b x$} if there exists a vertex $\b w$ of $\polZ$ adjacent to $\b x$ such that $\b x\b w\in C_{\b v}$.

For a subset $\c X$ of vertices of $\polZ(V)$, we call \defn{$\Omega_{V, \c X}$} the graph whose nodes are the edge-classes of $\polZ(V)$, and where two edge-classes are linked by an arc if there exists $\b x\in \c X$ such that both edge-classes are incident around $\b x$, see \Cref{fig:DeepTruncations} and \Cref{fig:SmallDimPandQ}.

For an edge-class $C_{\b v}$, we denote \defn{$\c X_{\b v}$} the set of vertices $\b x\in \c X$ which are adjacent to some edge $\pol[e]\in C_{\b v}$.
\end{definition}

We will use \Cref{thm:DCdimBoundedByNumberOfDependentSets} to control the dimension of $\DefoCone[\polZ(V) \ssm \c X]$, under suitable conditions on $\c X$.


\begin{theorem}\label{thm:ZonotopeWithDistinguishedSetOfVertices}
	Let $\polZ(V)$ be a parallelogramic zonotope, and $\c X$ a subset of its vertices, such that each of them admits a deep truncation.
	If $\c X$ is a stable in $\c G(\polZ)$ (\ie no two vertices of $\c X$ are linked by an edge of $\polZ(V)$), and $|\c X_{\b v}| \leq \dim\polZ(V) - 2$ for all $\b v\in V$, then $\dim\DefoCone[{\polZ(V)\ssm\c X}]$ is smaller than the number of connected components of $\Omega_{V, \c X}$.
	
	In particular, if $\Omega_{V, \c X}$ is connected, then $\polZ(V)\ssm \c X$ is an indecomposable polytope.
\end{theorem}

\begin{proof}
	As $\c X$ is a stable in $\c G(\polZ(V))$ and each $\b v\in \c X$ admits a deep truncation, the deep multi-truncation $\polZ(V)\ssm\c X$ is well-defined by \Cref{lem:deepmultitruncation}.
	
	For every edge-class $C_{\b v}$, if $\pol[e], \pol[f]\in C_{\b v}$ are opposite in a parallelogram and $\pol[e], \pol[f]$ do not contain any $\b x\in \c X$, then $\pol[e], \pol[f]$ are still edges in $\polZ(V)\ssm\c X$ and they are dependent because this parallelogram is still a 2-face of $\polZ(V)\ssm\c X$.
	Hence, the graph obtained from $ED_{\b v}$ by removing $\c X_{\b v}$ is a subgraph of $\ELD[{\polZ(V)\ssm\c X}]$.
	By \Cref{thm:ConnectivityOfEdgeClasses}, this deleted graph is connected: $ED_{\b v}$ is $(\dim\polZ(V)-1)$-connected and $|\c X_{\b v}| < \dim\polZ(V) - 1$.
	
	Besides, as all 2-faces of $\polZ(V)$ are parallelograms, the deep multi-truncation turns certain 2-faces into triangles.
	Let $\b w$, $\b w'$ be two vertices adjacent to some $\b x\in \c X$.
	If $\b x$, $\b w$ and $\b w'$ lie in a common 2-face (\ie if $\b w\b w'$ is an edge of the facet $\pol[X]$ of $\polZ(V)\ssm\c X$), let $\b y$ be the last vertex of this parallelogram: the edges $\b y\b w$ and $\b y\b w'$ are dependent in $\polZ(V)\ssm \c X$ because they form a triangle together with the edge $\b w\b w'$.
	Hence, if two edge-classes $C_{\b v}$ and $C_{\b v'}$ satisfy that there exist $\b w$ and $\b w'$ adjacent to $\b x$ with $\b x\b w\in C_{\b v}$, $\b x\b w'\in C_{\b v'}$, and $\b w\b w'$ is an edge of $\pol[X]$, then $(C_{\b v} \ssm \c X_{\b v})\cup (C_{\b v'} \ssm \c X_{\b v'})$ is dependent.
	Furthermore, as $\c G(\pol[X])$ is connected, so is its line graph, consequently, we can remove the condition ``and $\b w\b w'$ is an edge of $\pol[X]$'' from the previous sentence: if $C_{\b v}$ and $C_{\b v'}$ are incident around some $\b x\in \c X$, then $(C_{\b v} \ssm \c X_{\b v})\cup (C_{\b v'} \ssm \c X_{\b v'})$ is dependent.
	
	Finally, any 2 vertices of $\polZ(V)\ssm\c X$ are connected in $\c G(\polZ(V)\ssm\c X)$ by a path of edges in the union of all $C_{\b v}\ssm\c X_{\b v}$.
	By \Cref{thm:DCdimBoundedByNumberOfDependentSets}, $\dim\DefoCone[{\polZ\ssm\c X}]$ is less than the number of connected components of $\Omega_{V, \c X}$.
\end{proof}

\begin{corollary}\label{cor:ZonotopeWithDistinguishedVertex}
	Let $\polZ(V)$ be a parallelogramic zonotope.
	If there exists a vertex $\b x_\circ$ of $\polZ$ such that all its neighbors are on a common hyperplane, and all the edges-classes of $\polZ(V)$ are incident\footnote{Equivalently, if all its neighbors are on a common hyperplane, and $\polZ(V)$ is a translation of $\sum_{\b x_\circ\b w\in E(\polZ(V))} \,[\b x_\circ, \b w]$.} around $\b x_\circ$, then $\polZ(V)\ssm \b x_\circ$ is indecomposable.
	
	Moreover, $\polZ(V)\ssm\c X$ is indecomposable for any subset $\c X$ of vertices of $\polZ(V)$ with $\b x_\circ\in \c X$, and $|\c X|\leq \dim\polZ(V) - 2$ which is a stable set in $\c G(\polZ(V))$ of vertices admitting deep truncations.
\end{corollary}

\begin{proof}
	As all the neighbors of $\b x_\circ$ lie on a common hyperplane, $\b x_\circ$ admits a deep truncation (and $\{\b x_\circ\}$ is stable in $\c G(\polZ(V))$, with $|\{\b x_\circ\}| \leq \dim\polZ(V) - 2$).
	As all the edge-classes are incident around $\b x_\circ$, the graph $\Omega_{V, \{\b x_\circ\}}$ is connected.
	Hence, $\polZ\ssm\b x_\circ$ is indecomposable according to \Cref{thm:ZonotopeWithDistinguishedSetOfVertices}.
	Moreover, if $\b x_\circ\in \c X$, and $\c X$ is a stable in $\c G(\polZ(V))$ with $|\c X|\leq \dim\polZ(V) - 2$, then \Cref{thm:ZonotopeWithDistinguishedSetOfVertices} still applies.
\end{proof}

This corollary directly implies \Cref{thm:PandQareIndecomposable}.
Moreover, note that $\Qnm[2][2]$ is decomposable: this illustrates that the above \Cref{thm:ZonotopeWithDistinguishedSetOfVertices,cor:ZonotopeWithDistinguishedVertex} are tight, as we truncate 2 vertices of $\ZGKnm[2][2]$ to obtain $\Qnm[2][2]$, which is strictly more than $\dim\ZGKnm[2][2] - 2 = 1$.

We close this section with an example showing that it is easy to construct a zonotope satisfying the hypotheses of the previous \Cref{cor:ZonotopeWithDistinguishedVertex}.

\begin{example}
For a $d$-polytope $\pol\subset\R^d$, let $V_{\pol} = \left\{\binom{1}{\b v} ~;~ \b v\in V(\pol)\right\}$. 
The zonotope $\polZ(V_{\pol})$ satisfies the hypotheses of \Cref{cor:ZonotopeWithDistinguishedVertex}.
Indeed, let $\b x_\circ = \b 0\in\R^{d+1}$.
Then the neighbors of $\b 0$ are $\binom{1}{\b v}$ for $\b v\in V(\pol)$.
They all lie in a common hyperplane, $\{\b y\in\R^{d+1} ~;~ y_{d+1} = 1\}$, and, by construction, they generate $\polZ(V_{\pol})$. 
In addition, the 2-faces of $\polZ(V_{\pol})$ are parallelograms because no three vectors of $V_{\pol}$ lie on a common plane (this would contradict the convexity of $\pol$).

Consequently, by \Cref{cor:ZonotopeWithDistinguishedVertex}, the polytope $\polZ(V_{\pol})\ssm \b 0$ is indecomposable, and so is $\polZ(V_{\pol})\ssm\c X$ for any subset $\c X$ of vertices of $\polZ(V_{\pol})$ with $\b 0\in \c X$ and $|\c X|\leq d - 2$ such that $\c X$ is stable in $\c G(\polZ(V_{\pol}))$.
\end{example}

\subsection{Stacking vertices on parallelogramic sums of indecomposable polytopes}\label{ssec:StackingOnParallelogramicZonotopes}

We now apply similar ideas in order to the study of stacking over sums of indecomposable polytopes in parallelogramic position.
As in \Cref{ssec:DeepTruncationsOfParallelogramicZonotopes}, this will again allow us to construct large families of polytopes which are indecomposable even though they closely resemble $\pol_1 + \dots + \pol_r$, such as parallelogramic zonotopes (in particular, they will have ``few'' triangles, in proportion to their total number of 2-faces).
We want to use the fact that, for such a sum, the graph of edge-dependencies is easy to handle. 
We focus on stacking vertices of the rest of this section.

\begin{figure}
	\centering
	\includegraphics[width=0.24\linewidth]{Figures/ZonotopeC4.pdf}
	\includegraphics[width=0.24\linewidth]{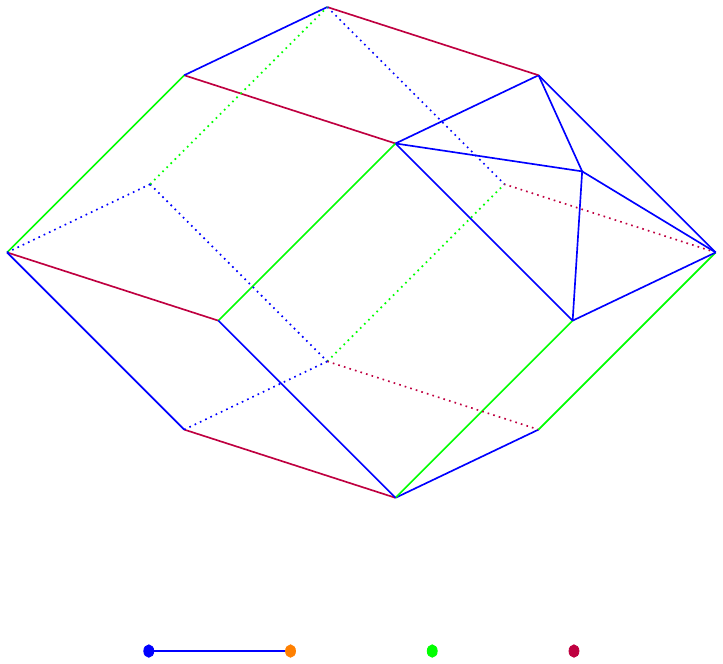}
	\includegraphics[width=0.24\linewidth]{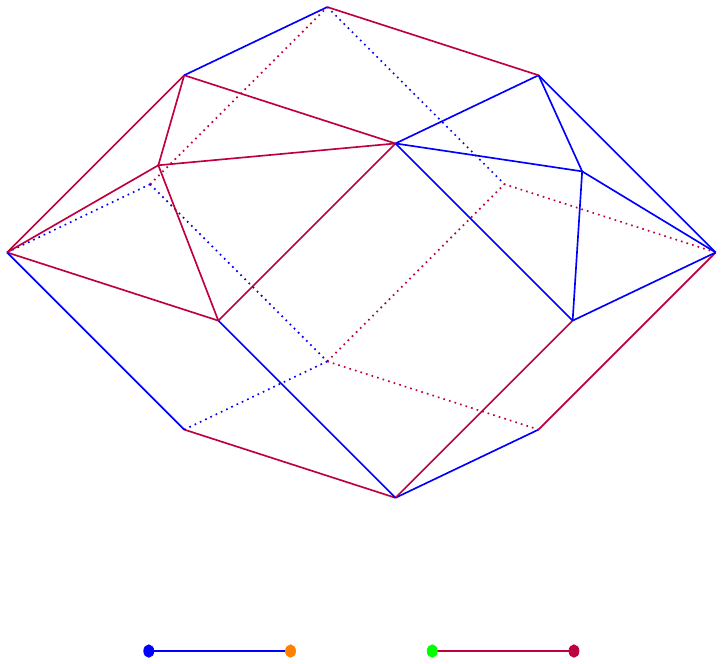}
	\includegraphics[width=0.24\linewidth]{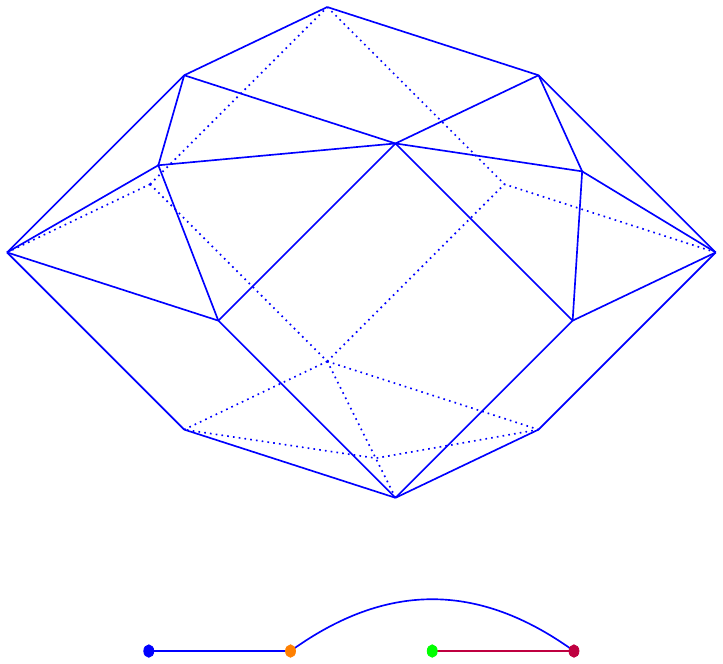}
	\caption[Stacking vertices on a zonotope to improve its indecomposability]{(Left to right) A zonotope with 4 edge-classes and 0, 1, 2 and 3 vertices stacked on its facets.
		(Top) the polytope $\polZ(V)_{\scr F}$;
        (bottom) its graph of stacking $\Gamma_{V}(\scr F)$.
		Edges in the same color are dependent.
		The exact locations of the stacked vertices do not matter: the (in)decomposability is only driven by the combinatorics of edge classes linked to these stacked vertices.
		The number of connected components of $\Gamma_{V}(\scr F)$ is the dimension of the deformation cone $\DefoCone[{\polZ(V)_{\scr F}}]$: the right-most example is indecomposable.}
	\label{fig:ZonotopeStacked}
\end{figure}

\begin{definition}
	Fix a $d$-polytope $\pol$ with $d\geq 3$.
	For a facet $\pol[G]$ of $\pol$, we denote \defn{$H_{\pol[G]}^-$} the closed half-space bounded by $\pol[G]$ and containing $\pol$, and by \defn{$H_{\pol[G]}^+$} the opposite closed half-space.
	
	
	A \defn{vertex stacking} on a facet $\pol[F]$ of a polytope $\pol$ is a polytope defined as $\pol_{\pol[F], \b q} = \conv(\pol \cup\{\b q\})$ where $\b q$ is a point chosen ``close to the facet $\pol[F]$ but outside of $\pol$'': \ie $\b q \in \interior\Bigl(H^+_{\pol[F]} \cap \bigcap_{\pol[G]\ne\pol[F]} H^-_{\pol[G]}\Bigr)$.
	
	The vertex stacking $\pol_{\pol[F], \b q}$ is combinatorially independent from the choice of $\b q$.
\end{definition}

As we will see now, the choice of $\b q$ does not affect the (in)decomposability of the vertex stacking $\pol_{\pol[F], \b q}$.

\begin{theorem}\label{thm:DecomposabilityVertexStacking}
	For a polytope $\pol$ with $\dim\pol\geq 3$,
	we have $\DefoCone[{\pol_{\polF, \b q}}] \simeq \DefoCone \cap \bigl\{\b\lambda_{\pol[e]} \simeq \b\lambda_{\pol[f]} ~;~ \pol[e], \pol[f] \text{ edges of }\polF\bigr\}$.
	
	Moreover, for $\pol_{\polF, \b q}$, the collection of edges $E(\pol[F])\cup \bigl([\b x, \b q] ~;~ \b x\in V(\pol[F])\bigr)$ is dependent.
\end{theorem}

\begin{proof}
	The edges of $\pol_{\pol[F], \b q}$ are the edges of $\pol$ and the segments $[\b x, \b q]$ for $\b x$ a vertex of $\pol[F]$.
	The cycles of $\pol_{\pol[F], \b q}$ are the cycles of $\pol$, and the triangles $\conv(\pol[e]\cup\{\b q\})$ for $\pol[e]$ and edge of $\pol[F]$.
	Hence the cycle equations of $\pol_{\pol[F], \b q}$ are the one of $\pol$, along with $\b\lambda_{\b u\b q} = \b\lambda_{\b x\b q} = \b\lambda_{\b u\b x}$ for the edges $[\b u, \b x]$ of $\pol[F]$.
	
	Note that the following graph is connected: its nodes are the edges $\pol[e]$ of $\pol[F]$ and $[\b x, \b q]$ for $\b x$ a vertex of $\pol[F]$, and two nodes are linked by an arc if they belong to a triangular face $\conv(\pol[e]\cup\{\b q\})$ for $\pol[e]$ and edge of $\pol[F]$ (it is the 1-incidence graph of the pyramid over $\pol[F]$, without the arcs coming from $\pol[F]$).
	Consequently, the collection of edges $E(\pol[F])\cup \bigl([\b x, \b q] ~;~ \b x\in V(\pol[F])\bigr)$ is dependent.
	
	Hence, we have $\b\lambda_{\pol[e]} = \b\lambda_{\pol[f]}$ for any $\pol[e], \pol[f]$ edges of $\polF$.
	By \Cref{lem:edge-deformations-frameworks}, this gives the claim on $\DefoCone[{\pol_{\polF, \b q}}]$.
\end{proof}

\begin{definition}
For $\pol_1, \dots, \pol_r$ indecomposable polytopes in parallelogramic position and $\scr F$ a collection of facets of $\pol_1 + \dots + \pol_r$, let \defn{$\Gamma_{\pol_1, \dots, \pol_r}(\scr F)$} be the \defn{graph of stacking} whose set of nodes is $[r]$, and where there is an arc between $i$ and $j$ if there is some facet $\polF\in \scr F$ containing an edge parallel to an edge of $\pol_i$ and an edge parallel to an edge of $\pol_j$, \ie there exist edges $\pol[e]\subseteq \pol_i$ and $\pol[f]\subseteq \pol_j$ such that $\polF \cap C_{\pol[e]} \ne\emptyset$ and $\polF \cap C_{\pol[f]} \ne\emptyset$.

We denote \defn{$(\pol_1 + \dots + \pol_r)_{\scr F}$} the (combinatorial) polytope obtained stacking vertices on the facet $\polF$ of the polytope $\pol_1 + \dots + \pol_r$ for each $\polF\in \scr F$.
\end{definition}

\begin{theorem}\label{thm:StackingParallelogramicPosition}
Let $\pol_1, \dots, \pol_r$ be indecomposable polytopes in parallelogramic position.
Let $\scr F$ be a subset of its facets.
The deformation cone $\DefoCone[{(\pol_1+\dots+\pol_r)_{\scr F}}]$ is a simplicial cone whose dimension is the number of connected components of $\Gamma_{\pol_1, \dots, \pol_r}(\scr F)$.
\end{theorem}

\begin{proof}
Fix stacking vertices $\b q_{\polF}$ for each $\polF\in \scr F$.
Note that the space of cycle equations of $\left(\sum_i \pol_i\right)_{\scr F}$ is generated by the cycle equations of $\sum_i \pol_i$ together with the ones coming from the triangles $\conv(\pol[e]\cup\b q_{\polF})$ for each $\polF\in \scr F$ and each edge $\pol[e]\in E(\polF)$.

Pick $\polF\in \scr F$.
Consider the edge-classes $C_1, \dots, C_k$ of $\sum_i \pol_i$ which are touched by $\pol[F]$, \ie $C_i\cap E(\pol[F])\ne\emptyset$:
by \Cref{thm:DecomposabilityVertexStacking},
the edges in $E(\pol[F])\cup \bigl([\b x, \b q_{\polF}] ~;~ \b x\in V(\pol[F])\bigr)$ are dependent.
Moreover, as $\sum_i \pol_i$ is a sum of polytopes in parallelogramic position, by \Cref{cor:ParallelogramicSumIndecomposablePolytopes}, for each $j$, the edges in $C_j$ are dependent.
Hence, the set of edges $C_1\cup\dots\cup C_k\cup \bigl([\b x, \b q_{\polF}] ~;~ \b x\in V(\pol[F])\bigr)$ is dependent.
Furthermore, as each $\pol_i$ is indecomposable, the edges classes $C_{\pol[e]}$ and $C_{\pol[f]}$ are dependent for any two edges $\pol[e], \pol[f]$ of some $\pol_i$.

Let $D$ be a connected component of $\Gamma_{\pol_1, \dots, \pol_r}(\scr F)$, and let:
$$\c D := \bigcup_{j\in D} \bigcup_{\pol[e]\in E(\pol_j)} C_{\pol[e]} ~\cup \bigcup_{\substack{\pol[F]\in\scr F \\ \exists j\in D,\, \exists\, \pol[e]\in E(\pol_j),~  E(\pol[F])\cap C_{\pol[e]}\ne\emptyset}} \bigl([\b x, \b q_{\polF}] ~;~ \b x\in V(\pol[F])\bigr)$$

The vector $\b\ell^{\c D}$ satisfies all the cycle equations of $\left(\sum_i \pol_i\right)_{\scr F}$:
indeed, these equations are generated by the ones of $\left(\sum_i \pol_i\right)$ and the ones coming from the added triangles, hence $\b\ell^{\c D}$ is the edge-deformation vector of $(\sum_{j\in D} \pol_j)_{\scr G}$ where $\scr G = \Bigl\{\polF \in \scr F ~;~ \exists j\in D,\, \exists\, \pol[e]\in E(\pol_j),~  E(\pol[F])\cap C_{\pol[e]}\ne\emptyset\Bigr\}$.
Hence, for each connected component $D$ of $\Gamma_{\pol_1, \dots, \pol_r}(\scr F)$, the set of edges $\c D$ is dependent and contractible.
By \Cref{thm:NiceConnectedComponentImplySimplicialCone}, the cone $\DefoCone[{\left(\sum_i \pol_i\right)_{\scr F}}]$ is simplicial, its dimension is the number of connected components of $\Gamma_{\pol_1, \dots, \pol_r}(\scr F)$.
\end{proof}

\begin{corollary}\label{cor:StackingParallelogramicPositionIndecomposability}
Let $\pol_1, \dots, \pol_r$ be indecomposable polytopes in parallelogramic position.
Let $\scr F$ be a collection of its facets of $(\pol_1+\dots+\pol_r)_{\scr F}$.
The polytope $(\pol_1+\dots+\pol_r)_{\scr F}$ is a indecomposable if and only if $\Gamma_{\pol_1, \dots, \pol_r}(\scr F)$ is connected.
\end{corollary}

\begin{proof}
Immediate application of \Cref{thm:StackingParallelogramicPosition}.
\end{proof}

\begin{example}
A parallelogramic zonotope $\polZ(V)$ is a sum in parallelogramic position of segments $[\b 0, \b v]$ for $\b v\in V$.
As segments are indecomposable polytopes, we can apply \Cref{thm:StackingParallelogramicPosition,cor:StackingParallelogramicPositionIndecomposability}.

Recall that every facet of a zonotope $\polZ(V)$ is a translation of a zonotope $\polZ(W)$ for some $W\subsetneq V$.
Consequently, as depicted on \Cref{fig:ZonotopeStacked}, the graph $\Gamma_{V}(\scr F)$ can be easily understood as the graph on $V$ where we link by an arc two vectors $\b v$ and $\b v'$ if they there is some facet $\polF\in \scr F$ which is a translation of some $\polZ(W)$ satisfying $\b v\in W$ and $\b v'\in W$.
This allows to explicitly construct families of facets $\scr F$, or equivalently families of $W\subsetneq V$, that make the graph $\Gamma_V(\scr F)$ connected.
By \Cref{cor:StackingParallelogramicPositionIndecomposability}, such a family yields an indecomposable stacked parallelogramic zonotope $\polZ(V)_{\scr F}$, as in \Cref{fig:ZonotopeStacked} (right).
\end{example}

\begin{example}
Remember that a product of polytopes is an example of a sum in parallelogramic position.
Let $\pol \coloneqq \pol_1\times\dots\times\pol_r$ be a product of indecomposable polytopes (\eg a product of simplices).

Suppose that $\dim\pol_j\geq 2$ for some $j\in [r]$.
Let $\polF \coloneqq \pol_1\times\dots\times\pol_{j-1}\times\polF_j\times\pol_{j+1}\times\dots\times\pol_r$ be a facet of $\pol$, where $\polF_j$ a facet of $\pol_j$.
Then, \Cref{thm:StackingParallelogramicPosition} implies that any stacking $\pol_{\polF, \b q}$ yields an indecomposable polytope.
For instance, in \Cref{sfig:Prism}, stacking on a quadrilateral facet yields an indecomposable polytope.

Conversely, suppose that $\dim\pol_j = 1$ for all $j\in [r]$ (with $r \geq 3)$, then $\pol$ is a product of segments, \ie a $r$-cube, see \Cref{sfig:Cube} for $r = 3$.
If we stack a vertex on the facet $\pol_1\times\dots\times\polF_j\times\dots\times\pol_r$, the resulting polytope admits a segment $\pol_j$ as a deformation, hence it is decomposable.
On the other hand, pick any $j, k\in [r]$ with $j\ne k$.
By \Cref{thm:StackingParallelogramicPosition}, stacking on the facets $\pol_1\times\dots\times\polF_j\times\dots\times\pol_r$ and $\pol_1\times\dots\times\polF_k\times\dots\times\pol_r$ where $\polF_j$ (respectively $\polF_k$) is a facet of $\pol_j$ (respectively of $\pol_k$), \ie a point, yields an indecomposable polytope. 
\end{example}

\begin{question}
Even if we did not develop it here, it is clear that deep truncations of sums of (indecomposable) polytopes in parallelogramic position are worth studying.
There are two problems in generalizing
\Cref{ssec:DeepTruncationsOfParallelogramicZonotopes}:
firstly, the criterion $|\c X_v|\leq \dim \polZ(V) - 2$ of \Cref{thm:ZonotopeWithDistinguishedSetOfVertices} does not seem to have a straightforward extension; and secondly deep truncating a vertex can affect some summand itself, and deep truncating an indecomposable polytope may result into a decomposable one.
The first problem concerns the dependency of the distinct copies of an edge of one of the summands, whereas the second problem concerns the dependency of the edges arising from each of the summands.


For example, a product of standard simplices $\simplex[d_1]\times\dots\times\simplex[d_r]$ with $\simplex[d_i] = \conv(\b e_j ~;~ j\in [d_i])$, is a sum in parallelogramic position.
If $d_i\geq 3$ for all $i\in[r]$, then every vertex admits a deep truncation, and one can show that deeply truncating any vertex of $\simplex[d_1]\times\dots\times\simplex[d_r]$ yields an indecomposable polytope. In fact, Federico Castillo and Spencer Backman have proven in \cite{BackmanCastillo} that such a statement stays valid when certain other faces, not only vertices, are truncated.
\end{question}

\section*{Acknowledgments}
This project started at the \emph{Santander Workshop on Geometric and Algebraic Combinatorics} in January~2024.
We want to thank Vincent Pilaud for advising GP, in December 2023, to write a short note on Strawberry \& Persimmon (this note was never written but instead transfigured into the present article and in the article written with Georg Loho~\cite{LohoPadrolPoullot2025RaysSubmodularCone}).
We also thank Georg Loho for many interesting conversations on the topic, as well as many pointers to the bibliography; 
and to Winfried Bruns for his help on the computation of the cone of deformed permutahedra of dimension 4 (and partially 5 \& 6).
Besides, thanks go to Frédéric Meunier for clarifying the algorithmic complexity of checking if a cone is 1-dimensional or not.
Finally, we want to thank Spencer Backman and Federico Castillo for sharing results from their work in progress~\cite{BackmanCastillo} with us, and in particular to Federico Castillo for encouraging us to add \Cref{ex:hyperorder,thm:ConnectedMatroidPolytopesAreIndecomposable}, and sharing his thoughts on how one might prove them.
An extended abstract of a previous version of this work was presented at the conference FPSAC2025~\cite{FPSAC2025Version}.


\bibliographystyle{alpha}
\bibliography{Biblio.bib}
\label{sec:biblio}

\end{document}